\documentclass[11pt, a4paper]{article}

\usepackage[utf8]{inputenc}
\usepackage[english]{babel}
\usepackage[T1]{fontenc}
\usepackage[colorlinks = true,
            linkcolor = blue,
            urlcolor  = blue,
            citecolor = blue,
            anchorcolor = blue]{hyperref}
\usepackage[a4paper,width=150mm,top=25mm,bottom=25mm]{geometry}
\usepackage{fancyhdr}

\usepackage{amsmath,amsthm,amssymb, enumitem, mathtools, mathrsfs}
\usepackage{siunitx}
\setitemize{noitemsep}

\usepackage{tikz}
\usepackage{tikz-cd}
\usepackage{caption}
\usepackage[all]{xy}
\usepackage{graphicx}
\usepackage{chngcntr}
\usepackage{booktabs}
\usepackage[labelformat=simple]{subcaption}

\usepackage{pdfpages}

\usepackage{csquotes}
\usepackage[
backend=biber,
style=alphabetic,
sorting=anyt
]{biblatex}
\graphicspath{ {images/} }

\theoremstyle{plain}
\newtheorem{theorem}{Theorem}[section]
\newtheorem*{theorem*}{Theorem}
\newtheorem*{satz*}{Satz}
\newtheorem{proposition}[theorem]{Proposition}
\newtheorem{lemma}[theorem]{Lemma}
\newtheorem{corollary}[theorem]{Corollary}
\theoremstyle{definition}
\newtheorem{definition}[theorem]{Definition}
\newtheorem{assumption}[theorem]{Assumption}

\theoremstyle{remark}
\newtheorem{remark}[theorem]{Remark}

\numberwithin{equation}{section}

\DeclareMathOperator{\supp}{supp}

\DeclareMathOperator{\trace}{tr}

\DeclareMathOperator{\leb}{Leb}

\newcommand{\ev}{\mathbb{E}}
\newcommand{\pr}{\mathbb{P}}

\newcommand{\R}{\mathbb{R}}
\renewcommand{\P}{\mathcal{P}}
\newcommand{\F}{\mathcal{F}}

\renewcommand{\L}{\mathcal{L}}
\newcommand{\M}{\mathcal{M}}
\renewcommand{\d}{\mathrm{d}}

\newcommand{\define}{\mathpunct{:}}
\newcommand{\bb}[1]{\mathbb{#1}}
\renewcommand{\bf}[1]{\mathbf{#1}}
\renewcommand{\cal}[1]{\mathcal{#1}}

\title{Title}
\author{Philipp Jettkant}

\begin{document}



\noindent
\begin{center}
    \Large
    \textbf{Semilinear BSPDEs and Applications to McKean-Vlasov Control with Killing}

    \vspace{1em}

    \normalsize
    Ben Hambly and Philipp Jettkant
    
    \textit{Mathematical Institute, University of Oxford}

\end{center}

\vspace{1em}

\begin{abstract}
\thispagestyle{plain}
\small
We introduce a novel class of semilinear nonlocal backward stochastic partial differential equations (BSPDE) on half-spaces driven by an infinite-dimensional c\`adl\`ag martingale. The equations exhibit a degeneracy and have no explicit condition at the boundary of the half-space. To treat the existence and uniqueness of such BSPDEs we establish a generalisation of It\^o's formula for infinite-dimensional c\`adl\`ag semimartingales, which addresses the occurrence of boundary terms.
Next, we employ these BSPDEs to study the McKean--Vlasov control problem with killing and common noise proposed in \cite{hambly_mvcp_arxiv_2023}. The particles in this control model live on the real line and are killed at a positive intensity whenever they are in the negative half-line. Accordingly, the interaction between particles occurs through the subprobability distribution of the living particles. We establish the existence of an optimal semiclosed-loop control that only depends on the particles' location and not their cumulative intensity. This problem cannot be addressed through classical mimicking arguments, because the particles' subprobability distribution cannot be reconstructed from their location alone. Instead, we represent optimal controls in terms of the solutions to BSPDEs of the above type and show those solutions do not depend on the intensity variable.
\end{abstract}

\normalsize

\section{Introduction} \label{sec:introduction}

The aim of this article is two-fold: first, we introduce a novel class of semilinear backward stochastic partial differential equations (BSPDE) driven by an infinite-dimensional c\`adl\`ag martingale. Then, we employ these BSPDEs to establish the equivalence between the open- and closed-loop formulation of the McKean--Vlasov control problem with killing at a state-dependent intensity recently proposed in \cite{hambly_mvcp_arxiv_2023}. In this control model, the location of the representative particle is described by a diffusion process $X = (X_t)_{0 \leq t \leq T}$ on the real line with McKean--Vlasov dynamics that feature a common noise $W$. Whenever the particle is located in the negative half-line, its cumulative intensity process $\Lambda = (\Lambda_t)_{0 \leq t \leq T}$ increases at a positive rate $\lambda(t, X_t)$. Once the accumulated intensity $\Lambda_t$ exceeds a standard exponential random variable $\theta$, the particle is killed. Accordingly, the particles do not interact through the conditional law $\L(X_t \vert W)$ of all particles but rather the conditional subprobability distribution $\pr(X_t \in \cdot, \, \theta > \Lambda_t \vert W)$ of the living particles. 

In \cite{hambly_mvcp_arxiv_2023}, we establish the existence of an optimal weak control $\gamma = (\gamma_t)_{0 \leq t \leq T}$ of the form $\gamma_t = g_t(X_t, \Lambda_t)$ for an $\bb{F}$-progressively measurable random function $g \define [0, T] \times \Omega \times \R \times [0, \infty) \to G$. Here $\bb{F} = (\F_t)_{0 \leq t \leq T}$ extends the filtration generated by the common noise and $G \subset \R^{d_G}$ is the control space. The need for an enlarged filtration comes from the tightness arguments used to establish the existence of the optimal weak control $\gamma$. In this weak formulation of the control problem, the common noise in the conditioning of the subprobability of the living particles is replaced by the larger $\sigma$-algebra $\F_T$. That is, the particles interact through the term $\pr(X_t \in \cdot, \, \theta > \Lambda_t \vert \F_T)$. The function $g$ is called a semiclosed-loop control, because while it can only depend on the idiosyncratic information through the state $(X_t, \Lambda_t)$, it has access to the entire history of the common noise $W$. For simplicity, we will nonetheless refer to such types of control as closed-loop in this work, as we will not encounter proper closed-loop controls.

Because the exponential random variable $\theta$ has the memoryless property, we would not expect that it is necessary to keep track of the cumulative intensity $\Lambda_t$ to optimally control the system. Only the increments of $\Lambda_t$, which are completely determined by the particle's location $X_t$, are relevant. To make this intuition precise, it is advantageous to reformulate the mean-field control (MFC) problem as the control of the nonlinear nonlocal stochastic Fokker--Planck equation satisfied by the conditional joint law $\mu_t = \L(X_t, \Lambda_t \vert \F_T)$. We derive a necessary stochastic maximum principle (SMP) for this infinite-dimensional control problem, which allows us to express the optimal closed-loop control $g$ from above in terms of a suitable adjoint process $\tilde{u} \define [0, T] \times \Omega \times \R \times [0, \infty) \to \R$, which depends on the particle's location $x$ and cumulative intensity $y$. The adjoint process solves a linear BSPDE with nonlocality in $\partial_x \tilde{u}_t$, arising from the nonlocalities of the stochastic Fokker--Planck equation, and a degeneracy in the intensity variable $y$. Moreover, because $g$ and, therefore, the conditional joint law $\mu_t$ are not necessarily adapted to the filtration of the common noise but the potentially larger filtration $\bb{F}$, to guarantee adaptedness of $\tilde{u}_t$ an additional martingale driver may appear next to $W$ in the BSPDE. This martingale is an infinite-dimensional c\`adl\`ag process adapted to $\bb{F}$. Now, by inserting the expression for the optimal closed-loop control $g$ in terms of $\partial_x \tilde{u}_t(x, y)$ into the BSPDE satisfied by the adjoint process $\tilde{u}$, we deduce a semilinear BSPDE for $\tilde{u}$. The last step in our analysis is to prove that the unique solution to this semilinear BSPDE takes the form $\tilde{u}_t(x, y) = e^{-y}u_t(x)$ for a random function $u \define [0, T] \times \Omega \times \R \to \R$. In view of the SMP, this implies that $g$ does not depend on the particle's intensity $y$. 

The coupled forward-backward system of stochastic Fokker--Planck equation and semilinear BSPDE can be naturally interpreted as a mean-field game (MFG) system \cite{lasry_mfg_fh_2006}. From this point of view the BSPDE can be understood as a stochastic Hamilton--Jacobi--Bellman (HJB) equation. Stochastic HJB equations were first studied by \cite{peng_stochastic_hjb_1992} and our analysis follows a similar strategy: we view the stochastic HJB equation as a backward stochastic evolution equation on the Hilbert space $L^2(\R \times [0, \infty))$. This requires rather restrictive growth assumptions on the coefficients and cost functions, but we are able to lift them through a limit procedure. The main novelty in the analysis of this evolution equation are the infinite-dimensional c\`adl\`ag martingale driver and the degeneracy in the intensity variable which appears together with a seemingly underspecified boundary at $y = 0$. To deal with these challenges, we establish a generalisation of It\^o's formula for $L^2(\R^d)$- and $L^2(\R^d \times [0, \infty))$-valued processes, which is of independent interest.

\subsection{Related Work}

Linear BSPDEs were initially studied in the context of the optimal control of stochastic partial differential equations, particularly those arising in the filtering of partially observed diffusions \cite{benoussan_mp_dp_1983, nagase_oc_spde_1990, pardoux_bspde_1980, zhou_nec_spde_1993, zhou_duality_1992}. The solutions to these BSPDEs serve as adjoint processes that appear in the stochastic maximum principles formulated for these control problems. Specifically, the necessary SMP appearing in \cite{benoussan_mp_dp_1983} closely resembles the one established in this work, although unlike \cite{benoussan_mp_dp_1983} we have to deal with the nonlocalities of the stochastic Fokker--Planck equation as well the c\`adl\`ag martingale driver. For this, we again crucially draw on our generalisation of It\^o's formula. Structurally closest to ours are the nonlinear BSPDEs analysed in \cite{hu_semilinear_bspde_1991, peng_stochastic_hjb_1992}, while BSPDEs with nondegeneracies were studied in \cite{ma_bspde_1997, jin_1999_bspde, hu_semilinear_bspde_2002}. In this work, we make frequent use of the BSPDE estimates derived by Ma and Yong \cite{jin_1999_bspde}. To our knowledge the only article that considers a similar class of BSPDEs in a non-Brownian filtration is \cite{al_hussein_bspde_2006, al_hussein_bspde_2009}. However, their work explicitly assumes that the filtration only generates continuous martingales in order to avoid jumps in the dynamics.

MFC and MFGs have seen a surge of interest in recent years and have been studied both from the probabilistic and analytic perspective. Particularly for MFGs, the analytic approach, which leads to a coupled forward-backward PDE consisting of the Fokker--Planck equation and the Hamilton--Jacobi--Bellman equation, has received a lot of attention \cite{lasry_mfg_fh_2006, lasry_mfg_2007, cardaliaguet_mfg_2014, gomes_mfg_2016}. The corresponding HJB in the MFC setting is nonlocal and, thus, more challenging to analyse, so the literature is comparatively sparse \cite{bensoussan_mfg_mfc_2013, achdou_mfc_pde_2015}. An additional challenge we face in this work is the stochasticity in the PDE system generated by the common noise. One trick that has been exploited in \cite{cardaliaguet_master_2019} and, more recently, in \cite{cardaliaguet_mfg_common_2022} is to tackle stochastic MFG PDE systems is to apply a time-dependent random shift to the particles' conditional law to remove the stochastic integral appearing in the stochastic Fokker--Planck equation, leading to a PDE with random coefficients. The correspondingly transformed stochastic HJB equation exhibits a simpler structure and can be analysed more easily. Exploring the applicability of this technique in our setting is an interesting avenue for future research.

Equivalence of open- and closed loop formulations of (stochastic) control problems is a classical topic in the control literature. Usually, equivalence is established either through a probabilistic \cite{davis_oc_1973, fleming_soc_1975, bismut_ctrl_1976, el_karoui_control_1979} or analytic \cite{krylov_control_1980, lions_dpp_1983, lions_viscosity_1983} study of the value function or Markovian selection and mimicking theorems \cite{haussmann_mark_control_1986, el_karoui_optimal_ctrl_1987, haussmann_oc_1990, lacker_mfg_controlled_mgale_2015}. The analysis of the value function in MFC and the associated HJB equation on the space of probability measures is an ongoing and challenging area of research, particularly in the common noise setting, so the most general equivalence results for MFC draw on mimicking theorems, see \cite{lacker_mimicking_2020}. However, a straightforward application of the results in \cite{lacker_mimicking_2020} only establishes the existence of optimal closed-loop controls that depend both on the particle's location and cumulative intensity. To approach the more subtle notion of equivalence we consider in this work, new ideas are required.

The equivalence problem was previously studied in the related context of MFGs with absorption by Campi and coauthors \cite{campi_mfg_absorption_2018, campi_mfg_hitting_2021}. In an MFG with absorption, a particle is removed as soon as it leaves a prescribed domain $D$. The McKean--Vlasov dynamics with killing at state-dependent intensity described above can be interpreted as a regularisation of such a system: instead of killing particles immediately when they leave the domain, they accumulate intensity whenever they are located outside of $D$. The framework with immediate killing is recovered in the limit as the intensity tends to infinity (cf.\@ \cite[Theorem 2.14]{hambly_mvcp_arxiv_2023}). Campi and Fischer \cite{campi_mfg_absorption_2018} consider a conditional MFG framework with interaction through the conditional law of the living particles instead of their subprobability distribution. Here the conditioning is not with respect to a common noise, which is absent in \cite{campi_mfg_absorption_2018}, but rather the event that a particle has not been killed yet. They establish the existence of a Nash equilibrium using the weak formulation of MFGs analysed in \cite{carmona_weak_form_mfg_2015}. A verification argument shows that this equilibrium is in fact Markovian, which corresponds to an optimal closed-loop control in the MFC setting. The weak formulation and the related verification argument are not applicable to MFC, so this strategy does not apply in our context.

Concurrently with our work, Carmona, Lauri\`ere, and Lions \cite{carmona_nssc_2023} addressed the equivalence problem for the control of conditional processes, originally introduced by Lions in a lecture at the Coll\`ege de France \cite{lions_cond_proc_2016}. In this problem, the running and terminal cost are computed not from the subprobability distribution of the living particles but their conditional law. This is closely related to the framework in \cite{campi_mfg_absorption_2018} discussed above, though in the latter the conditional law only appears in the interaction between particles while the cost is computed based on the particles' subprobability distribution. \cite{carmona_nssc_2023} employs a strategy similar to the one followed in this article to establish the equivalence of the open- and closed-loop formulation for the control of conditional processes, albeit in a setting with constant coefficient and without common noise. Their article also considers regularised dynamics, where particles are killed according to a state-dependent intensity instead of simply being killed at the boundary of a given domain $D$. The original problem proposed by Lions was formulated for dynamics with immediate killing. Diffusions with immediate killing are less regular and consequently more challenging to deal with analytically, which means that the strategy developed in this article (and \cite{carmona_nssc_2023}) cannot be straightforwardly extended to cover this case. However, in contrast to the regularised framework, by absorbing, i.e.\@ stopping, the state process in the model with immediate killing
once it hits the boundary of the domain $D$, it is possible to read a particle's status (alive or killed) from its location. This observation enables one to apply the mimicking theorem, from which one can easily establish the equivalence of the open- and closed-loop formulation for the control of conditional processes in the case with immediate killing. See Appendix \ref{app:cond_proc} for details. Compare this with the regularised model, in which a particle's location reveals little about its cumulative intensity. 

\subsection{Main Contributions and Structure of the Paper}

In Section \ref{sec:main_results} we present the main results of the paper. In the course of this we introduce a novel class of BSPDEs with semilinear nonlocal coefficients, a degeneracy and underspecified boundary condition, as well as an infinite-dimensional c\`adl\`ag martingale driver. 

The analysis of this BSPDE is provided in Section \ref{sec:bspde}. We first establish a generalisation of It\^o's formula for c\`adl\`ag semimartingales with values in $L^2(\R^d)$ and $L^2(\R^d \times (0, \infty))$, see Theorems \ref{thm:ito} and Theorem \ref{thm:ito_2_dim}, and then prove existence and uniqueness for the BSPDE based on a Galerkin approximation and fixed-point arguments. 

In Section \ref{sec:equivalence} we discuss the equivalence of the open- and closed loop formulation for the McKean--Vlasov control problem with killing at a state-dependent intensity. First, we reformulate this control problem as the control of the associated stochastic Fokker--Planck equation, see Subsection \ref{sec:sfpe}. Next, in Subsection \ref{sec:smp}, we derive the necessary SMP by computing the Gateaux derivative of the value function. Subsequently, we study the adjoint process $\tilde{u}_t$ through the lens of the BSPDE theory developed in Section \ref{sec:bspde} and, thereby, demonstrate its separable structure $\tilde{u}_t(x, y) = e^{-y} u_t(x)$ (Subsection \ref{sec:bspde_1_2}). We conclude Section \ref{sec:equivalence} by extending the equivalence result to more general coefficients and cost functions. To achieve this, in Subsection \ref{sec:extending}, we reformulate stochastic Fokker--Planck equation satisfied by the conditional subprobability distribution of the living particles as a martingale problem and show the stability of the martingale problem under perturbations of the coefficients. As a byproduct we introduce a relaxed formulation for the control of stochastic Fokker--Planck equations, which is of independent interest.

\section{Main Results and Outline of the Equivalence Proof} \label{sec:main_results}

In this section we state the main results of this paper and provide an outline of the proof of equivalence for the open- and closed loop formulation of the McKean--Vlasov control problem discussed in the introduction. In what follows we will use $\nabla$ and $\nabla^2$ to denote the gradient and the Hessian, and write $a \cdot b = a^{\top} b$ for $a$, $b \in \R^d$ as well as $A : B = \trace(A^{\top}B)$ for $A$, $B \in \R^{d \times d}$, where $(\cdot)^{\top}$ is the transpose.

\subsection{Semilinear BSPDE with C\`adl\`ag Noise} \label{sec:main_results_bspde}

Let $(\Omega, \F, \pr)$ be a complete probability space equipped with a complete filtration $\bb{F} = (\F_t)_{0 \leq t \leq T}$ and a $d$-dimensional $\bb{F}$-Brownian motion $W = (W_t)_{0 \leq t \leq T}$. We study two types of semilinear backward SPDEs. First, for the domain $D = \R^d \times (\ell, \infty)$, with $\ell \in \R$, we consider the BSPDE
\begin{align} \label{eq:bspde}
\begin{split}
    \d u_t(x, y) &= -\Bigl(a_t(x, y) : \nabla_x^2 u_t(x, y) + H_t\bigl(x, y, u_t(x, y), \nabla_x u_t(x, y)\bigr) + F_t(x, y, u_t, \nabla_x u_t)\Bigr) \, \d t \\
    &\ \ \ - \lambda_t(x) \partial_y u_t(x, y) \, \d t - \bigl(\alpha_t : \nabla_x q_t(x, y) + \beta_t(x, y) \cdot q_t(x, y)\bigr) \, \d t \\
    &\ \ \ + q_t(x, y) \cdot \d W_t + \d m_t(x, y)
\end{split}
\end{align}
for $(x, y) \in D$ with terminal condition $u_T(x, y) = \psi(x, y)$, where $\psi \in L^2(\Omega, \F_T, \pr; L^2(D))$. Here $x \in \R^d$ denotes the first $d$ variables and $y \in (\ell, \infty)$ the remaining one. The process $m = (m_t)_{0 \leq t \leq T}$ is a c\`adl\`ag $\bb{F}$-martingale with values in the Hilbert space $L^2(D)$ and together with $q$ ensures that the solution $(u, q, m)$ is $\bb{F}$-progressively measurable. Note that the BSPDE is only defined on a half-plane, even though we have not introduced any boundary condition. Since we assume that $\lambda$ is nonnegative, this does not cause any issues. Indeed, since the SPDE proceeds backward in time, the nonnegativity of $\lambda$ implies that the value of $u_t(x, y)$ for $(t, \omega, x, y) \in [0, T] \times \Omega \times D$ only depends on the future values $u_s$, $q_s$, and $m_s$, $(s, \omega) \in [t, T] \times \Omega$, on $\R^d \times [y, \infty)$. Hence, we do not need to specify the behaviour of $u$ on the boundary $\R^d \times \{\ell\}$.

The second type of equation has a similar structure but is not defined on the half-plane $\R^d \times [\ell, \infty)$ but instead on the full space $D = \R^d$. It reads
\begin{align} \label{eq:bspde_1d}
\begin{split}
    \d u_t(x) &= -\Bigl(a_t(x) : \nabla_x^2 u_t(x) + H_t\bigl(x, u_t(x), \nabla_x u_t(x)\bigr) + F_t(x, u_t, \nabla_x u_t)\Bigr) \, \d t \\
    &\ \ \ - \bigl(\alpha_t : \nabla_x q_t(x) + \beta_t(x) \cdot q_t(x)\bigr) \, \d t  + q_t(x) \cdot \d W_t + \d m_t(x).
\end{split}
\end{align}
In both cases the coefficients are $\bb{F}$-progressively measurable random functions $a \define [0, T] \times \Omega \times \R \times D \to \bb{S}^d$, $H \define [0, T] \times \Omega \times D \times \R \times \R^d \to \R$, $F \define [0, T] \times \Omega \times D \times L^2(D) \times L^2(D; \R^d) \to \R$, $\lambda \define [0, T] \times \Omega \times \R^d \to \R$, $\alpha \define [0, T] \times \Omega \to \R^{d \times d}$, and $\beta \define [0, T] \times \Omega \times D \to \R^d$, where $\bb{S}^d$ denotes the space of symmetric $d \times d$ matrices. We will formulate appropriate assumptions on those coefficients below. The BSPDEs are nonlinear, nonlocal, and have a noise that takes values in an infinite-dimensional Hilbert space. 

To define our notion of solutions to BSPDE \eqref{eq:bspde}, we first need to introduce some parabolic function spaces. Let $\cal{H}$ be a separable Hilbert space with scalar product $\langle \cdot, \cdot \rangle_{\cal{H}}$. If $\cal{H}$ is an an $L^2$-space, we will drop the subscript $\cal{H}$ from the scalar product. Then, we let $L^2_{\bb{F}}([0, T]; \cal{H})$ denote the space of $\cal{H}$-valued $\bb{F}$-progressively measurable stochastic processes $h = (h_t)_{0 \leq t \leq T}$ such that $\int_0^T \ev \langle h_t, h_t\rangle_{\cal{H}} \, \d t < \infty$. The function $(h^1, h^2) \mapsto \int_0^T \ev \langle h^1_t, h^2_t\rangle_{\cal{H}} \, \d t$ for $h^1$, $h^2 \in L^2_{\bb{F}}([0, T]; \cal{H})$ defines an inner product and turns $L^2_{\bb{F}}([0, T]; \cal{H})$ into a Hilbert space. Next, we let $C_{\bb{F}}^2([0, T]; \cal{H})$ and $D_{\bb{F}}^2([0, T]; \cal{H})$ be the space of $\bb{F}$-adapted continuous and c\`adl\`ag processes $h$ with values in $\cal{H}$, such that $\ev \sup_{0 \leq t \leq T} \lVert h_t\rVert_{\cal{H}}^2 < \infty$. We endow both spaces with the norm
\begin{equation*}
    h \mapsto \bigl(\ev \sup_{0 \leq t \leq T} \lVert h_t\rVert_{\cal{H}}^2\bigr)^{1/2}
\end{equation*}
which turns them into Banach spaces.

We let $\cal{M}_{\bb{F}}^2([0, T]; \cal{H})$ be the space of $\cal{H}$-valued $\bb{F}$-martingales $m = (m_t)_{0 \leq t \leq T}$ with a.s.\@ c\`adl\`ag trajectories for which $\ev[\langle m_T, m_T\rangle_{\cal{H}}] < \infty$. We equip $\cal{M}_{\bb{F}}^2([0, T]; \cal{H})$ with the inner product $(m^1, m^2) \mapsto \ev \langle m^1_T, m^2_T\rangle_{\cal{H}}$ which turns $\cal{M}_{\bb{F}}^2([0, T]; \cal{H})$ into a Hilbert space as well. 
We say that two martingales $m^1 \in \cal{M}_{\bb{F}}^2([0, T]; \cal{H}_1)$, $m^2 \in \cal{M}_{\bb{F}}^2([0, T]; \cal{H}_2)$ with values in separable Hilbert spaces $\cal{H}_1$ and $\cal{H}_2$, respectively, are \textit{very strongly orthogonal} if for all elements $h_1 \in \cal{H}_1$, $h_2 \in \cal{H}_2$ the real-valued c\`adl\`ag martingales $\langle m^1, h_1\rangle_{\cal{H}_1}$ and $\langle m^2, h_2\rangle_{\cal{H}_2}$ are strongly orthogonal, i.e.\@ $\bigl[\langle m^1, h_1\rangle_{\cal{H}_1}, \langle m^2, h_2\rangle_{\cal{H}_2}\bigr]_t = 0$ for all $t \in [0, T]$.

Lastly, we define 
\begin{equation*}
    \bb{H}(D) = L^2_{\bb{F}}([0, T]; H^{1, 0}(D)) \times L^2_{\bb{F}}([0, T]; L^2(D; \R^d)) \times \cal{M}_{\bb{F}}^2([0, T]; L^2(D))
\end{equation*}
if the domain $D$ is $\R^d \times (\ell, \infty)$ and
\begin{equation*}
    \bb{H}(D) = L^2_{\bb{F}}([0, T]; H^1(D)) \times L^2_{\bb{F}}([0, T]; L^2(D; \R^d)) \times \cal{M}_{\bb{F}}^2([0, T]; L^2(D))
\end{equation*}
if $D = \R^d$. Here $H^{1, 0}(D)$ denotes the space of functions $f \in L^2(D)$ that have a weak derivative $\nabla_x f \define D \to \R$ in the first $d$ variables $x$, such that $\lVert \nabla_x f\rVert_{L^2} < \infty$. It becomes a Hilbert space when equipped with the inner product
\begin{equation*}
    (f_1, f_2) \mapsto \int_D f_1(x, y)f_2(x, y) \, \d x \d y + \int_D \nabla_x f_1(x, y) \cdot \nabla_x f_2(x, y) \, \d x \d y.
\end{equation*}
We denote its dual space by $H^{-1, 0}(D)$. Furthermore, for $k \geq 1$, $H^k(D)$ is the classical Sobolev space with dual $H^{-k}(D)$ and $H^1_0(D)$ denotes the subspace of elements in $H^1(D)$ that vanish on the boundary of $D$.

\begin{definition} \label{def:bspde_solution}
A \textit{solution} to BSPDE \eqref{eq:bspde} consists of a triple $(u, q, m) \in \bb{H}(D)$ such that $m$ is started from zero and is very strongly orthogonal to $W$, and
\begin{align} \label{eq:bspde_solution}
\begin{split}
    \langle u_t, \varphi\rangle &= \langle \psi, \varphi\rangle + \int_t^T \Bigl\langle H_s\bigl(\cdot, \cdot, u_s(\cdot, \cdot), \nabla_x u_s(\cdot, \cdot)\bigr) + F_s(\cdot, \cdot, u_s, \nabla_x u_s) + \beta_s \cdot q_s, \varphi \Bigr\rangle \, \d s \\
    &\ \ \ - \int_t^T \Bigl(\bigl\langle \nabla_x u_s, \nabla_x \cdot (a_s \varphi)\bigr\rangle + \langle u_s, \lambda_s \partial_y \varphi\rangle + \langle \alpha_s q_s, \nabla_x \varphi\rangle\Bigr) \, \d s \\
    &\ \ \ - \int_t^T \langle q_s, \varphi\rangle \, \d W_s - \langle m_T - m_t, \varphi\rangle
\end{split}
\end{align}
for all $\varphi \in H^1_0(D)$ and $\textup{Leb} \otimes \pr$-a.e.\@ $(t, \omega) \in [0, T] \times \Omega$. A \textit{solution} to BSPDE \eqref{eq:bspde_1d} is defined in an analogous manner but with test function from $H^1(D)$. 

We call a solution $(u, q, m)$ to BSPDE \eqref{eq:bspde} or BSPDE \eqref{eq:bspde_1d}, respectively, \textit{unique} if for all other solutions $(u', q', m')$ to BSPDE \eqref{eq:bspde} or BSPDE \eqref{eq:bspde_1d}, respectively, it holds that $(u, q, m) = (u', q', m')$ as elements of $\bb{H}(D)$.
\end{definition}

In Equation \eqref{eq:bspde_solution} we write $H_s\bigl(\cdot, \cdot, u_s(\cdot, \cdot), \nabla_x u_s(\cdot, \cdot)\bigr)$ to indicate that we integrate the function $(x, y) \mapsto H_s\bigl(x, y, u_s(x, y), \nabla_x u_s(x, y)\bigr)$ in the scalar product. In contrast, $F_s$ depends on $u_s$ and $\nabla_x u_s$ nonlocally, so that we integrate $(x, y) \mapsto F_s(x, y, u_s, \nabla_x u_s)$ in the scalar product. The test functions for BSPDE \eqref{eq:bspde} are in $H^1_0(D)$ to avoid boundary terms from appearing. Note that Definition \ref{def:bspde_solution} implicitly assumes that the coefficient $a$ is differentiable in $x$ and the coefficient $\lambda$ is differentiable in $y$. We shall always assume that this is the case.

\begin{assumption} \label{ass:bspde}
We assume that $\psi \in L^2(\Omega, \F_T, \pr; L^2(D))$ and that there exist $c_{a, \alpha} > 0$, $C > 0$, and $h_H$, $h_F \in L^2(D)$ such that
\begin{enumerate}[noitemsep, label = (\roman*)]
    \item \label{it:lipschitz_bspde} for all $(t, \omega, z) \in [0, T] \times \Omega \times D$, $(r, r', p, p') \in \R^2 \times (\R^d)^2$, and $(u, u', v, v') \in L^2(D)^2 \times L^2(D; \R^d)^2$ we have
    \begin{align*}
        \lvert H_t(z, r, p) - H_t(z, r', p')\rvert &\leq C(\lvert r - r'\rvert + \lvert p - p'\rvert), \\
        \lvert F_t(z, u, v) - F_t(z, u', v')\rvert &\leq h_F(z) (\lVert u - u'\rVert_{L^2} + \lVert v - v'\rVert_{L^2});
    \end{align*}
    \item \label{it:diff_reg} the map $z \mapsto a_t(z)$ has a weak derivative in the first $d$ variables and $z \mapsto \lambda_t(z)$ has a weak derivative in the last variable for $\pr \otimes \leb$-a.e.\@ $(t, \omega)$;
    \item \label{it:growth_bspde} the coefficients $a$, $\nabla_x \cdot a$, $\alpha$, $\beta$, $\lambda$, and $\partial_y \lambda$ are bounded by $C$ and for all $(t, \omega, z)$, $(r, p) \in \R \times \R^d$, and $(u, v) \in L^2(D) \times L^2(D; \R^d)$ we have
    \begin{align*}
        \lvert H_t(z, r, p)\rvert &\leq C(h_H(z) + \lvert r\rvert + \lvert p\rvert), \\
        \lvert F_t(z, u, v)\rvert &\leq h_F(z)(1 + \lVert u\rVert_{L^2} + \lVert v\rVert_{L^2});
    \end{align*}
    \item \label{it:nondeg_bspde} for all $(t, \omega, z)$ we have $a_t(z) - \frac{\alpha_t \alpha_t^{\top}}{2} \geq c_{a, \alpha}I_d$.
\end{enumerate}
Here $I_d$ is the identity matrix in $d$ dimensions and for two matrices $A$, $B \in \bb{S}^d$ we write $A \geq B$ if $A - B$ is positive semidefinite.
\end{assumption}

Note that Assumption \ref{ass:bspde} covers both BSPDE \eqref{eq:bspde} and BSPDE \eqref{eq:bspde_1d}. Of course, the regularity and boundedness condition on $\lambda$ is not needed for the latter, as the coefficient $\lambda$ does not enter this BSPDE.

\begin{theorem} \label{thm:bspde}
Let Assumption \ref{ass:bspde} be satisfied. Then BSPDE \eqref{eq:bspde} and BSPDE \eqref{eq:bspde_1d} have a unique strong solution $(u, q, m)$ for which $u \in D_{\bb{F}}^2([0, T]; L^2(D))$. Moreover, it holds that
\begin{equation} \label{eq:bspde_semilinear_estimate}
    \sup_{t \in [0, T]}\ev\lVert u_t\rVert_{L^2}^2 + \ev\lVert m_T\rVert_{L^2}^2 + \ev\int_0^T \lVert \nabla_x u_t \rVert_{L^2}^2 + \lVert q_t \rVert_{L^2}^2 \, \d t \leq C\bigl(1 + \ev\lVert \psi \rVert_{L^2}^2\bigr)
\end{equation}
for a constant $C > 0$ that only depends on the coefficients and $T$.
\end{theorem}

\subsection{Application to a McKean--Vlasov Control Problem with Killing at a State-Dependent Intensity} \label{sec:main_results_application}

From now on we set $d = 1$. Let us now assume that on $(\Omega, \F, \pr)$ there exists another one-dimensional Brownian motion $B$ as well as two real-valued random variables $\xi$ and $\zeta \geq 0$ with joint distribution $\L(\xi, \zeta) \in \P^2(\R^2)$, such that the quadruple $(\xi, \zeta, B, W)$ is independent. 
Then we consider the McKean--Vlasov SDE
\begin{equation} \label{eq:mfl}
    \d X_t = b(t, X_t, \nu_t, \gamma_t) \, \d t + \sigma(t, X_t) \, \d B_t + \sigma_0(t) \, \d W_t, \quad \d \Lambda_t = \lambda(t, X_t) \, \d t
\end{equation}
with initial conditions $X_0 = \xi$, $\Lambda_0 = \zeta$, and \textit{conditional subprobability measure} $\nu_t$ defined through $\langle \nu_t, \varphi\rangle = \ev[e^{-\Lambda_t}\varphi(X_t) \vert W]$ for $\varphi \in C_b(\R)$. The control process $\gamma = (\gamma_t)_{0 \leq t \leq T}$ takes values in some nonempty, compact, and convex subset $G$ of $\R^{d_G}$ with $d_G \geq 1$. We assume that $\gamma$ is $\bb{F}^{\xi, \zeta, B, W}$-progressively measurable. Next, we introduce the \textit{cost functional}
\begin{equation}
    J(\gamma) = \ev\biggl[\int_0^T e^{-\Lambda_t} f(t, X_t, \nu_t, \gamma_t) \, \d t + \psi(\nu_T)\biggr]
\end{equation}
and denote by $V$ the infimum of $J$ over control processes $\gamma$ with the properties outlined above. We call $V$ the \textit{value} of the control problem. We refer to this control problem as the \textit{strong formulation} for reasons that will become clear later on.

We studied the limit theory of a generalisation of this McKean--Vlasov control problem in a previous paper \cite{hambly_mvcp_arxiv_2023}. In particular, we showed that the McKean--Vlasov control problem is the limit of a control problem for a suitable particle approximation of McKean--Vlasov SDE \eqref{eq:mfl} as the number of particles tends to infinity. What we are interested in here is whether we can replace the progressively measurable \textit{open-loop control} $\gamma_t$ by a \textit{semiclosed-loop control} of the form $g_t(X_t)$, for a $\bb{F}^W$-progressively measurable random function $g \define [0, T] \times \Omega \times \R \to G$, without increasing the value $V$. Now, if $x \mapsto g_t(x)$ is only measurable, it is not clear if McKean--Vlasov SDE \eqref{eq:mfl}, with the control $g$ as input, admits a strong solution. Consequently, we need to work with a weakened formulation of the control problem. In the weaker formulation, $\nu_t$ is the conditional subprobability with respect to a $\sigma$-algebra $\F_T$ for a filtration $\bb{F} = (\F_t)_{0 \leq t \leq T}$ that suitably extends that of the common noise $W$, and $g$ may be progressively measurable with respect to $\bb{F}$ instead of $\bb{F}^W$. 

However, even if we admit this change, one might expect that the value over weak semiclosed-loop controls is larger than $V$ unless we allow the ``weak'' semiclosed-loop controls to also depend on $\Lambda_t$. But as stated in the introduction, this is in fact not the case.
Before we make this more precise, let us discuss the coefficients of McKean--Vlasov SDE \eqref{eq:mfl} as well as the cost functions. For $p \geq 1$, let $\cal{M}^2_{\leq 1}(\R)$ denote the space of measures $v$ on $\R$ such that $v(\R) \leq 1$ and $M_p^p(v) = \int_{\R} \lvert x\rvert^p \, \d v(x) < \infty$. We equip $\cal{M}^2_{\leq 1}(\R)$ with the metric
\begin{equation} \label{eq:subprob_metric}
    d_p(v_1, v_2) = W_p(\eta_1, \eta_2) + \lvert \nu_1(\R) - \nu_2(\R)\rvert,
\end{equation}
where $\eta_i = v_i + (1 - v_i(\R))\delta_0$, $W_p$ denotes the $p$-Wasserstein distance for probability measures on $\R$, and $v_1$, $v_2 \in \M^p_{\leq 1}(\R)$. By the Rubinstein-Kantorovich duality, we can rewrite $d_1$ as $d_1(v_1, v_2) = \sup_{\lVert \varphi\rVert_{\text{Lip}} \leq 1} \langle v_1 - v_2, \varphi\rangle$, where for a Lipschitz continuous function $\varphi \define \R \to \R$ we denote by $\lVert \varphi\rVert_{\text{Lip}}$ the maximum between the Lipschitz constant of $\varphi$ and $\lvert \varphi(0)\rvert$. This motivates us to further introduce the metric $d_0$ defined as $d_0(v_1, v_2) = \sup_{\lVert \varphi\rVert_{\text{bLip}} \leq 1} \langle v_1 - v_2, \varphi\rangle$, where now $\lVert \varphi\rVert_{\text{bLip}}$ denotes the maximum between the Lipschitz constant and $\lVert \varphi\rVert_{\infty}$ for a bounded Lipschitz continuous function $\varphi \define \R \to \R$. 

With this in mind, we let $b$ and $f \define [0, T] \times \R \times \cal{M}^2_{\leq 1}(\R) \times G \to \R$, $\sigma$ and $\lambda \define [0, T] \times \R \to \R$, $\sigma_0 \define [0, T] \to \R$, and $\psi \define \cal{M}^2_{\leq 1}(\R) \to \R$ be measurable functions. Further, we assume that there exist measurable maps $b_0$ and $f_0 \define [0, T] \times \R \times \cal{M}^2_{\leq 1}(\R) \to \R$, $b_1$ and $f_1 \define [0, T] \times \R \times G \to \R$ such that $b(t, x, v, g) = b_0(t, x, v) + b_1(t, x, g)$ and $f(t, x, v, g) = f_0(t, x, v) + f_1(t, x, g)$ for all $(t, x, v, g) \in [0, T] \times \R \times \cal{M}^2_{\leq 1}(\R) \times G \to \R$.

\begin{assumption} \label{ass:mfl}
We assume there exists constant $C > 0$ and $c > 0$ such that
\begin{enumerate}[noitemsep, label = (\roman*)]
    \item \label{it:continuity_coeff_mfl} the map $b_1$ is continuous in $g$ and for $(t, x, x', v, v', g)$ and $h \in \{b_1(\cdot, \cdot, g), \sigma, \lambda\}$ we have
    \begin{align*}
        \lvert b_0(t, x, v) - b_0(t, x', \nu')\rvert + \lvert h(t, x) - h(t, x')\rvert \leq C(\lvert x - x'\rvert + d_1(v, v'));
    \end{align*}
    \item \label{it:continuity_cost_mfl} the maps $(x, v) \mapsto f_0(t, x, v)$ and $(x, g) \mapsto f_1(t, x, g)$ are uniformly continuous for $(t, x, v, g)$ in compact subsets of $[0, T] \times \R \times \cal{M}^2_{\leq 1}(\R) \times G$, and $\psi$ is continuous;
    \item \label{it:growth_mfl} the coefficients $\sigma$ and $\sigma_0$ are bounded by $C$ and for all $(t, x, v, g)$ we have
    \begin{align*}
        \lvert b_0(t, x, v) \rvert + \lvert b_1(t, x, g) \rvert + \lvert \lambda(t, x)\rvert &\leq C(1 + \lvert x \rvert + M_2(v) + \lvert g\rvert), \\
        \lvert f_0(t, x, v) \rvert + \lvert f_1(t, x, g) \rvert + \lvert \psi(v)\rvert &\leq C(1 + \lvert x \rvert^2 + M_2^2(\nu));
    \end{align*}
    \item \label{it:nondeg_mfl} for all $(t, x)$ we have $\sigma^2(t, x) \geq c$;
    \item \label{it:intensity_mfl} for all $(t, x)$ we have $\lambda(t, x)$ if $x < 0$ and $\lambda(t, x) = 0$ otherwise.
\end{enumerate}
\end{assumption}

\begin{definition} \label{def:weak_setup}
A \textit{weak setup} is a tuple $\bb{S} = (\Omega, \F, \pr, \bb{G}, \bb{F}, \xi, \zeta, B, W)$ such that
\begin{enumerate}[noitemsep, label = (\roman*)]
    \item $(\Omega, \F, \bb{G}, \pr)$ is a complete filtered probability space and $\bb{F}$ is a complete subfiltration of $\bb{G}$;
    \item $\xi$ and $\zeta$ are $\cal{G}_0$-measurable random variables, $B$ and $W$ are $\bb{G}$-Brownian motions, and $W$ is $\bb{F}$-adapted;
    \item $\xi$, $\zeta$, $B$, and $\F_T$ are independent and for $t \in [0, T]$ we have $\pr(A \vert \F_t) = \pr(A \vert \F_T)$ a.s.\@ for any $A \in \cal{G}_t \lor \sigma(B)$.
\end{enumerate}
By an \textit{extension} of a weak setup $\bb{S} = (\Omega, \F, \pr, \bb{G}, \bb{F}, \xi, \zeta, B, W)$ we mean a weak setup $(\tilde{\Omega}, \tilde{\F}, \tilde{\pr}, \tilde{\bb{G}}, \tilde{\bb{F}}, \tilde{\xi}, \tilde{\zeta}, \tilde{B}, \tilde{W})$ such that $\tilde{\Omega} = \Omega \times \Omega'$, $\tilde{\F} = \F \otimes \F'$, $\bb{G} = (\cal{G}_t \otimes \cal{G}'_t)_{0 \leq t \leq T}$, and $\bb{F} = (\F_t \otimes \{\varnothing, \Omega'\})_{0 \leq t \leq T}$, where $(\cal{G}'_t)_{0 \leq t \leq T}$ is a filtration on the measurable space $(\Omega', \F')$, and the random variables $\tilde{\xi}$, $\tilde{\zeta}$, $\tilde{B}$, and $\tilde{W}$ on $\tilde{\Omega}$ are defined by $\tilde{\xi}(\omega, \omega') = \xi(\omega)$, $\tilde{\zeta}(\omega, \omega') = \zeta(\omega)$, $\tilde{B}(\omega, \omega') = B(\omega)$, and $\tilde{W}(\omega, \omega') = W(\omega)$ for $(\omega, \omega') \in \tilde{\Omega}$. The probability $\tilde{\pr}$ on $(\tilde{\Omega}, \tilde{\F})$ satisfies $\tilde{\pr}(A \times \Omega') = \pr(A)$ for all $A \in \F$.
\end{definition}

For a given weak setup $\bb{S} = (\Omega, \F, \pr, \bb{G}, \bb{F}, \xi, \zeta, B, W)$ and every $\bb{G}$-progressively measurable $G$-valued process $\gamma = (\gamma_t)_{0 \leq t \leq T}$, referred to as an $\bb{S}$\textit{-admissible control}, we can consider the McKean--Vlasov SDE
\begin{equation} \label{eq:mfl_weak}
    \d X_t = b(t, X_t, \nu_t, \gamma_t) \, \d t + \sigma(t, X_t) \, \d B_t + \sigma_0(t) \, \d W_t, \quad \d \Lambda_t = \lambda(t, X_t, \nu_t) \, \d t
\end{equation}
with initial conditions $X_0 = \xi$, $\Lambda_0 = \zeta$, and $\nu_t$ defined through $\langle \nu_t, \varphi\rangle = \ev[e^{-\Lambda_t}\varphi(X_t) \vert \F_T]$ for $\varphi \in C_b(\R)$. In \cite[Proposition 3.12]{hambly_mvcp_arxiv_2023} we proved that if Assumption \ref{ass:mfl} is in place, then McKean--Vlasov SDE \eqref{eq:mfl_weak} has a unique strong solution for any $\bb{S}$-admissible control $\gamma$. We say that the control $\gamma$ is a \textit{semiclosed-loop control} if there exists an $\bb{F}$-progressively measurable random function $g \define [0, T] \times \Omega \times \R^2 \to G$ such that $\gamma_t = g_t(X_t, \Lambda_t)$. We call $g$ the \textit{feedback function} associated to $\gamma$. 
Next, we define the \textit{cost functional}
\begin{equation}
    J_{\bb{S}}(\gamma) = \ev\biggl[\int_0^T e^{-\Lambda_t} f(t, X_t, \nu_t, \gamma_t) \, \d t + \psi(\nu_T)\biggr]
\end{equation}
and the corresponding \textit{value} $V_{\bb{S}}$, i.e.\@ the infimum of $J_{\bb{S}}$ over all $\bb{S}$-admissible controls. We say that an $\bb{S}$-admissible control $\gamma$ is \textit{optimal} if $J_{\bb{S}}(\gamma) = V_{\bb{S}}$. 
According to \cite[Appendix A.2]{hambly_mvcp_arxiv_2023}, if Assumption \ref{ass:mfl} is in place, the value $V_{\bb{S}}$ is independent of the setup $\bb{S}$ and is, in fact, equal to $V$. We call this control problem the \textit{weak formulation}.

With all this in mind, we can now formulate our central claim: there exists an optimal semiclosed-loop control $\gamma^{\ast}$ on some weak setup $\bb{S}$, such that the corresponding feedback function $\tilde{g}^{\ast}$ does not depend on the intensity, i.e.\@ there exists an $\bb{F}$-progressively measurable random function $g^{\ast} \define [0, T] \times \Omega \times \R \to G$ such that $\tilde{g}^{\ast}_t(x, y) = g^{\ast}_t(x)$ for a.e.\@ $(x, y) \in \R^2$ and $\pr \otimes \text{Leb}$-a.e.\@ $(t, \omega) \in [0, T] \times \Omega$. 
In what follows, we will provide an outline that explains how and under what assumptions we prove this result.

\begin{assumption} \label{ass:convexity}
The map $g \mapsto b_1(t, x, g)$ is linear and the map $g \mapsto f_1(t, x, g)$ is convex for all $(t, x)$.
\end{assumption}

If Assumption \ref{ass:convexity} holds in addition to Assumption \ref{ass:mfl}, then \cite[Proposition A.5]{hambly_mvcp_arxiv_2023} states that there exists a weak setup $\bb{S} = (\Omega, \F, \pr, \bb{G}, \bb{F}, \xi, \zeta, B, W)$ which admits an optimal $\bb{S}$-admissible semiclosed-loop control. We will fix this weak setup from now on.

If $g$ is any admissible semiclosed-loop control on $\bb{S}$, then the \textit{conditional joint law} $\mu_t = \L(X_t, \Lambda_t \vert \F_T)$ of $X$ and $\Lambda$ satisfies the stochastic Fokker--Planck equation 
\begin{equation} \label{eq:sfpe_joint}
    \d \langle \mu_t, \varphi\rangle = \bigl\langle \mu_t, \L_2\varphi(t, \cdot, \cdot, \nu_t, g_t)\bigr\rangle\, \d t + \langle \mu_t, \sigma_0(t) \partial_x \varphi \rangle \, \d W_t
\end{equation}
for $\varphi \in C_c^2(\R^2)$ with initial condition $\mu_0 = \L(\xi, \zeta)$. Here $\L_2$ denotes the generator of the two-dimensional process $(X, \Lambda)$ which is given by
\begin{equation} \label{eq:generator_sfpe_joint}
    \L_2 \varphi(t, x, y, v, h) = b(t, x, v, h)\partial_x \varphi(x, y) + \lambda(t, x, v)\partial_y \varphi(x, y) + a_t(x) \partial_x^2 \varphi(x, y)
\end{equation}
for $\varphi \in C_b^2(\R^2)$ and $(t, x, y, v, h) \in [0, T] \times \R^2 \times \cal{M}^2_{\leq 1}(\R) \times G$. The function $a \define [0, T] \times \R \to \R$ is defined by $a_t(x) = \frac{1}{2}(\sigma^2(t, x) + \sigma_0^2(t))$. Note, that $\nu_t$ is determined by the stochastic Fokker--Planck equation, so that the SPDE \eqref{eq:sfpe_joint} is nonlinear. Indeed, we can introduce the function $\cal{S} \define \P(\R \times [0, \infty)) \to \cal{M}_{\leq 1}(\R)$ defined by
\begin{equation} \label{eq:subprobability_function}
    \langle \cal{S}(m), \varphi\rangle = \int_{\R \times [0, \infty)} e^{-y} \varphi(x) \, \d m(x, y)
\end{equation}
for $\varphi \in C_b(\R)$, which allows us to write $\nu_t = \cal{S}(\mu_t)$.

Since on our chosen setup $\bb{S}$ the infimum of $J_{\bb{S}}$ over arbitrary $\bb{S}$-admissible controls and semiclosed-loop controls coincides, we can consider the control problem for McKean--Vlasov SDE \eqref{eq:mfl_weak} from the point of view of the stochastic Fokker--Planck equation for $\mu = (\mu_t)_{0 \leq t \leq T}$. That is, we study the optimal control of SPDE \eqref{eq:sfpe_joint}. We will adopt $L^2(\R^2)$ as a state space for this control problem. We justify this with Proposition \ref{prop:sfpe_joint} below. But first we need to introduce some stronger assumptions.

\begin{assumption} \label{ass:spfe_joint}
The coefficients $b$, $\sigma$, $\sigma_0$, and $\lambda$ and costs $f$ and $\psi$ satisfy Assumption \ref{ass:mfl} and there exists a constant $C > 0$ such that
\begin{enumerate}[noitemsep, label = (\roman*)]
    \item \label{it:spde_ic} the initial distribution $\L(\xi, \zeta)$ has a density in $L^2(\R^2)$;
    \item \label{it:spde_growth} the coefficients $b_0$, $b_1$, and $\lambda$ are bounded by $C$;
    \item \label{it:spde_lipschitz} for all $v$, $v' \in \cal{M}^2_{\leq 1}(\R) \cap L^2(\R)$ and $(t, x)$ we have $\lvert b_0(t, x, v) - b_0(t, x, v')\rvert \leq C \lVert v - v'\rVert_{L^2}$.
\end{enumerate}
\end{assumption}

In Assumption \ref{ass:spfe_joint} we identified an element of $\cal{M}^2_{\leq 1}(\R)$ with its density in $L^2(\R)$. Let $\cal{G}_{\bb{F}}^k$ denote the space of $\bb{F}$-progressively measurable random functions $[0, T] \times \Omega \times \R^k \to G$.


\begin{proposition} \label{prop:sfpe_joint}
Let Assumption \ref{ass:mfl} be satisfied. Then for any semiclosed-loop control $g$, the process $\mu = (\mu_t)_{0 \leq t \leq T}$ defined by $\mu_t = \L(X_t, \Lambda_t \vert \F_T)$ solves SPDE \eqref{eq:sfpe_joint}. 

Next, if Assumption \ref{ass:spfe_joint} holds and $g \in \cal{G}_{\bb{F}}^2$, then SPDE \eqref{eq:sfpe_joint} has a unique strong solution $\mu = (\mu_t)_{0 \leq t \leq T}$ with values in $C([0, T]; \P^2(\R^2))$ such that $\ev \sup_{0 \leq t \leq T} M_2^2(\mu_t) < \infty$ and $\supp \mu_t \subset \R \times [0, \infty)$. Moreover, $\pr$-a.s.\@ the probability measure $\mu_t$ has a density $\rho_t$ in $L^2(\R^2)$ for all $t \in [0, T]$ and $\rho = (\rho_t)_{0 \leq t \leq T} \in C_{\bb{F}}^2([0, T]; L^2(\R^2)) \cap L_{\bb{F}}^2([0, T]; H^{1, 0}(\R^2))$ with
\begin{equation} \label{eq:sfpe_uniform_bound}
    \sup_{0 \leq t \leq T} \lVert \rho_t\rVert_{L^2}^2 + \int_0^T \lVert \rho_t\rVert_{H^{1, 0}}^2 \, \d t \leq C\lVert \rho_0\rVert_{L^2}^2
\end{equation}
for a constant $C > 0$ depending only on the coefficients and $T$.

Finally, for $g$ and $\mu$ as above, extending the weak setup if necessary, there exists a solution $(X, \Lambda)$ of McKean--Vlasov SDE \eqref{eq:mfl_weak} with semiclosed-loop control $\gamma = (\gamma_t)_{0 \leq t \leq T}$ given by $\gamma_t = g_t(X_t)$ and such that $\L(X_t, \Lambda_t \vert \F_T) = \mu_t$ for $t \in [0, T]$.
\end{proposition}

From now on we will no longer distinguish between $\mu_t$ and its density and simply write $\mu \in C_{\bb{F}}^2([0, T]; L^2(\R^2)) \cap L_{\bb{F}}^2([0, T]; H^{1, 0}(\R^2))$.

Note that the last statement of Proposition \ref{prop:sfpe_joint} implies that under Assumption \ref{ass:spfe_joint}, up to changing the weak setup, any $g \in \cal{G}_{\bb{F}}^2$ induces a semiclosed-loop control and vice versa. Thus, the control problem for the stochastic Fokker--Planck equation \eqref{eq:sfpe_joint} with controls $g \in \cal{G}_{\bb{F}}^2$, which we refer to as the \textit{SPDE formulation}, is equivalent to the control problem for McKean--Vlasov SDE \eqref{eq:mfl_weak}, i.e.\@ the weak formulation, with semiclosed-loop controls on an arbitrary weak setup. We can now introduce a cost functional $J_{\text{cl}}$ for controls $g \in \cal{G}_{\bb{F}}^2$ defined by
\begin{equation} \label{eq:cost_func_closed}
    J_{\text{cl}}(g) = \ev\biggl[\int_0^T \bigl\langle \mu_t, \tilde{f}(t, \cdot, \cdot, \nu_t, g_t)\bigr\rangle + \psi(\nu_T)\biggr],
\end{equation}
where $\tilde{f} \define [0, T] \times \R^2 \times \cal{M}^2_{\leq 1}(\R) \times G \to \R$ is given by $\tilde{f}(t, x, y, v, h) = e^{-y}f(t, x, v, h)$. We similarly define $\tilde{f}_0$ and $\tilde{f}_1$. Note that because of the equivalence alluded to above, the feedback function $\tilde{g}^{\ast} \in \cal{G}_{\bb{F}}^2$ of the optimal semiclosed-loop control $\gamma^{\ast}$, whose existence is guaranteed under Assumption 
\ref{ass:convexity}, is an optimal control for the SPDE formulation. That is, $J_{\text{cl}}(\tilde{g}^{\ast}) = \inf_{g \in \cal{G}_{\bb{F}}^2}J_{\text{cl}}(g)$.

For the control problem for the stochastic Fokker--Planck equation, we can derive a necessary stochastic maximum principle (SMP). Indeed, let $g \in \cal{G}_{\bb{F}}^2$ and let $\mu$ denote the corresponding solution to SPDE \eqref{eq:sfpe_joint} by $\mu$. Then we define the functions $\tilde{K}^{\mu} \define [0, T] \times \Omega \times \R \times (0, \infty) \times \R \times G \to \R$ and $\tilde{F}^{\mu} \define [0, T] \times \Omega \times \R \times [0,\infty) \times L^2(\R \times (0, \infty)) \to \R$ by
\begin{align*}
    \tilde{K}^{\mu}_t(x, y, p, h) = b(t, x, \nu_t, h)p + \tilde{f}(t, x, y, \nu_t, h)
\end{align*}
for $(t, \omega, x, y, p) \in [0, T] \times \Omega \times \R \times (0, \infty) \times \R \times G$ and
\begin{align*}
    \tilde{F}^{\mu}_t(x, y, p) &= \bigl\langle \mu_t, e^{-y} Db_0(t, \cdot, \nu_t)(x) p\bigr\rangle + \bigl\langle \mu_t, e^{-y} D\tilde{f}_0(t, \cdot, \cdot, \nu_t)(x) \bigr\rangle
\end{align*}
for $(t, \omega, x, y, p) \in [0, T] \times \Omega \times \R \times (0, \infty) \times L^2(\R \times (0,\infty))$. Here $D$ denotes the linear functional derivative of a function $\cal{M}^2_{\leq 1}(\R) \to \R$, which we define next.

\begin{definition} \label{def:linear_functional_derivatve}
Let $E \subset \R^d$ be closed and assume that $\cal{E}$ is a closed and convex subset of $\cal{M}^p(E)$ for $p \geq 1$. A map $F \define \cal{E} \to \R$ is said to have a \textit{linear functional derivative} if there exists a continuous function $DF \define \cal{E} \times E \to \R$ such that 
\begin{enumerate}[noitemsep, label = (\roman*)]
    \item \label{it:growth_lfd} for any bounded subset $\cal{K} \subset \cal{E}$ the function $x \mapsto DF(m)(x)$ has at most polynomial growth of order $p$ in $x \in E$ uniformly in $m \in \cal{K}$;
    \item \label{eq:ftc_lfd} for any $m$, $m' \in \cal{E}$ it holds that
    \begin{equation}
        F(m) - F(m') = \int_0^1 \int_E DF(t m + (1 - t) m')(x) \, \d (m - m')(x) \, \d t.
    \end{equation}
\end{enumerate}
\end{definition}
With this at hand, we consider the adjoint process
\begin{align} \label{eq:adjoint}
\begin{split}
    \d \tilde{u}_t(x, y) &= -\Bigl(\tilde{K}^{\mu}_t\bigl(x, y, \partial_x \tilde{u}_t(x, y), g_t(x, y)\bigr) + \lambda_t(x)\partial_y \tilde{u}_t(x, y) + \tilde{F}^{\mu}_t\bigl(x, y, \partial_x \tilde{u}_t\bigr)\Bigr) \, \d t \\
    &\ \ \ - \bigl(a_t(x) \partial_x^2 \tilde{u}_t(x, y) + \sigma_0(t) \partial_x \tilde{q}_t(x, y)\bigr) \, \d t + \tilde{q}_t(x, y) \, \d W_t + \d \tilde{m}_t(x, y)
\end{split}
\end{align}
with terminal condition $\tilde{u}_T(x, y) = e^{-y}D\psi(\nu_T)(x)$ for $(x, y) \in \R \times (\ell, \infty)$, where $\ell \in \R$. Here we set $\lambda_t(x) = \lambda(t, x)$. Usually, we will assume that $\ell = 0$, but for the proof of the SMP it is convenient to have a solution that is defined in an arbitrarily small extension of the interval $(0, \infty)$. We interpret BSPDE \eqref{eq:adjoint} in the sense of Definition \ref{def:bspde_solution}. Before we can state an existence result for this BSPDE, we have to impose additional assumptions.

\begin{assumption} \label{ass:adjoint}
We assume that there exist $C > 0$ and $h_b$, $h_f$ and $h_{\psi} \in L^2(\R)$ such that
\begin{enumerate}[noitemsep, label = (\roman*)]
    \item \label{it:cost_l2} for all $(t, x, v, g)$ we have $\lvert f_0(t, x, v)\rvert + \lvert f_1(t, x, g)\rvert \leq h_f(x)$;
    \item \label{it:lfd} for all $(t, x)$ the maps $v \mapsto b_0(t, x, v)$, $v \mapsto f_0(t, x, v)$, and $v \mapsto \psi(v)$ have linear functional derivatives $Db_0$, $Df_0$ and $D\psi$, which are measurable functions $[0, T] \times \R \times \cal{M}^2_{\leq 1}(\R) \times \R \to \R$ and $\cal{M}^2_{\leq 1}(\R) \times \R \to \R$, respectively;
    \item \label{it:lfd_l2} for all $(t, x, v)$ and $y$ we have
    \begin{equation*}
        \lvert Db_0(t, x, v)(y)\rvert \leq h_b(y), \qquad \lvert Df_0(t, x, v)(y)\rvert \leq h_f(y), \qquad \lvert D\psi(v)(y)\rvert \leq h_{\psi}(y).
    \end{equation*}
\end{enumerate}
\end{assumption}

Under Assumptions \ref{ass:spfe_joint} and \ref{ass:adjoint} the existence of solutions to BSPDE \eqref{eq:adjoint} is a simple application of Theorem \ref{thm:bspde}.

\begin{proposition} \label{prop:adjoint_existence}
Let Assumptions \ref{ass:spfe_joint} and \ref{ass:adjoint} be satisfied. Then, for any $g \in \cal{G}_{\bb{F}}^2$, BSPDE \eqref{eq:adjoint} has a unique strong solution $(\tilde{u}, \tilde{q}, \tilde{m})$.
\end{proposition}

Note that since $\mu_t(x, y) = 0$ if $y \leq 0$, the uniqueness implies that if we have two distinct boundaries $\ell < \ell' \leq 0$ but use the same control $g \in \cal{G}_{\bb{F}}^2$, then the solutions to BSPDE \eqref{eq:adjoint} on the domains $\R \times (\ell, \infty)$ and $\R \times (\ell', \infty)$ coincide on the smaller half-plane $\R \times (\ell', \infty)$. Hence, the choice of $\ell$ is irrelevant and we will suppress it from now on. 

\begin{assumption} \label{ass:smp}
We assume that the drift coefficient $b_1$ and the running cost $f_1$ satisfy Assumption \ref{ass:convexity} and that there exists a $C > 0$ such that
\begin{enumerate}[noitemsep, label = (\roman*)]
    \item \label{it:smp_cont} for all $(t, x)$ and $y$ the maps $\cal{M}^2_{\leq 1}(\R) \cap L^2(\R) \ni v \mapsto Db_0(t, x, v)(y)$, $\cal{M}^2_{\leq 1}(\R) \cap L^2(\R) \ni v \mapsto Df_0(t, x, v)(y)$ and $\cal{M}^2_{\leq 1}(\R) \cap L^2(\R) \ni v \mapsto D\psi(v)(y)$ are continuous with respect to $\lVert \cdot\rVert_{L^2}$;
    \item \label{it:smp_ext} there exists an extension of $f_1$ (also denoted by $f_1$) to a measurable function $[0, T] \times \R \times O \to \R$ for some open set $O \subset \R^{d_G}$ that includes $G$, such that for all $(t, x)$ the map $O \ni g \mapsto f_1(t, x, g)$ is convex and continuously differentiable and for all $\partial_g f_1$ is bounded by $C$.
\end{enumerate}
\end{assumption}

We can now state the SMP.

\begin{theorem} \label{thm:smp}
Let Assumptions \ref{ass:spfe_joint}, \ref{ass:adjoint}, and \ref{ass:smp} be satisfied. Suppose further that $g \in \cal{G}_{\bb{F}}^2$ is an optimal control and let $(\tilde{u}, \tilde{q}, \tilde{m})$ be a solution to BSPDE \eqref{eq:adjoint}. Then it holds that
\begin{equation} \label{eq:smp}
    \tilde{K}^{\mu}_t\bigl(x, y, \partial_x \tilde{u}_t(x, y), g_t(x, y)\bigr) = \inf_{h \in G} \tilde{K}^{\mu}_t\bigl(x, y, \partial_x \tilde{u}_t(x, y), h\bigr)
\end{equation}
for $\mu_t$-a.e.\@ $(x, y) \in \R \times (0, \infty)$ and $\textup{Leb} \otimes \pr$-a.e.\@ $(t, \omega) \in [0, T] \times \Omega$.
\end{theorem}

We would like to use the SMP to derive a semilinear BSPDE from BSPDE \eqref{eq:adjoint} satisfied by the adjoint process. However, this is not completely straightforward, since the equality in \eqref{eq:smp} only holds $\mu_t$-a.e.\@ and there is no guarantee that $\mu_t$ is equivalent to the Lebesgue measure on $\R \times (0, \infty)$. To deal with this issue, we need to modify $g_t$ on the set $\{(x, y) \in \R \times (0, \infty) \define \mu_t(x, y) = 0\}$. Of course, this has no impact on $\mu$, i.e.\@ $\mu$ will solve SPDE \eqref{eq:sfpe_joint} with the modification $g^{\ast}$ of $g$ as an input, but the adjoint process corresponding to $g^{\ast}$ may differ. To describe the modification, we have to consider the one-dimensional semilinear BSPDE 
\begin{align} \label{eq:adjoint_1d}
\begin{split}
    \d u_t(x) &= -\Bigl(a_t(x) \partial_x^2 u_t(x) + H^{\nu}_t\bigl(x, u_t(x), \partial_x u_t(x)\bigr) + F^{\nu}_t\bigl(x, \partial_x u_t\bigr)\Bigr) \, \d t \\
    &\ \ \ - \sigma_0(t) \partial_x q_t(x) \, \d t + q_t(x) \, \d W_t + \d m_t(x)
\end{split}
\end{align}
with terminal condition $u_T(x) = D\psi(\nu_T)(x)$ for $x \in \R$. Here, the functions $H \define [0, T] \times \Omega \times \R^3 \to \R$ and $F \define [0, T] \times \Omega \times \R \times L^2(\R) \to \R$ are given by
\begin{equation}
    H^{\nu}_t(x, p, q) = \inf_{h \in G} \bigl(b(t, x, \nu_t, h) q + f(t, x, \nu_t, h)\bigr) - \lambda_t(x) p
\end{equation}
for $(t, \omega, x, p, q) \in [0, T] \times \Omega \times \R^3$ and
\begin{align*}
    F^{\nu}_t(x, p) = \bigl\langle \nu_t, Db_0(t, \cdot, \nu_t)(x)p\bigr\rangle + \bigl\langle \nu_t, Df_0(t, \cdot, \nu_t)(x)\bigr\rangle
\end{align*}
for $(t, \omega, x, p) \in [0, T] \times \Omega \times \R \times L^2(\R)$. BSPDE \eqref{eq:adjoint_1d} is understood in the sense of Definition \ref{def:bspde_solution} and under the following additional assumptions we can deduce existence and uniqueness of solutions from Theorem \ref{thm:bspde}.

\begin{proposition} \label{prop:bspde_1}
Let Assumptions \ref{ass:spfe_joint} and \ref{ass:adjoint} be satisfied. Then BSPDE \eqref{eq:adjoint_1d} has a unique strong solution $(u, q, m)$.
\end{proposition}

Assumption \ref{ass:convexity} guarantees the existence of a measurable function $g_{\ast} \define [0, T] \times \R \times \R \to G$ such that $g_{\ast}(t, x, p)$ is the minimiser of the map $h \mapsto b_1(t, x, h)p + f_1(t, x, h)$ for $(t, x, p) \in [0, T] \times \R \times \R$. 
Let us define the control $g^{\ast} \in \cal{G}_{\bb{F}}^2$ by
\begin{equation}
    g^{\ast}_t(x, y) = \bf{1}_{\mu_t(x, y) > 0} g_t(x, y) + \bf{1}_{\mu_t(x, y) = 0} g_{\ast}\bigl(t, x, \partial_x u_t(x)\bigr).
\end{equation}
Next, we introduce the function $\tilde{H}^{\mu} \define [0, T] \times \Omega \times \R \times (0, \infty) \times \R \to \R$ defined by
\begin{equation} \label{eq:hamiltonian_2d}
    \tilde{H}^{\mu}_t(x, y, p) = \bf{1}_{\mu_t(x, y) > 0}\inf_{h \in G} \tilde{K}^{\mu}_t(x, y, p, h) + \bf{1}_{\mu_t(x, y) = 0}\tilde{K}^{\mu}_t\bigl(x, y, p, g_{\ast}\bigl(t, x, \partial_x u_t(x)\bigr)\bigr).
\end{equation}
Then, in view of Theorem \ref{thm:smp}, if $g$ and, therefore, $g^{\ast}$ are optimal, then for the adjoint process $(\tilde{u}, \tilde{q}, \tilde{m})$ corresponding to $g^{\ast}$, we obtain
\begin{align*}
    \tilde{K}^{\mu}_t\bigl(x, y, \partial_x \tilde{u}_t(x, y), g^{\ast}_t(x, y)\bigr) &= \bf{1}_{\mu_t(x, y) > 0} \inf_{h \in G}\tilde{K}^{\mu}_t\bigl(x, y, \partial_x \tilde{u}_t(x, y), h\bigr) \\
    &\ \ \ + \bf{1}_{\mu_t(x, y) = 0}\tilde{K}^{\mu}_t\Bigl(x, y, \partial_x \tilde{u}_t(x, y), g_{\ast}\bigl(t, x, \partial_x u_t(x)\bigr)\Bigr) \\
    &= \tilde{H}^{\mu}_t\bigl(x, y, \partial_x \tilde{u}_t(x, y)\bigr),
\end{align*}
so that $(\tilde{u}, \tilde{q}, \tilde{m})$ satisfies the semilinear BSPDE
\begin{align} \label{eq:bspde_2d}
\begin{split}
    \d \tilde{u}_t(x, y) &= -\Bigl(a_t(x) \partial_x^2 \tilde{u}_t(x, y) + \tilde{H}^{\mu}_t\bigl(x, y, \partial_x \tilde{u}_t(x, y)\bigr) + \tilde{F}^{\mu}_t\bigl(x, y, \partial_x \tilde{u}_t\bigr)\Bigr) \, \d t \\
    &\ \ \ - \bigl(\lambda_t(x)\partial_y \tilde{u}_t(x, y) + \sigma_0(t) \partial_x \tilde{q}_t(x, y)\bigr) \, \d t + \tilde{q}_t(x, y) \, \d W_t + \d \tilde{m}_t(x, y)
\end{split}
\end{align}
with terminal condition $\tilde{u}_t(x, y) = e^{-y}D\psi(\nu_T)(x)$ for $(x, y) \in \R \times (0, \infty)$. Again, we interpret this BSPDE in the sense of Definition \ref{def:bspde_solution}. 

\begin{proposition} \label{prop:bspde_2_from_1}
Let Assumptions \ref{ass:spfe_joint} and \ref{ass:adjoint} be satisfied and let $(u, q, m)$ be the unique strong solution to BSPDE \eqref{eq:bspde_1d}. Then the unique strong solution $(\tilde{u}, \tilde{q}, \tilde{m})$ of BSPDE \eqref{eq:bspde_2d} is given by $\tilde{u}_t(x, y) = e^{-y}u_t(x)$, $\tilde{q}_t(x, y) = e^{-y}q_t(x)$, $\tilde{m}_t(x, y) = e^{-y}m_t(x)$ for $(t, \omega, x, y) \in [0, T] \times \Omega \times \R \times (0, \infty)$.
\end{proposition}

Let us assume that $h \mapsto f_1(t, x, h)$ is strictly convex, so that $h \mapsto b_1(t, x, h)p + f_1(t, x, h)$ has a unique minimum in $G$ for all $(t, x, p) \in [0, T] \times \R \times \R$. Then combining Proposition \ref{prop:bspde_2_from_1} with the necessary SMP, Equation \eqref{eq:smp}, implies that on the set $\{(x, y) \in \R \times (0, \infty) \define \mu_t(x, y) > 0\}$ it holds that
\begin{align*}
    \tilde{K}^{\mu}_t\bigl(x, y, \partial_x \tilde{u}_t(x, y), g^{\ast}_t(x, y)\bigr)
    &= \inf_{h \in G} \tilde{K}^{\mu}_t\bigl(x, y, \partial_x \tilde{u}_t(x, y), h\bigr) \\
    &= \inf_{h \in G} e^{-y} \bigl(b(t, x, \nu_t, h) \partial_x u_t(x) + f(t, x, \nu_t, h)\bigr)
\end{align*}
for $\pr \otimes \text{Leb}$-a.e.\@ $(t, \omega) \in [0, T] \times \Omega$. The infimum on the right-hand side is uniquely attained at $g_{\ast}\bigl(t, x, \partial_x u_t(x)\bigr)$, so we deduce $g^{\ast}_t(x, y) = g_{\ast}\bigl(t, x, \partial_x u_t(x)\bigr)$. On the set $\{(x, y) \in \R \times (0, \infty) \define \mu_t(x, y) = 0\}$ this equality holds by definition of $g^{\ast}$, so altogether we find $g^{\ast}_t(x, y) = g_{\ast}\bigl(t, x, \partial_x u_t(x)\bigr)$ for $(x, y) \in \R \times (0, \infty)$ and $\pr \otimes \text{Leb}$-a.e.\@ $(t, \omega) \in [0, T] \times \Omega$. Hence, we have proved the following result.

\begin{corollary} \label{cor:equivalence_l2}
Let Assumptions \ref{ass:spfe_joint}, \ref{ass:adjoint}, and \ref{ass:smp} be satisfied and suppose that $h \mapsto f_1(t, x, h)$ is strictly convex. Then there exists an optimal control $g^{\ast} \in \cal{G}_{\bb{F}}^1$ of the form $g^{\ast}_t(x) = g_{\ast}\bigl(t, x, \partial_x u_t(x)\bigr)$ for $x \in \R$ and $\pr \otimes \text{Leb}$-a.e.\@ $(t, \omega) \in [0, T] \times \Omega$, where $(u, q, m)$ is the unique strong solution of BSPDE \eqref{eq:bspde_1d}. In fact, every optimal control is of this form.
\end{corollary}

The assumptions of Corollary \ref{cor:equivalence_l2} are rather strong, particularly the assumption that the functions $x \mapsto f(t, x, v, g)$, $x' \mapsto Db_0(t, x ,v, g)(x')$, $x' \mapsto Df_0(t, x, v, g)(x')$, and $x' \mapsto D\psi(v)(x')$ are in $L^2(\R)$. However, if we relax them, Proposition \ref{prop:sfpe_joint} is no longer applicable, so we do not know whether SPDE \eqref{eq:sfpe_joint} has a unique strong solution for any $g \in \cal{G}_{\bb{F}}^2$. Hence, the cost functional $J_{\text{cl}}$ is not necessarily defined for all $g \in \cal{G}_{\bb{F}}^2$ and, consequently, the SPDE formulation is not well-posed. Thus, if we weaken the assumptions, we have to revert back to the weak formulation. In that case we can establish the following theorem.

\begin{theorem} \label{thm:equivalence}
Let Assumptions \ref{ass:mfl} and \ref{ass:convexity} be satisfied and suppose that $\L(\xi) \in L^2(\R)$ and $\zeta = 0$. Then, changing the weak setup if necessary, there exists an optimal semiclosed-loop control that does not depend on the intensity.
\end{theorem}

The proof of the theorem is based on tightness arguments, so changing the weak setup may be necessary to accommodate the subsequential limits of the tight sequence.

Note that if $g \in \cal{G}_{\bb{F}}^1$ and $\mu = (\mu_t)_{0 \leq t \leq T}$ is a (possibly weak) solution to SPDE \eqref{eq:sfpe_joint} with input $g$, then $\nu = (\nu_t)_{0 \leq t \leq T}$ defined by $\nu_t = \cal{S}(\mu_t)$ satisfies the stochastic Fokker--Planck equation
\begin{equation} \label{eq:sfpe}
    \d \langle \nu_t, \varphi\rangle = \bigl\langle \nu_t, \L\varphi(t, \cdot, \cdot, \nu_t, g_t)\bigr\rangle\, \d t + \langle \nu_t, \sigma_0(t) \partial_x \varphi \rangle \, \d W_t
\end{equation}
for $\varphi \in C_c^2(\R)$ with initial condition $\nu_0 = \L(\xi)$. The generator $\L$ is defined by
\begin{equation} \label{eq:generator_sfpe}
    \L \varphi(t, x, v, h) = b(t, x, v, h)\partial_x \varphi(x) - \lambda(t, x) \varphi(x) + a_t(x) \partial_x^2 \varphi(x)
\end{equation}
for $\varphi \in C_b^2(\R)$ and $(t, x, v, h) \in [0, T] \times \R \times \cal{M}^2_{\leq 1}(\R) \times G$. 
Hence, we can rewrite the cost
\begin{equation*}
    \ev\biggl[\int_0^T \bigl\langle \mu_t, \tilde{f}(t, \cdot, \nu_t, g_t) \bigr\rangle \, \d t + \psi(\nu_T)\biggr] = \ev\biggl[\int_0^T \bigl\langle \nu_t, f(t, \cdot, \nu_t, g_t) \bigr\rangle \, \d t + \psi(\nu_T)\biggr].
\end{equation*}

\section{Semilinear BSPDEs with C\`adl\`ag Noise} \label{sec:bspde}

\subsection{C\`adl\`ag Martingales with Values in separable Hilbert Spaces} \label{sec:hilbert_space_martingale}

We recall a few notions from the theory of martingales with values in infinite-dimensional Hilbert spaces. For more information we refer to \cite[Section 26]{metivier_semimartingales_1982}. We fix a separable Hilbert space $\cal{H}$ with inner product $\langle \cdot, \cdot \rangle_{\cal{H}}$ and induced norm $\lVert \cdot \rVert_{\cal{H}}$. In Subsection \ref{sec:main_results_bspde}, we introduced the space $\cal{M}_{\bb{F}}^2([0, T]; \cal{H})$ of $\cal{H}$-valued $\bb{F}$-martingales $m$ with a.s.\@ c\`adl\`ag trajectories for which $\ev \lVert m_T\rVert_{\cal{H}}^2$. We denote the \textit{quadratic variation} of a martingale $m \in \cal{M}_{\bb{G}}^2([0, T]; \cal{H})$ by $[m]$ and its \textit{tensor quadratic variation} by $[[m]]$. The tensor quadratic variation is a process with values in the trace class operators from $\cal{H}$ to itself and is symmetric as such. Moreover, it holds that $[m]_t = \trace [[m]]_t$. We can define the \textit{quadratic covariation} $[m^1, m^2]$ and \textit{tensor quadratic covariation} $[[m^1, m^2]]$ of two martingales $m^1$, $m^2 \in \cal{M}_{\bb{G}}^2([0, T]; \cal{H})$ through polarisation. Note that the two martingales $m^1$, $m^2$ are very strongly orthogonal as defined in Section \ref{sec:main_results_bspde} if and only if $[[m^1, m^2]] = 0$.

For $m \in \cal{M}_{\bb{G}}^2([0, T]; \cal{H})$ let $\P_m^2$ denote the space of $\bb{F}$-progressively measurable $\cal{H}$-valued processes such that $\int_0^T \lVert h_t\rVert_{\cal{H}}^2 \, \d [m]_t < \infty$. For $h \in \P_m^2$ we let $\int_0^t \langle h_s, \d m_s\rangle_{\cal{H}}$ be the stochastic integral of $h$ against $m$ over the interval $(0, t]$, i.e.\@ we exclude the left endpoint. The indefinite stochastic integral $\int_0^{\cdot} \langle h_s, \d m_s\rangle_{\cal{H}}$ is a real-valued local martingale and its quadratic variation is given by $\int_0^{\cdot} \langle h_s, h_s \cdot \d [[m]]_s\rangle_{\cal{H}}$. Here again the integral excludes the left endpoint, and for a bounded operator $A \define \cal{H} \to \cal{H}$ and an element $h \in \cal{H}$ we write $h \cdot A$ for $Ah$.
Since the operator norm of a symmetric trace class operator $A$ is bounded from above by the trace, it holds that $\langle h, h \cdot A\rangle_{\cal{H}} \leq \lVert h\rVert_{\cal{H}} \trace A$. We can apply this inequality to simple integrands $h \in \P_m^2$ of the form $h = \sum_{i = 0}^{n - 1} h_i \bf{1}_{(t_i, t_{i + 1}]}$ for $h_i$ bounded and $\bb{F}_{t_i}$-measurable and $0 = t_0 < t_1 < \dots < t_n = T$ to see $\int_0^t \langle h_s, h_s \cdot \d [[m]]_s\rangle_{\cal{H}} \leq \int_0^t \lVert h_s\rVert_{\cal{H}}^2 \, \d [m]_s$. Then, we extend this inequality to arbitrary $h \in \P_m^2$ by density of the simple integrands. This implies the following Burkholder-Davis-Gundy type inequality.

\begin{proposition} \label{prop:bdg}
Let $m \in \cal{M}_{\bb{G}}^2([0, T]; H)$ and $h \in \cal{P}^2_m$. Then for all $p \geq 1$ and any stopping time $\tau$ with values in $[0, T]$ we have
\begin{equation}
    \ev \sup_{0 \leq t \leq \tau} \biggl\lvert\int_0^t \langle h_s, \d m_s\rangle_{\cal{H}}\biggr\rvert^p \leq C_p \ev\biggl(\int_0^{\tau} \lVert h_t\rVert_{\cal{H}}^2 \, \d [m]_t\biggr)^{p/2}
\end{equation}
for a constant $C_p > 0$ independent of $m$, $h$, and $\tau$.
\end{proposition}

\begin{proof}
By the classical Burkholder-Davis-Gundy inequality we have
\begin{equation*}
    \ev \sup_{0 \leq t \leq \tau} \biggl\lvert\int_0^t \langle h_s, \d m_s\rangle_{\cal{H}}\biggr\rvert^p \leq C_p \ev\biggl(\int_0^{\tau} \langle h_t, h_t \cdot \d [[m]]_t\rangle_{\cal{H}}\biggr)^{p/2}.
\end{equation*}
Then the bound $\int_0^t \langle h_s, h_s \cdot \d [[m]]_s\rangle_{\cal{H}} \leq \int_0^t \lVert h_s\rVert_{\cal{H}}^2 \, \d [m]_s$ allows us to conclude.
\end{proof}

\subsection{It\^o's formula for C\`adl\`ag Hilbert space-valued Semimartingales}

Our next goal is to derive a generalisation of It\^o's formula which is suitable for our analysis of c\`adl\`ag semimartingales with values in Hilbert spaces. For a separable Hilbert space $\cal{H}$, we let $B_2(\cal{H})$ be the space of Hilbert-Schmidt operators from $\cal{H}$ to itself.

\begin{theorem} \label{thm:ito}
For $i = 1$,~\ldots, $k$, let $u^i \in L_{\bb{F}}^2([0, T]; H^1(\R^d))$, $\zeta^i \in L_{\bb{F}}^2([0, T]; H^{-1}(\R^d))$, and $m^i \in \cal{M}_{\bb{F}}^2([0, T]; L^2(\R^d))$ such that
\begin{equation*}
    \langle u^i_t, \varphi\rangle = \int_0^t \zeta^i_s(\varphi) \, \d s + \langle m^i_t, \varphi\rangle
\end{equation*}
for all $\varphi \in H^1(\R^d)$. Then $u^1$,~\ldots, $u^k$ lie in $D_{\bb{F}}^2([0, T]; L^2(\R^d))$ (or $C_{\bb{F}}^2([0, T]; L^2(\R))$ if $m^1$,~\ldots, $m^k$, respectively, have continuous trajectories).

Next, let $F \define L^2(\R^d; \R^k) \to \R$ be a twice continuously Fr\'echet differentiable function and assume that for all $1 \leq i, j \leq k$ it holds that
\begin{enumerate}[noitemsep, label = (\roman*)]
    \item \label{it:bounded_on_h_1} for each $h \in H^1(\R^d)$ the derivative $D_i F(h)$ is an element of $H^1(\R^d)$ and the map $H^1(\R^d) \to H^1(\R^d)$, $h \mapsto D_i F(h)$ is uniformly continuous on bounded sets of $H^1(\R^d)$;
    \item \label{it:hs_bounded} for each $h \in L^2(\R^d)$ the second derivative $D_{ij}^2 F(h)$ is a Hilbert-Schmidt operator and the map $L^2(\R^d) \to B_2(L^2(\R^d))$, $h \mapsto D_{ij}^2 F(h)$ is uniformly continuous on bounded sets of $L^2(\R^d)$.
\end{enumerate}
Then for any $t \in [0, T]$ it holds that
\begin{align} \label{eq:ito_formula}
\begin{split}
    F(u_t) &= F(m_0) + \sum_{i = 1}^k \biggl(\int_0^t \zeta^i_s(D_i F(u_s)) \, \d s + \int_0^t \langle D_i F(u_{s-}), \, \d m^i_s\rangle \biggr) \\
    &\ \ \ + \sum_{i, j = 1}^k \frac{1}{2}\int_0^t \trace\bigl(D^2_{ij}F(u_{s-}) 
    \, \d [[m^i, m^j]]^c_s\bigr) \\
    &\ \ \ + \sum_{0 < s \leq t} \biggl(F(u_{s-} + \Delta m_s) - F(u_{s-}) - \sum_{i = 1}^k \langle D_i F(u_{s-}), \Delta m^i_s\rangle\biggr).
\end{split}
\end{align}
\end{theorem}

Here $\Delta m^{\epsilon}_s$ denotes the jump of $m^{\epsilon}$ at time $s$, not the Laplacian.

\begin{proof}[Proof of Theorem \ref{thm:ito}]
We proceed in two steps.

\textit{Step 1}: First, we assume that $\zeta^i$ is an element of $L_{\bb{F}}^2([0, T]; L^2(\R^d))$ for $i = 1$,~\ldots, $k$, so that $u^i_t = \int_0^t \zeta^i_s \, \d s + m^i_t$, where $\int_0^t \zeta^i_s \, \d s$ is understood as a Bochner integral. Then $u^1$,~\ldots, $u^k$ are $L^2(\R^d)$-valued c\`adl\`ag semimartingales, so we simply apply \cite[Theorem 27.2]{metivier_semimartingales_1982}, which gives precisely \eqref{eq:ito_formula}.

\textit{Step 2}: Now, we let $\zeta^i \in L_{\bb{F}}^2([0, T]; H^{-1}(\R^d))$. We will reduce to the previous case by mollification. Let $\eta \define \R^d \to \R$ be the standard mollifier defined by $\eta(x) = c_d e^{-1/(1 - \lvert x\rvert^2)}$ if $x$ is in the interior of the centered ball of radius $1$ in $\R^d$ and zero otherwise, and set $\eta_{\epsilon}(x) = \frac{1}{\epsilon^d} \eta(x/\epsilon)$. Here $c_d > 0$ is a normalisation constant which ensures that $\int_{\R^d} \eta(x) \, \d x = 1$. For a function $\varphi \in L^2(\R^d)$ we let $(\eta_{\epsilon} \ast \varphi)(x) = \int_{\R^d} \eta_{\epsilon}(x - y) \varphi(x) \, \d y$. Then we set $u^{i, \epsilon}_t = \eta_{\epsilon} \ast u^i_t$, and $m^{i, \epsilon}_t = \eta_{\epsilon} \ast m^i_t$, and define $\zeta^{i, \epsilon}$ via $\zeta^{i, \epsilon}_t(\varphi) = \zeta^i_t(\eta_{\epsilon} \ast \varphi)$. With this we obtain
\begin{equation} \label{eq:mollified_differential}
    \d \langle u^{i, \epsilon}_t, \varphi\rangle = \d \langle u^i_t, \eta_{\epsilon} \ast \varphi\rangle = \zeta^i_t(\eta_{\epsilon} \ast \varphi) \, \d t + \d \langle m^i_t, \eta_{\epsilon} \ast \varphi\rangle = \zeta^{i, \epsilon}_t(\varphi) \, \d t + \d \langle m^{i, \epsilon}_t, \varphi\rangle
\end{equation}
for $\varphi \in L^2(\R^d)$. Note that for any $\xi \in H^{-1}(\R^d)$ it holds that $\lvert \xi(\eta_{\epsilon} \ast \varphi) \rvert \leq \lVert \xi\rVert_{H^{-1}} \lVert \eta_{\epsilon} \ast \varphi \rVert_{H^1} \leq C_{\epsilon, d} \lVert \partial_x \eta_{\epsilon}\rVert_{\infty} \lVert \xi\rVert_{H^{-1}} \lVert \varphi\rVert_{L^2}$ for some constant $C_{\epsilon, d} > 0$. Hence, the mollification of $\xi$ is a bounded linear functional on $L^2(\R^d)$, so we may identify it with an element of $L^2(\R^d)$. Consequently, we may view $\zeta^i$ as a process in $L_{\bb{F}}^2([0, T]; L^2(\R^d))$. Thus, by Step 1 it holds that
\begin{align} \label{eq:ito_formula_epsilon}
\begin{split}
    F(u^{\epsilon}_t) &= F(m^{\epsilon}_0) + \sum_{i = 1}^k \biggl(\int_0^t \zeta^{i, \epsilon}_s(D_i F(u^{\epsilon}_s)) \, \d s + \int_0^t \langle D_i F(u^{\epsilon}_{s-}), \, \d m^{i, \epsilon}_s\rangle \biggr) \\
    &\ \ \ + \sum_{i, j = 1}^k \frac{1}{2}\int_0^t \trace\bigl(D^2_{ij}F(u^{\epsilon}_{s-}) 
    \, \d [[m^{i, \epsilon}, m^{j, \epsilon}]]^c_s\bigr) \\
    &\ \ \ + \sum_{0 < s \leq t} \biggl(F(u^{\epsilon}_{s-} + \Delta m^{\epsilon}_s) - F(u^{\epsilon}_{s-}) - \sum_{i = 1}^k \langle D_i F(u^{\epsilon}_{s-}), \Delta m^{i, \epsilon}_s\rangle\biggr).
\end{split}
\end{align}
Our goal is to pass to the limit $\epsilon \to 0$ in Equation \eqref{eq:ito_formula_epsilon}. However, while the sequences $(u^{\epsilon})_{\epsilon > 0}$, $(\zeta^{\epsilon})_{\epsilon > 0}$, and $(m^{\epsilon})_{\epsilon > 0}$ converge in their respective Hilbert spaces $L_{\bb{F}}^2([0, T]; H^1(\R^d))$, $L_{\bb{F}}^2([0, T]; H^{-1}(\R^d))$, and $\cal{M}_{\bb{F}}^2([0, T]; L^2(\R^d))$, we do not even know whether $(u^{\epsilon})_{\epsilon > 0}$ also converges as a sequence in $D_{\bb{F}}^2([0, T]; L^2(\R^d))$. But this is certainly necessary to pass to the limit in the stochastic integral term $\int_0^t \langle D_i F(u^{\epsilon}_{s-}), \, \d m^{i, \epsilon}_s\rangle$ and to make sense of its limit $\int_0^t \langle D_iF(u_{s-}), \d m^i_s\rangle$. 

To deduce the convergence in $D_{\bb{F}}^2([0, T]; L^2(\R^d))$, we first apply \eqref{eq:ito_formula_epsilon} to $u^{\epsilon, \delta} = u^{\epsilon} - u^{\delta}$ for the specific choice $F(\varphi) = \sum_{i = 1}^k\lVert \varphi^i\Vert_{L^2}^2$. This implies that for any $\epsilon$, $\delta > 0$ that
\begin{equation} \label{eq:l_2_app_ito}
    \sum_{i = 1}^k \lVert u^{i, \epsilon, \delta}_t \Vert_{L^2}^2 = \sum_{i = 1}^k \biggl(\lVert m^{i, \epsilon, \delta}_0 \Vert_{L^2}^2 + \int_0^t 2\zeta^{i, \epsilon, \delta}_s(u^{i, \epsilon, \delta}_s) \, \d s + \int_0^t 2\bigl\langle u^{i, \epsilon, \delta}_{s-}, \, \d m^{i, \epsilon, \delta}_s\bigr\rangle + [m^{i, \epsilon, \delta}]_t\biggr),
\end{equation}
where $\zeta^{i, \epsilon, \delta}$ and $m^{i, \epsilon, \delta}$ are defined in the obvious way. We take the supremum over $t \in [0, T]$ and then expectation on both sides. Further, we estimate with Proposition \ref{prop:bdg},
\begin{align*}
    \ev \sup_{t \in [0, T]} 2\biggl\lvert \int_0^t \bigl\langle u^{i, \epsilon, \delta}_{s-}, \, \d m^{i, \epsilon, \delta}_s\bigr\rangle\biggr\rvert &\leq 2C_1 \ev\biggl(\int_0^T \lVert u^{i, \epsilon, \delta}_{t-}\rVert_{L^2}^2 \, \d [m^{i, \epsilon, \delta}]_t\biggr)^{1/2} \\
    &\leq 2 C_1 \ev\biggl(\sup_{0 < t \leq T} \lVert u^{i, \epsilon, \delta}_{t-}\rVert_{L^2}^2 [m^{i, \epsilon, \delta}]_T\biggr)^{1/2} \\
    &\leq \frac{1}{2} \ev \sup_{0 \leq t \leq T} \lVert u^{i, \epsilon, \delta}_t\rVert_{L^2}^2 + 8 C_1^2 \ev [m^{i, \epsilon, \delta}]_T.
\end{align*}
Inserting this into \eqref{eq:l_2_app_ito} and rearranging gives
\begin{align*}
    \sum_{i = 1}^k \frac{1}{2} \ev \sup_{0 \leq t \leq T} \lVert u^{i, \epsilon, \delta}_t\rVert_{L^2}^2 \leq \sum_{i = 1}^k \biggl(&\lVert m^{i, \epsilon, \delta}_0 \Vert_{L^2}^2 + \ev\int_0^T 2 \lVert \zeta^{i, \epsilon, \delta}_t\rVert_{H^{-1}} \lVert u^{i, \epsilon, \delta}_t\rVert_{H^1} \, \d t \\
    &+ (2C_1 + 1) \ev[m^{i, \epsilon, \delta}]_T\biggr).
\end{align*}
The right-hand side clearly tends to zero as $\epsilon$, $\delta \to 0$, which means that $(u^{i, \epsilon})_{\epsilon > 0}$, for $i = 1$,~\ldots, $k$, is a Cauchy sequence in the space $D_{\bb{F}}^2([0, T]; L^2(\R^d))$. Hence, it has a unique limit, which must coincide with $u$, since $u^{\epsilon} \to u$ in $L_{\bb{F}}^2([0, T]; L^2(\R^d))$. Consequently, we have $\ev \sup_{0 \leq t \leq T} \lVert u^{\epsilon}_t - u_t\rVert_{L^2}^2 \to 0$ and $u$ lies in $D_{\bb{F}}^2([0, T]; L^2(\R^d))$.

Let us return to Equation \eqref{eq:ito_formula_epsilon}. We can clearly pass to the limit in probability in the terms $F(u^{\epsilon}_t)$, $F(m^{\epsilon}_0)$, and $\int_0^t \zeta^{i, \epsilon}_s(D_i F(u^{\epsilon}_s)) \, \d s$ resulting in the expressions $F(u_t)$, $F(m_0)$, and $\int_0^t \zeta^i_s(D_iF(u_s)) \, \d s$. Here we make use of Assumption \ref{it:bounded_on_h_1}. Next, by Lemma \ref{lem:stoch_int_conv}, the stochastic integrals and the quadratic variation terms converge to $\int_0^t \langle D_i F(u_{s-}), \d m^i_s\rangle$ and $\frac{1}{2}\int_0^t \trace\bigl(D_{ij}^2 F(u_{s-}) \, \d [[m^i, m^j]]^c_s\bigr)$, respectively, in probability as $\epsilon \to 0$. This leaves the jump terms. First note that the jumps of $m^{\epsilon}$ and $m$ align perfectly and, moreover, it holds that $\lVert \Delta m^{\epsilon}_s\rVert_{L^2} \leq \lVert \Delta m_s\rVert_{L^2}$. We can use Taylor's theorem to bound $F(u^{\epsilon}_{s-} + \Delta m^{\epsilon}_s) - F(u^{\epsilon}_{s-}) - \sum_{i = 1}^k \langle D_i F(u^{\epsilon}_{s-}), \Delta m^{i, \epsilon}_s\rangle$ by
\begin{align*}
    \sum_{i, j = 1}^k \bigl\langle \Delta m^{i, \epsilon}_s, D_{ij}^2 F(u^{\epsilon}_{s-}) m^{j, \epsilon}_s\bigr\rangle &\leq \sum_{i, j = 1}^k \frac{\lVert D_{ij}^2F(u^{\epsilon}_{s-})\rVert_{B_2}}{2} \bigl(\lVert \Delta m^{i, \epsilon}_s\rVert_{L^2}^2 + \lVert \Delta m^{j, \epsilon}_s\rVert_{L^2}^2\bigr) \\
    &\leq \sum_{i, j = 1}^k \frac{\lVert D_{ij}^2F(u^{\epsilon}_{s-})\rVert_{B_2}}{2} \bigl(\lVert \Delta m^i_s\rVert_{L^2}^2 + \lVert \Delta m^j_s\rVert_{L^2}^2\bigr),
\end{align*}
where $\lVert \cdot \rVert_{B_2}$ denotes the Hilbert-Schmidt norm. From above we know $\ev \sup_{0 < s \leq t} \lVert u^{\epsilon}_{s-} - u_{s-}\rVert_{L^2}^2 \to 0$, so since $D_{ij}^2 F$ is bounded on bounded sets of $L^2(\R^d)$ as a Hilbert-Schmidt operator by Assumption \ref{it:hs_bounded}, we deduce that $\lVert D_{ij}^2F(u^{\epsilon}_{s-})\rVert_{B_2}$ is bounded by some random constant $C > 0$ uniformly in $\epsilon$. Hence, we can upper bound $F(u^{\epsilon}_{s-} + \Delta m^{\epsilon}_s) - F(u^{\epsilon}_{s-}) - \sum_{i = 1}^k \langle D_i F(u^{\epsilon}_{s-}), \Delta m^{i, \epsilon}_s\rangle$ by the a.s.\@ summable expression $Ck \sum_{i = 1}^k \lVert \Delta m^i_s\rVert_{L^2}^2$. Thus, we can apply the dominated convergence theorem to pass to the a.s.\@ limit in the sum in the third line of Equation \eqref{eq:ito_formula_epsilon}. Putting all convergences together yields the desired formula \eqref{eq:ito_formula}.
\end{proof}

Note that Theorem \ref{thm:ito} is not applicable to solutions $(u, q, m)$ of BSPDE \eqref{eq:bspde}, since in general we would not expect $u$ to have a weak derivative in $y$. Consequently, we need a generalisation of Theorem \ref{thm:ito} to semimartingales in $L^2([0, T]; H^{1, 0}(D))$, where $D = \R^d \times (\ell, \infty)$ for $\ell \in \R$. For a separable Banach space $\cal{B}$ with norm $\lVert \cdot \rVert_{\cal{B}}$ we let $L_{\bb{F}}^{\infty}([0, T]; \cal{B})$ be the space of $\cal{B}$-valued $\bb{F}$-progressively measurable process $h$ such that $\lVert h_t\rVert_{\cal{B}}$ is $\text{Leb} \otimes \pr$-essentially bounded.

\begin{theorem} \label{thm:ito_2_dim}
For $i = 1$, $2$, let $u^i \in L_{\bb{F}}^2([0, T]; H^{1, 0}(D))$, $\zeta^i \in L_{\bb{F}}^2([0, T]; H^{-1, 0}(D))$, $m^i \in \cal{M}_{\bb{F}}^2([0, T]; L^2(D))$, and $h \in L_{\bb{F}}^{\infty}([0, T]; L^{\infty}(\R^d))$, where $h$ is nonnegative,
such that 
\begin{enumerate}[noitemsep, label = (\roman*)]
    \item for $\leb \otimes \pr$-a.e.\@ $(t, \omega) \in [0, T] \times \Omega$ we have $\supp u^1_t \subset \R^d \times [\ell + \theta, \infty)$ for some $\theta > 0$; 
    \item it holds that
    \begin{equation} \label{eq:differential_2_d}
        \langle u^i_t, \varphi\rangle = \int_0^t \bigl(\zeta^i_s(\varphi) + \langle u^i_s, h_s\partial_y \varphi\rangle\bigr) \, \d s + \langle m^i_t, \varphi\rangle
    \end{equation}
    for all $\varphi \in H^1(D)$ if $i = 1$ and for all $\varphi \in H^1_0(D)$ if $i = 2$.
\end{enumerate}
Then $u^1$ and $u^2$ are elements of $D_{\bb{F}}^2([0, T]; L^2(D))$ (or $C_{\bb{F}}^2([0, T]; L^2(D))$ if $m^1$ and $m^2$, respectively, have continuous trajectories) and for any $t \in [0, T]$ it holds that
\begin{align} \label{eq:ito_formula_2_dim}
\begin{split}
    \langle u^i_t, u^j_t\rangle &= \langle m^i_0, m^j_0\rangle + \int_0^t \zeta^i_s(u^2_s) + \zeta^2_s(u^j_s) \, \d s + \int_0^t \bigl(\langle u^i_{s-}, \d m^i_s\rangle + \langle u^j_{s-}, \d m^j_s\rangle\bigr) \\
    &\ \ \ + [m^i, m^j]_t + 2\bf{1}_{i = j = 2} \xi_t,
\end{split}
\end{align}
for some nondecreasing real-valued process $\xi = (\xi_t)_{0 \leq t \leq T}$ with $\ev \xi_t < \infty$ started from zero.
\end{theorem}

Note that the coefficient $h$ in the statement of Theorem \ref{thm:ito_2_dim} is a random function on $\R^d$ and not $D$, so it is independent of the last variable. Hence, the expression $\langle u^i_s, h_s\partial_y \varphi\rangle$ stands for the integral $\int_{D} u^i_s(x, y) h_s(x) \partial_y \varphi(x, y) \, \d x \d y$.

\begin{proof}
We proceed in two steps.

\textit{Step 1}: First, we assume that for $i = 2$ instead of \eqref{eq:differential_2_d}, it holds that
\begin{equation*}
    \langle u^2_t, \varphi\rangle = \int_0^t \bigl(\zeta^2_s(\varphi) + \langle u^2_s, h_s  \partial_y\varphi\rangle + \langle \xi_s, \varphi\bigr) \, \d s + \langle m^2_t, \varphi\rangle
\end{equation*}
for all $\varphi \in H^1(D)$, where $\xi \in L_{\bb{F}}^2([0, T]; L^2(D))$ satisfies $\supp \xi_t \subset \R \times [\ell, \ell + \theta/2]$ for $\leb \otimes \pr$-a.e.\@ $(t, \omega) \in [0, T] \times \Omega$. Let $\eta \define \R \to \R$ be the standard mollifier in one dimension defined by $\eta(x) = c e^{-1/(1 - x^2)}$ if $x \in (-1, 1)$ and zero otherwise, and set $\eta_{\epsilon}(x) = \frac{1}{\epsilon} \eta(x/\epsilon)$. Here $c > 0$ is a normalisation constant which ensures that $\int_{\R} \eta(x) \, \d x = 1$. For a function $\varphi \in L^2(D)$ we let $(\eta_{\epsilon} \ast_2 \varphi)(x, y) = \int_{\R} \eta_{\epsilon}(x, y - z) \varphi(x, z) \, \d z$ for $(x, y) \in \R^d \times \R$ be the mollification in the second component, where $\varphi \define D \to \R$ is extended to a function on all of $\R^d \times \R$ by zero. Then, we set $u^{i, \epsilon}_t = \eta_{\epsilon} \ast_2 u^i_t$, $m^{i, \epsilon}_t = \eta_{\epsilon} \ast_2 m^i_t$, $\xi^{\epsilon}_t = \eta_{\epsilon} \ast_2 \xi_t$, and define $\zeta^{i, \epsilon}$ via $\zeta^{i, \epsilon}_t(\varphi) = \zeta^i_t(\eta_{\epsilon} \ast_2 \varphi) + \bigl\langle u^i_t, h_t \partial_y (\eta_{\epsilon} \ast_2 \varphi)\bigr\rangle + \delta_{i2}\langle \xi^{\epsilon}_t, \varphi\rangle$. With this we obtain
\begin{align*}
    \d \langle u^{i, \epsilon}_t, \varphi\rangle &= \d \langle u^i_t, \eta_{\epsilon} \ast_2 \varphi\rangle \\
    &= \Bigl(\zeta^i_t(\eta_{\epsilon} \ast_2 \varphi) + \bigl\langle u^i_t, h_t \partial_y (\eta_{\epsilon} \ast_2 \varphi)\bigr\rangle + \delta_{i2}\langle \xi, \eta_{\epsilon} \ast_2 \varphi\rangle\Bigr) \, \d t + \d \langle m^i_t, \eta_{\epsilon} \ast_2 \varphi\rangle \\
    &= \zeta^{i, \epsilon}_t(\varphi) \, \d t + \d \langle m^{i, \epsilon}_t, \varphi\rangle
\end{align*}
for $\varphi \in H^1(D)$. Note that $u^{i, \epsilon} \in L_{\bb{F}}^2([0, T]; H^1(\R^d \times \R)$, $\zeta^{i, \epsilon} \in L_{\bb{F}}^2([0, T]; H^{-1}(\R^d \times \R))$, and $m^{i, \epsilon} \in \cal{M}_{\bb{F}}^2([0, T]; L^2(\R^d \times \R))$, so we can apply Theorem \ref{thm:ito}, whereby $u^{i, \epsilon}$ is an element of $D_{\bb{F}}^2([0, T]; L^2(\R^d \times \R))$. Moreover, applying \eqref{eq:ito_formula} with the choice $F(v) = \lVert v\rVert_{L^2}^2$ to the process $u^{i, \epsilon, \delta} = u^{i, \epsilon} - u^{i, \delta}$ yields
\begin{align*}
    \lVert u^{i, \epsilon, \delta}_t\rVert_{L^2}^2 = \lVert m^{i, \epsilon, \delta}_0\rVert_{L^2}^2 + \int_0^t 2 \zeta^{i, \epsilon, \delta}(u^{i, \epsilon, \delta}_s) \, \d s + \int_0^t 2 \bigl\langle u^{i, \epsilon, \delta}_{s-}, \, \d m^{i, \epsilon, \delta}_s\bigr\rangle + [m^{i, \epsilon, \delta}]_t.
\end{align*}
Now, it holds that 
\begin{align} \label{eq:second_var_vanish}
\begin{split}
    \bigl\langle u^i_s, h_s  \bigl(\partial_y (\eta_{\epsilon} - \eta_{\delta}) \ast_2 u^{i, \epsilon, \delta}_s\bigr)\bigr\rangle &= - \bigl\langle \partial_y (\eta_{\epsilon} - \eta_{\delta}) \ast_2 u^i_s, h_s u^{i, \epsilon, \delta}_s\bigr\rangle = - \bigl\langle \partial_y u^{i, \epsilon, \delta}_s, h_s u^{i, \epsilon, \delta}_s\bigr\rangle = 0,
\end{split}
\end{align}
where we used in the first equality that $\partial_y(\eta_{\epsilon} - \eta_{\delta})$ is an odd function and $h_s$ is independent of the last variable $y$.
Hence, it holds that
\begin{align*}
    \zeta^{i, \epsilon, \delta}(u^{i, \epsilon, \delta}_s) &= \zeta^i_s\bigl((\eta_{\epsilon} - \eta_{\delta}) \ast_2 u^{i, \epsilon, \delta}_s\bigr) + \bigl\langle u^i_s, h_s \bigl(\partial_y (\eta_{\epsilon} - \eta_{\delta}) \ast_2 u^{i, \epsilon, \delta}_s\bigr)\bigr\rangle  \delta_{i2}\langle \xi^{\epsilon, \delta}_s, u^{2, \epsilon, \delta}_s\rangle \\
    &= \zeta^i_s\bigl((\eta_{\epsilon} - \eta_{\delta}) \ast_2 u^{i, \epsilon, \delta}_s\bigr) + \delta_{i2}\langle \xi^{\epsilon, \delta}_s, u^{2, \epsilon, \delta}_s\rangle,
\end{align*}
so we can estimate $\bigl\lvert \zeta^{i, \epsilon, \delta}(u^{i, \epsilon, \delta}_s)\bigr\rvert \leq 2\lVert \zeta^i_s\rVert_{H^{-1, 0}} \lVert u^{i, \epsilon, \delta}_s\rVert_{H^{1, 0}} + \delta_{i2} \lVert \xi^{\epsilon, \delta}_s\rVert_{L^2} \lVert u^{2, \epsilon, \delta}_s\rVert_{L^2}$. The right-hand side tends to zero as $\epsilon$, $\delta \to 0$. Hence, as in the proof of Theorem \ref{thm:ito}, we conclude that $\ev \sup_{0 \leq t \leq T} \lVert u^{i, \epsilon, \delta}_t\rVert_{L^2}^2 \to 0$ as $\epsilon$, $\delta \to 0$. Consequently, $(u^{i, \epsilon})_{\epsilon}$ is a Cauchy sequence in the space $D_{\bb{F}}^2([0, T]; L^2(\R^d \times \R))$ and its limit in $D_{\bb{F}}^2([0, T]; L^2(\R^d \times \R))$ is $u^i$.

Next, let us deduce Equation \eqref{eq:ito_formula_2_dim}. For that we apply It\^o's formula \eqref{eq:ito_formula} with the choice $F(v^1, v^2) = \langle v^1, v^2\rangle$ to the process $(u^{i, \epsilon}, u^{j, \epsilon})$ for $\epsilon < \theta/4$, whereby
\begin{align} \label{eq:ito_formula_2_dim_eps}
\begin{split}
    \langle u^{i, \epsilon}_t, u^{j, \epsilon}_t\rangle &= \langle m^{i, \epsilon}_0, m^{j, \epsilon}_0\rangle + \int_0^t \zeta^{i, \epsilon}_s(u^{j, \epsilon}_s) + \zeta^{j, \epsilon}_s(u^{i, \epsilon}_s) \, \d s \\
    &\ \ \ + \int_0^t \bigl(\langle u^{j, \epsilon}_{s-}, \d m^{i, \epsilon}_s\rangle + \langle u^{i, \epsilon}_{s-}, \d m^{j, \epsilon}_s\rangle\bigr) \, \d s + [m^{i, \epsilon}, m^{j, \epsilon}]_t.
\end{split}
\end{align}
Similarly to Equation \eqref{eq:second_var_vanish}, we compute
\begin{align*}
    \bigl\langle u^i_s, h_s (\partial_y \eta_{\epsilon} \ast_2 u^{j, \epsilon}_s)\bigr\rangle + \bigl\langle u^j_s, h_s (\partial_y \eta_{\epsilon} \ast_2 u^{i, \epsilon}_s)\bigr\rangle &= \bigl\langle \eta_{\epsilon} \ast_2 u^i_s, h_s \partial_y u^{j, \epsilon}_s)\bigr\rangle - \bigl\langle \partial_y \eta_{\epsilon} \ast_2 u^j_s, h_s  u^{i, \epsilon}_s\bigr\rangle \\
    &= \bigl\langle u^{i, \epsilon}_s, h_s \partial_y u^{j, \epsilon}_s\bigr\rangle - \bigl\langle \partial_y u^{j, \epsilon}_s, h_s  u^{i, \epsilon}_s\bigr\rangle \\
    &= 0,
\end{align*}
where we used in the first equality that $\eta_{\epsilon}$ is even, $\partial_y \eta_{\epsilon}$ is odd, and that $h_s$ does not depend on the last variable $y$. Further, if $i$ or $j$ are different from $2$, then since $\xi^{\epsilon}_t$ is supported in $\R \times [\ell, \ell + 3\theta/4)$ and $u^{1, \epsilon}_t$ is supported in $\R \times (\ell + 3\theta/4, \infty)$, we have $\delta_{2i} \langle \xi^{\epsilon}_s, u^{j, \epsilon}_s\rangle + \delta_{2j} \langle \xi^{\epsilon}_s, u^{i, \epsilon}_s\rangle = 0$. Consequently, we get that 
\begin{equation*}
    \zeta^{i, \epsilon}_s(u^{j, \epsilon}_s) + \zeta^{j, \epsilon}_s(u^{i, \epsilon}_s) = \zeta^i_s(\eta_{\epsilon} \ast_2 u^{j, \epsilon}_s) + \zeta^j_s(\eta_{\epsilon} \ast_2 u^{i, \epsilon}_s) + 2\bf{1}_{i = j = 1} \langle \xi^{\epsilon}_s, u^{2, \epsilon}_s\rangle.
\end{equation*}
This together with Lemma \ref{lem:stoch_int_conv} allows us to pass to the limit in probability in Equation \eqref{eq:ito_formula_2_dim_eps} as $\epsilon \to 0$. As a result, we obtain \eqref{eq:ito_formula_2_dim} with $\xi_t$ replaced by $\int_0^t 2\langle \xi_s, u^2_s\rangle \, \d s$. Note that in contrast to \eqref{eq:ito_formula_2_dim} this process is not necessarily nondecreasing.

\textit{Step 2}: Let us define $\kappa_{\epsilon}$ by $\kappa_{\epsilon}(y) = \int_{-\infty}^x \eta_{\epsilon}(z - \ell - \epsilon) \, \d z$, so that $\kappa_{\epsilon} \in C^{\infty}(\R)$ and $\kappa_{\epsilon}(y) = 0$ for $y \in (-\infty, \ell]$, $\kappa_{\epsilon}(y) \in [0, 1]$ for $y \in [\ell, \ell + 2\epsilon)$, and $\kappa_{\epsilon}(y) = 1$ for $y \in [\ell + 2\epsilon, \infty)$. We set $u^{2, \epsilon}_t(x, y) = \kappa_{\epsilon}(y) u^2_t(x, y)$, $m^{2, \epsilon}_t(x, y) = \kappa_{\epsilon}(y) m^2_t(x, y)$, and define $\zeta^{2, \epsilon}_t$ via $\zeta^{2, \epsilon}_t(\varphi) = \zeta^2_t(\varphi_{\epsilon})$, where $\varphi_{\epsilon}(x, y) = \kappa_{\epsilon}(y) \varphi(x, y)$ for $\varphi \in H^1(D)$. This implies that
\begin{equation*}
    \langle u^{2, \epsilon}_t, \varphi\rangle = \int_0^t \bigl(\zeta^{2, \epsilon}_s(\varphi) + \langle u^{2, \epsilon}_s, h_s \partial_y \varphi\rangle + \langle u^2_s, h_s \partial_y \kappa_{\epsilon} \varphi\rangle\bigr) \, \d s + \langle m^{2, \epsilon}_s, \varphi\rangle
\end{equation*}
for $\varphi \in H^1(D)$. Thus, $u^{2, \epsilon}$ satisfies the assumptions of Step 1 with $\xi_t = u^2_t h_t \partial_y \kappa_{\epsilon}$, so that $u^{2, \epsilon}$ is an element of $D_{\bb{F}}^2([0, T]; L^2(D)$, and \eqref{eq:ito_formula_2_dim} holds with the process $\xi$ replaced by $\int_0^t \langle u^2_s, h_s \partial_y \kappa_{\epsilon} u^{2, \epsilon}_s\rangle \, \d s$. Hence, we obtain
\begin{align*}
    \lVert u^{2, \epsilon}_T\rVert_{L^2}^2 &= \lVert u^{2, \epsilon}_t\rVert_{L^2}^2 + \int_t^T 2\bigl(\zeta^{2, \epsilon}_s(u^{2, \epsilon}_s) + \langle u^2_s, h_s \partial_y \kappa_{\epsilon} u^{2, \epsilon}_s\rangle\bigr) \, \d s \\
    &\ \ \ + \int_t^T 2 \langle u^{2, \epsilon}_{s-}, \d m^{2, \epsilon}_s\rangle + [m^{2, \epsilon}]_T - [m^{2, \epsilon}]_t.
\end{align*}
Since $\langle u^2_s, h_s \partial_y \kappa_{\epsilon} u^{2, \epsilon}_s\rangle$ and $[m^{2, \epsilon}]_T - [m^{2, \epsilon}]_t$ are nonnegative, we can bound
\begin{equation} \label{eq:bound_sup_1}
    \sup_{0 \leq t \leq T} \lVert u^{2, \epsilon}_t\rVert_{L^2}^2 \leq \lVert u^{2, \epsilon}_T\rVert_{L^2}^2 + \int_0^T 2\zeta^{2, \epsilon}_s(u^{2, \epsilon}_s) \, \d s + \sup_{0 \leq t \leq T} 4\int_0^t \langle u^{2, \epsilon}_{s-}, \d m^{2, \epsilon}_s\rangle.
\end{equation}
We take expectations on both sides, use Proposition \ref{prop:bdg} to estimate
\begin{align*}
    \ev \sup_{0 \leq t \leq T} 4\int_0^t \langle u^{2, \epsilon}_{s-}, \d m^{2, \epsilon}_s\rangle &\leq 4C_1 \ev \biggl(\int_0^T \lVert u^{2, \epsilon}_{t-}\rVert_{L^2}^2 \, \d [m^{2, \epsilon}]_t\biggr)^{1/2} \\
    &\leq \frac{1}{2}\ev\sup_{0 \leq t \leq T} \lVert u^{2, \epsilon}_t\rVert_{L^2}^2 + 8 C_1^2\ev [m^{2, \epsilon}]_T,
\end{align*}
and then rearrange \eqref{eq:bound_sup_1} to find
\begin{equation*}
    \ev\sup_{0 \leq t \leq T} \lVert u^{2, \epsilon}_t\rVert_{L^2}^2 \leq 2 \ev \lVert u^{2, \epsilon}_T\rVert_{L^2}^2 + 4\ev\int_0^T \lvert \zeta^{2, \epsilon}_t(u^{2, \epsilon}_t)\rvert \, \d t + 16C_1^2\ev [m^{2, \epsilon}]_T.
\end{equation*}
The sequence $(\lVert u^{2, \epsilon}_t\rVert_{L^2}^2)_{\epsilon > 0}$ increases to $\lVert u^2_t\rVert_{L^2}^2$ as $\epsilon \to 0$, so the left-hand side above converges by the monotone convergence theorem. For the right-hand side convergence follows from the assumptions on $u^2$, $u^2_T$, $\zeta^2$, and $m^2$, so we find that $\ev\sup_{0 \leq t \leq T} \lVert u^2_t\rVert_{L^2}^2 < \infty$. Thus, 
\begin{equation*}
    \ev\sup_{0 \leq t \leq T} \lVert u^{2, \epsilon}_t - u^2_t\rVert_{L^2}^2 \leq \ev\sup_{0 \leq t \leq T} \lVert (1 - \eta_{\epsilon}) u^2_t\rVert_{L^2}^2 \to 0
\end{equation*}
by the monotone convergence theorem. In particular, $u^2$ is an element of $D_{\bb{F}}^2([0, T]; L^2(D))$. Obtaining Equation \eqref{eq:ito_formula_2_dim} is now a simple limit procedure. Note that $\xi_t$ is the $L^1$-limit of the integral $\int_0^t \langle u^2_s, h_s \partial_y \kappa_{\epsilon} u^{2, \epsilon}_s\rangle \, \d s$, which is nondecreasing as a function of $t$, so that the same is true for $\xi = (\xi_t)_{0 \leq t \leq T}$.
\end{proof}

\subsection{Semilinear BSPDE}

In this subsection, we prove the existence and uniqueness for the semilinear BSPDEs introduced in Section \ref{sec:main_results_bspde}. The proofs for the two domains $D = \R^d \times (\ell, \infty)$ and $D = \R^d$ are rather similar, so we focus on $D = \R^d \times (\ell, \infty)$, in which case a few more technicalities arise. We obtain existence and uniqueness to a solution to BSPDE \eqref{eq:bspde} through a fixed point argument, so we first have to consider linear BSPDEs of the form
\begin{align} \label{eq:bspde_linear}
\begin{split}
    \d u_t(x, y) &= -\Bigl(a_t(x, y) : \nabla_x^2 u_t(x, y) + \alpha_t : \nabla_x q_t(x, y) + \beta_t(x, y) \cdot q_t(x, y)\Bigr) \, \d t \\
    &\ \ \ - \bigl(\lambda_t(x) \partial_y u_t(x, y) + h_t(x, y)\bigr) \, \d t + q_t(x, y) \, \d W_t + \d m_t(x, y)
\end{split}
\end{align}
for $(x, y) \in D$ with terminal condition $u_T(x, y) = \psi(x, y)$ for a random function $\psi \in L^2\bigl(\Omega, \F, \pr; L^2(D)\bigr)$. The process $h$ lies in $L^2_{\bb{F}}([0, T]; L^2(D))$. A solution $(u, q, m) \in \bb{H}(D)$ is understood in the sense of Definition \ref{def:bspde_solution}, so the test functions lie in the space $H^1_0(D)$. We construct a solution to BSPDE \eqref{eq:bspde_linear} through a Galerkin approximation. 

\begin{proposition} \label{prop:bspde_linear}
Let Assumption \ref{ass:bspde} be satisfied and let $h \in L^2_{\bb{F}}([0, T]; L^2(D))$. Then BSPDE \eqref{eq:bspde_linear} has a unique solution $(u, q, m)$ for which $u \in D_{\bb{F}}^2([0, T]; L^2(D))$. Moreover, for any $t \in [0, T]$, it holds that
\begin{align} \label{eq:bspde_linear_estimate}
\begin{split}
    \ev\lVert u_t\rVert_{L^2}^2 + \ev\lVert m_T - m_t\rVert_{L^2}^2 + \ev&\int_t^T \lVert \nabla_x u_s \rVert_{L^2}^2 + \lVert q_s \rVert_{L^2}^2 \, \d s \\
    &\leq C\biggl(\ev\lVert \psi \rVert_{L^2}^2 + \ev\int_t^T  \lvert \langle u_s, h_s\rangle\rvert \, \d s\biggr)
\end{split}
\end{align}
for a constant $C > 0$ independent of $h$ and $t$.
\end{proposition}

\begin{proof}
\textit{Existence}: Our arguments mirror those given in the proof of \cite[Theorem 3.1]{jin_1999_bspde} and \cite[Lemma 4.5]{al_hussein_bspde_2009}. The main difference lies in the fact that BSPDE \eqref{eq:bspde_linear} is not defined on the whole space $\R^d \times \R$ but only on the half-plane $D = \R^d \times (\ell, \infty)$.

Let us fix an orthonormal basis $(e_n)_{n \geq 1}$ of $L^2(D)$ consisting of elements in $C^{\infty}_c(D)$, such that the linear span of $(e_n)_{n \geq 1}$ of $L^2(D)$ is dense in $H^1_0(D)$. Note that it is not possible to choose $(e_n)_{n \geq 1}$ from $C^{\infty}_c(D)$ with a linear span that is dense in $H^1(D)$, since limits in $H^1(D)$ of functions in $C^{\infty}_c(D)$ vanish on $\R \times \{\ell\}$. This explains why the test functions in Definition \ref{def:bspde_solution} are from $H^1_0(D)$. Now, for $n \geq 1$ we consider the BSDE
\begin{align*}
    \d u^{ni}_t &= -\sum_{j = 1}^n \Bigl(\bigl\langle a_t : \nabla_x^2 e_j + \lambda_t \partial_y e_j, e_i\bigr\rangle u^{nj}_t  + \langle \alpha_t : \nabla_x e_j + \beta_t \cdot e_j, e_i\rangle q^{nj}_t\Bigr) \, \d t \\
    &\ \ \ - \langle h_t, e_i\rangle\, \d t  + q^{ni}_t \, \d W_t + \d m^{ni}_t
\end{align*}
with terminal condition $u^{ni}_T = \langle \psi, e_i\rangle$ for $1 \leq i \leq n$. The existence of an $\bb{F}$-progressively measurable solution $u^{ni} \in D_{\bb{F}}^2([0, T])$, $q^{ni} \in L^2_{\bb{F}}([0, T]; \R^d)$, $m^{ni} \in \cal{M}^2_{\bb{F}}([0, T])$ for $1 \leq i \leq n$ is guaranteed by \cite[Example 1.20]{carmona_mfg_ii_2018}. 
Next, we define $u^n \in D_{\bb{F}}^2([0, T]; H^2(D))$ by $u^n_t(x, y) = \sum_{i = 1}^n u^{ni}_t(x, y) e_i(x, y)$ and, similarly, $q^n \in L_{\bb{F}}^2([0, T]; H^2(D; \R^d))$ and $m^n \in \cal{M}_{\bb{F}}^2([0, T]; H^2(D))$. Then let $\pi_n$ denote the orthogonal projection of $L^2(D)$ onto the linear span of $e_1$,~\ldots, $e_n$, which keeps $u^n_t$, $q^n_t$, and $m^n_t$ fixed. This allows us to rewrite the finite-dimensional linear BSDE above in the form
\begin{align} \label{eq:fd_adjoint}
\begin{split}
    \d u^n_t(x, y) &= -\pi_n\Bigl(a_t : \nabla_x^2 u^n_t + \lambda_t\partial_y u^n_t + h_t + \alpha_t : \nabla_x q^n_t + \beta_t \cdot q^n_t\Bigr)(x, y) \, \d t \\
    &\ \ \ + q^n_t(x, y) \, \d W_t + \d m^n_t(x, y)
\end{split}
\end{align}
for $(x, y) \in \R \times (\ell, \infty)$ with terminal condition $u^n_T(x, y) = (\pi_n\psi)(x, y)$. 

Next, we apply the generalisation of It\^o's formula from Theorem \ref{thm:ito} for Hilbert space-valued semimartingales with jumps to obtain
\begin{align} \label{eq:fd_galerkin_ito}
\begin{split}
    \lVert u^n_t\rVert_{L^2}^2 &= \lVert \pi_n \psi\rVert_{L^2}^2 + \int_t^T 2\Bigl(\bigl\langle \nabla_x \cdot (a_s \nabla_x u^n_s) + \beta_s \cdot q^n_s, u^n_s\bigr\rangle - \bigl\langle a_s \nabla_x u^n_s + \alpha_s q^n_s, \nabla_x u^n_s \bigr\rangle\Bigr) \, \d s \\
    &\ \ \ + \int_t^T \langle h_s, u^n_s\rangle \, \d s- \int_t^T 2\langle q^n_s, u^n_s\rangle \cdot \d W_s - \int_t^T 2 \langle u^n_{s-}, \d m^n_s\rangle \\
    &\ \ \ - \int_t^T \lVert q^n_s\rVert_{L^2}^2 \, \d s - ([m^n]_T - [m^n]_t),
\end{split}
\end{align}
where we use that $2\langle \lambda_s \partial_y u^n_s, u^n_s\rangle = \langle \partial_y \lvert u^n_s\rvert^2, \lambda_s\rangle = 0$ because $\lambda_s$ does not depend on $y$ and $u^n_s(x, y)$ vanishes for $y = \ell$. The stochastic integrals appearing in \eqref{eq:fd_galerkin_ito} are true martingales by Proposition \ref{prop:bdg}, so disappear upon taking expectations on both sides. Then we apply the same estimates as in the proof of \cite[Theorem 3.1]{jin_1999_bspde} to find that for any $\delta > 0$ it holds that
\begin{align} \label{eq:fd_galerkin_est}
\begin{split}
    \ev\lVert u^n_t\rVert_{L^2}^2 &\leq \ev\lVert \pi_n \psi\rVert_{L^2}^2 + \ev \int_t^T \Bigl(\delta \lVert \nabla_x u^n_s\rVert_{L^2}^2 + \tfrac{C_{\alpha}^2 + C_{\beta}^2 + \delta}{\delta} \lVert u^n_s\rVert_{L^2}^2 + \delta \lVert q^n_s\rVert_{L^2}^2 + \lVert h_s\rVert_{L^2}^2\Bigr) \, \d s \\
    &\ \ \ - \ev\int_t^T \Bigl(2 \lVert \sigma_s^{\top} \nabla_x u^n_s\rVert_{L^2}^2 + \lVert \alpha_s^{\top} \nabla_x u^n_s + q^n_s\rVert_{L^2}^2\Bigr) \, \d s - \ev([m^n]_T - [m^n]_t),
\end{split}
\end{align}
where $\sigma_s(x, y)\sigma_s^{\top}(x, y) = a_s(x, y) - \tfrac{\alpha_s \alpha_s^{\top}}{2}$ satisfies $\sigma_s(x, y)\sigma_s^{\top}(x, y) \geq c_{a, \alpha} I_d$ by Assumption \ref{ass:bspde} \ref{it:nondeg_bspde}. For any $\eta \in (0, 1)$ we have
\begin{equation*}
    -\lVert \alpha_s^{\top} \nabla_x u^n_s + q^n_s\rVert_{L^2}^2 \leq \frac{(1 - \eta)\lvert \alpha_s\rvert^2}{\eta} \lVert \nabla_x u^n_s\rVert_{L^2}^2 - (1 - \eta) \lVert q^n_s\rVert_{L^2}^2.
\end{equation*}
Now, we first choose $\eta$ sufficiently close to $1$ such that $\frac{(1 - \eta)\lvert \alpha_s\rvert^2}{\eta} \leq \frac{c_{a, \alpha}}{2}$, then we fix $\delta = \tfrac{c_{a, \alpha}}{2} \land \frac{1 - \eta}{2}$. In view of Equation \eqref{eq:fd_galerkin_est} this yields
\begin{align*}
    \ev\lVert u^n_t\rVert_{L^2}^2 + \ev([m^n]_T - [&m^n]_t) + \ev\int_t^T \lVert \nabla_x u^n_s\rVert_{L^2}^2 + \lVert q^n_s\rVert_{L^2}^2 \, \d s \\
    &\leq C \ev\lVert \pi_n \psi\rVert_{L^2}^2 + C \ev\int_t^T \lVert h_s \rVert + \lVert u^n_s\rVert_{L^2}^2 \, \d s
\end{align*}
for a positive constant $C > 0$ independent of $n \geq 1$ and $h$. Now, by Gr\"onwall's inequality, we can deduce that
\begin{align} \label{eq:galerkin_estimate_2}
\begin{split}
    \ev\lVert u^n_t\rVert_{L^2}^2 + \ev \lVert m^n_T - m^n_t&\rVert_{L^2}^2 + \ev\int_t^T \lVert \nabla_x u^n_s\rVert_{L^2}^2 + \lVert q^n_s\rVert_{L^2}^2 \, \d s \\
    &\leq e^{CT}\ev\lVert \pi_n \psi\rVert_{L^2}^2 + e^{CT}\ev \int_t^T \lVert h_t\rVert_{L^2}^2 \, \d t,
\end{split}
\end{align}
where we used that $\ev([m^n]_T - [m^n]_t) = \ev\bigl(\lVert m^n_T\rVert_{L^2}^2 - \lVert m^n_t\rVert_{L^2}^2\bigr) = \ev \lVert m^n_T - m^n_t\rVert_{L^2}^2$ because as a martingale $m^n$ satisfies $\ev\langle m^n_T, m^n_t\rangle = \ev \langle m^n_t, m^n_t\rangle = \ev \lVert m^n_t\rVert_{L^2}^2$. Consequently, the sequence $(u^n, q^n, m^n)_n$ is bounded in $\bb{H}(D)$ and, consequently, has a weakly convergent subsequence. Then we can proceed as in the proofs of \cite[Theorem 3.1]{jin_1999_bspde} and \cite[Lemma 4.5]{al_hussein_bspde_2009} to show that any weak limit $(u, q, m) \in \bb{H}(D)$ is a solution of BSPDE \eqref{eq:bspde_linear}. We will not provide the details and instead refer the reader to the cited references. Note that the test functions for the solutions lie in the closure of the linear of span of $(e_n)_{n \geq 1}$ in $H^1(D)$. By our choice of $(e_n)_{n \geq 1}$ this closure is given by $H^1_0(D)$, as desired. 

\textit{Uniqueness}: To prove uniqueness for BSPDE \eqref{eq:bspde_linear}, it is convenient to first establish \eqref{eq:bspde_linear_estimate} for arbitrary solutions to BSPDE \eqref{eq:bspde_linear}.
So let us fix a solution $(u, q, m)$ of BSPDE \eqref{eq:bspde_linear}. We apply the generalisation of It\^o's formula from Theorem \ref{thm:ito_2_dim} to $u$, using that $\lambda$ is nonnegative, to obtain
\begin{align*}
    \lVert u_t\rVert_{L^2}^2 &\leq \lVert \psi\rVert_{L^2}^2 + \int_t^T 2\Bigl(\bigl\langle \nabla_x \cdot (a_s \nabla_x u_s) + \beta_s \cdot q_s + h_s, u_s\bigr\rangle - \bigl\langle a_s \nabla_x u_s + \alpha_s q_s, \nabla_x u_s \bigr\rangle\Bigr) \, \d s \\
    &\ \ \ - \int_t^T 2\langle q_s, u_s\rangle \cdot \d W_s - \int_t^T 2 \langle u_{s-}, \d m_s\rangle - \int_t^T \lVert q_s\rVert_{L^2}^2 \, \d s - ([m]_T - [m]_t),
\end{align*}
where we use that the process $\xi$ in \eqref{eq:ito_formula_2_dim} is nonnegative. Note that Theorem \ref{thm:ito_2_dim} also tells us that $u$ must be an element of $D_{\bb{F}}^2([0, T]; L^2(D))$. Then, we employ the same arguments as above to deduce the inequality
\begin{align*}
    \ev\lVert u_t\rVert_{L^2}^2 + \ev&([m]_T - [m]_t) + \ev\int_t^T \lVert \nabla_x u_s\rVert_{L^2}^2 + \lVert q_s\rVert_{L^2}^2 \, \d s \\
    &\leq C \ev\lVert \psi\rVert_{L^2}^2 + C \ev\int_t^T \lvert \langle h_s, u_s\rangle\rvert + \lVert u_s\rVert_{L^2}^2 \, \d s
\end{align*}
for a positive constant $C > 0$ independent of $h$. Then a simple application of Gr\"onwall's inequality yields the estimate \eqref{eq:bspde_linear_estimate} for $(u, q, m)$. Now, if $(u', q', m')$ is another solution to BSPDE \eqref{eq:bspde_linear} with the same terminal condition $\psi$ and input $h$, then the difference $(u - u', q - q', m - m')$ satisfies BSPDE \eqref{eq:bspde_linear} with vanishing initial condition and input. Hence, the just established bound from \eqref{eq:bspde_linear_estimate} immediately implies that the difference $(u - u', q - q', m - m')$ vanishes. Thus, BSPDE \eqref{eq:bspde_linear} has a unique solution.
\end{proof}

We can use Proposition \ref{prop:bspde_linear} and, in particular, the estimate \eqref{eq:bspde_linear_estimate} to obtain a solution to BSPDE \eqref{eq:bspde} via a fixed point argument. Let us note again that we will only prove Theorem \ref{thm:bspde} for the case $D = \R^d \times (\ell, \infty)$.

\begin{proof}[Proof of Theorem \ref{thm:bspde}]
We construct a map $\Phi \define \bb{H}(D) \to \bb{H}(D)$ as follows: for any $(u, q, m) \in \bb{H}(D)$ we define the random function $h^u \in L^2_{\bb{F}}([0, T]; L^2(D))$ by
\begin{align*}
    h^u_t(x, y) = H_t\bigl(x, y, u_t(x, y), \nabla_x u_t(x, y)\bigr) + F_t\bigl(x, y, u_t, \nabla_x u_t\bigr).
\end{align*}
Note that Assumption \ref{ass:bspde} \ref{it:growth_bspde} implies that
\begin{align} \label{eq:h_bound}
\begin{split}
    \ev \int_0^T \lVert h^u_t\rVert_{L^2}^2 \, \d t &\leq \bigl(3 C_H^2 + 2 \lVert h_F\rVert_{L^2}^2\bigr) \ev \int_0^T \lVert u_t\rVert_{L^2}^2 + \lVert \nabla_x u_t\rVert_{L^2}^2 \, \d t \\
    &\ \ \ + T\bigl(3C_H^2\lVert h_H\rVert_{L^2}^2 + 2\lVert h_F\rVert_{L^2}^2\bigr) \\
    &< \infty,
\end{split}
\end{align}
so that $h^u$ is indeed an element of $L^2_{\bb{F}}([0, T]; L^2(D))$. Then, we define $\Phi(u, q, m)$ as the unique solution $(u', q', m')$ of the linear BSPDE \eqref{eq:bspde_linear} with input $h$ replaced by $h^u$. 

Now, fix two elements $(u^1, q^1, m^1)$, $(u^2, q^2, m^2) \in \bb{H}(D)$. Then the difference $\Delta = (\Delta^u, \Delta^q, \Delta^m) = \Phi(u^1, q^1, m^1) - \Phi(u^2, q^2, m^2)$ solves the linear BSPDE \eqref{eq:bspde_linear} with input $h^{u^1} - h^{u^2}$ and initial condition zero. Hence, using \eqref{eq:bspde_linear_estimate} we get that
\begin{align} \label{eq:bspde_cutoff_estimate}
    \ev\lVert \Delta^u_t\rVert_{L^2}^2  + \ev\lVert \Delta^m_T - \Delta^m_t&\rVert_{L^2}^2 + \ev\int_t^T \lVert \nabla_x \Delta^u_s \rVert_{L^2}^2 + \lVert \Delta^q_s \rVert_{L^2}^2 \, \d s \notag \\
    &\leq C\ev\int_t^T \bigl\lvert \bigl\langle \Delta^u_s, h^{u^1}_s - h^{u^2}_s\bigr\rangle \bigr\rvert \, \d s \notag \\
    &\leq \frac{C}{2\epsilon} \ev \int_t^T \lVert \Delta^u_s\rVert_{L^2}^2 \, \d t + \frac{\epsilon}{2} \ev \int_t^T \bigl\lVert h^{u^1}_s - h^{u^2}_s\bigr\rVert_{L^2}^2 \, \d s.
\end{align}
Our goal is to estimate the term $\ev\bigl\lVert h^{u^1}_s - h^{u^2}_s \bigr\rVert_{L^2}^2$. First, by Assumption \ref{ass:bspde} \ref{it:lipschitz_bspde} on $H$ we have that
\begin{align*}
    \int_D \Bigl\lvert H_s&\bigl(x, y, u^1_s(x, y), \nabla_x u^1_s(x, y)\bigr) - H_s\bigl(x, y, u^2_s(x, y), \nabla_x u^2_s(x, y)\bigr)\Bigr\rvert^2 \, \d x\d y \\
    &\leq \int_D 2 C_H^2 \Bigl(\bigl\lvert u^1_s(x, y) - u^2_s(x, y)\bigr\rvert^2 + \bigl\lvert \nabla_x u^1_s(x, y) - \nabla_x u^2_s(x, y)\bigr\rvert^2\Bigr) \, \d x \d y \\
    &\leq 2 C_H^2 \lVert u^1_s - u^2_s\rVert_{H^{1, 0}}^2.
\end{align*}
Secondly, Assumption \ref{ass:bspde} \ref{it:lipschitz_bspde} on $F$ yields
\begin{align*}
    \int_D \bigl\lvert F_s&\bigl(x, y, u^1_s, \nabla_x u^1_s\bigr) - F_s\bigl(x, y, u^2_s, \nabla_x u^2_s\bigr)\bigr\rvert^2 \, \d x \d y \\
    &\leq 2\lVert h_F\rVert_{L^2}^2 \bigl(\lVert u^1_s - u^2_s \rVert_{L^2}^2 + \lVert \nabla_x u^1_s - \nabla_x u^2_s \rVert_{L^2}^2\bigr) \\
    &\leq 2 \lVert h_F\rVert_{L^2}^2 \lVert u^1_s - u^2_s\rVert_{H^{1, 0}}^2
\end{align*}
Combining the previous two displays implies that $\ev\bigl\lVert h^{u^1}_s - h^{u^2}_s \bigr\rVert_{L^2}^2 \leq C_0\ev \lVert u^1_s - u^2_s\rVert_{H^{1, 0}}^2$, where $C_0 = 3C_H^2(\lVert h_H\rVert_{L^2}^2 \lor 1) + 2 \lVert h_F\rVert_{L^2}^2$. We insert this inequality into \eqref{eq:bspde_cutoff_estimate} and then fix $\epsilon = C_0^{-1}$ to find that
\begin{align} \label{eq:bspde_estimate_1}
\begin{split}
    \ev\lVert \Delta^u_t\rVert_{L^2}^2  + \ev\lVert \Delta^m_T - \Delta^m_t&\rVert_{L^2}^2 + \ev\int_t^T \lVert \nabla_x \Delta^u_s \rVert_{L^2}^2 + \lVert \Delta^q_s \rVert_{L^2}^2 \, \d s \\
    &\leq \frac{C C_0}{2} \ev \int_t^T \lVert \Delta^u_s\rVert_{L^2}^2 \, \d t + \frac{1}{2} \ev \int_t^T \lVert u^1_s - u^2_s\rVert_{H^{1, 0}}^2 \, \d s.
\end{split}
\end{align}
Now, we proceed as in \cite[Example 1.20]{carmona_mfg_ii_2018}. We multiply this inequality by $e^{\eta t}$ for $\eta \geq 1$ and then integrate from over $[0, T]$ to obtain
\begin{align*}
    \int_0^T &e^{\eta t} \ev\lVert \Delta^u_t\rVert_{L^2}^2 \, \d t + \int_0^T \frac{1}{\eta}(e^{\eta s} - 1) \bigl(\ev\lVert \nabla_x \Delta^u_s \rVert_{L^2}^2 + \ev\lVert \Delta^q_s \rVert_{L^2}^2\bigr) \, \d s \\
    &\leq \frac{C C_0}{2\eta} \int_0^T e^{\eta s} \ev\lVert \Delta^u_s\rVert_{L^2}^2 \, \d s + \frac{1}{2} \int_0^T \frac{1}{\eta}(e^{\eta s} - 1) \ev \lVert u^1_s - u^2_s\rVert_{H^{1, 0}}^2\, \d s.
\end{align*}
Next, we evaluate Equation \eqref{eq:bspde_estimate_1} at $t = 0$, multiply it by $\frac{1}{2\eta}$, and add it to the inequality above. This yields
\begin{align*}
    \int_0^T e^{\eta t} \ev\lVert\Delta^u_t\rVert_{L^2}^2& \, \d t + \frac{1}{2\eta} \ev \lVert m_T\rVert_{L^2}^2 + \int_0^T \theta(s) \bigl(\ev\lVert \nabla_x \Delta^u_s \rVert_{L^2}^2 + \ev\lVert \Delta^q_s \rVert_{L^2}^2\bigr) \, \d s \\
    &\leq \frac{3C C_0}{4\eta} \int_0^T e^{\eta s} \ev\lVert \Delta^u_s\rVert_{L^2}^2 \, \d s + \frac{1}{2} \int_0^T \theta(s) \ev \lVert u^1_s - u^2_s\rVert_{H^{1, 0}}^2\, \d s,
\end{align*}
where $\theta(s) = \frac{1}{2\eta} + \frac{1}{\eta}(e^{\eta s} - 1)$.
Then we fix $\eta \geq 1$ large enough so that $\frac{3C C_0}{4\eta} < \frac{1}{4}$ and rearrange, which finally implies
\begin{equation} \label{eq:bspde_estimate_2}
    \int_0^T \theta(t) \bigl(\ev\lVert\Delta^u_t\rVert_{H^{1, 0}}^2 + \ev\lVert \Delta^q_t \rVert_{L^2}^2\bigr) \, \d t + \frac{2}{3\eta} \ev \lVert m_T\rVert_{L^2}^2 \leq \frac{2}{3} \int_0^T \theta(t) \ev \lVert u^1_t - u^2_t\rVert_{H^{1, 0}}^2\, \d t.
\end{equation}
Here we used that $\theta(t) \leq e^{\eta t}$. We can equipped $\bb{H}(D)$ with the norm
\begin{equation*}
    (u, q, m) \mapsto \int_0^T \theta(t) \bigl(\ev\lVert\Delta^u_t\rVert_{H^{1, 0}}^2 + \ev\lVert \Delta^q_t \rVert_{L^2}^2\bigr) \, \d t + \frac{2}{3\eta} \ev \lVert m_T\rVert_{L^2}^2,
\end{equation*}
which turns $\bb{H}(D)$ into a Banach space. Moreover, owing to \eqref{eq:bspde_estimate_2}, the map $\Phi$ is a contraction on this Banach space and by the Banach fixed point theorem has a unique fixed point $(u, q, m)$. This fixed point is a solution to BSPDE \eqref{eq:bspde} in the sense of Definition \eqref{def:bspde_solution}. Moreover, since any other solution to BSPDE \eqref{eq:bspde} is also a fixed point of $\Phi$, the solution $(u, q, m)$ is unique.

Lastly, we derive \eqref{eq:bspde_semilinear_estimate} for $(u, q, m)$. Since $(u, q, m)$ solves BSPDE \eqref{eq:bspde_linear} with input $h^u$, the estimate \eqref{eq:bspde_linear_estimate} yields
\begin{align*}
    \ev\lVert u_t\rVert_{L^2}^2 + \ev\lVert m_T - m_t\rVert_{L^2}^2 + \ev&\int_t^T \lVert \nabla_x u_s \rVert_{L^2}^2 + \lVert q_s \rVert_{L^2}^2 \, \d s \\
    &\leq C\biggl(\ev\lVert \psi \rVert_{L^2}^2 + \ev\int_t^T  \lvert \langle u_s, h^u_s\rangle\rvert \, \d s\biggr)
\end{align*}
Now, we proceed as above, using a bound analogous \eqref{eq:h_bound} to estimate $\ev \int_t^T \lVert h^u_s\rVert_{L^2}^2 \, \d s$, in order to establish \eqref{eq:bspde_semilinear_estimate}. We will not provide the details here. 
\end{proof}

\section{Application to a McKean--Vlasov Control Problem with Killing at a State-Dependent Intensity} \label{sec:equivalence}

\subsection{Control Problem for the Stochastic Fokker--Planck Equation} \label{sec:sfpe}

In this subsection we reformulate the McKean--Vlasov control problem introduced at the beginning of Subsection \ref{sec:main_results_application} as the control of the stochastic Fokker--Planck equation \eqref{eq:sfpe_joint}, which means proving Proposition \ref{prop:sfpe_joint}. Subsequently, we introduce a process that describes the variations of the state process $\mu$ under infinitesimal variations of the control $g$.

\begin{proof}[Proof of Proposition \ref{prop:sfpe_joint}]
We prove the three statements of the proposition in sequence. That $\mu = (\L(X_t, \Lambda_t \vert \F_T))_{0 \leq t \leq T}$ solves SPDE \eqref{eq:sfpe_joint} follows along the same lines as the proof of Theorem 1.9 in \cite{hammersley_weak_ex_uni_2021}, so we skip the details here and immediately move to the second statement of the proposition. 

Our next goal is to construct a solution to SPDE \eqref{eq:sfpe_joint}. We proceed in two steps. 

\textit{Step 1}: For arbitrary coefficients $\beta \in L_{\bb{F}}^{\infty}([0, T]; L^{\infty}(\R^2))$, $h \in L_{\bb{F}}^2([0, T]; H^{-1, 0}(\R^2))$ let us consider the linear SPDE
\begin{equation} \label{eq:sfpe_random}
    \d \langle \mu_t, \varphi\rangle = \bigl(\bigl\langle \mu_t, \beta_t \partial_x \varphi + \lambda_t \partial_y \varphi + a_t \partial_x^2 \varphi\bigr\rangle + h_t(\varphi)\bigr)\, \d t + \langle \mu_t, \sigma_0(t) \partial_x \varphi \rangle \, \d W_t
\end{equation}
for $\varphi \in C_c^2(\R^2)$ with initial condition $\mu_0 \in L^2(\R^2)$. 
It is a straightforward exercise to obtain a unique strong solution $\mu \in C_{\bb{F}}^2([0, T]; L^2(\R^2)) \cap L_{\bb{F}}^2([0, T]; H^{1, 0}(\R^2))$ to SPDE \eqref{eq:sfpe_random} through a Galerkin approximation. The continuity of the trajectories of $\mu$ as a process with values in $L^2(\R^2)$ follows from Theorem \ref{thm:ito_2_dim}. Using the generalisation of It\^o's formula from Theorem \ref{thm:ito_2_dim} and the fact that the martingale term therein vanishes, since $\langle \mu_t, \sigma_0(t) \partial_x \mu_t\rangle = 0$, we obtain the standard estimate
\begin{equation} \label{eq:sfpe_estimate}
    \sup_{0 \leq s \leq t} \lVert \mu_s\rVert_{L^2}^2 + \int_0^t \lVert \mu_s\rVert_{H^{1, 0}}^2 \, \d s \leq C\lVert \mu_0\rVert_{L^2}^2 + C \int_0^t \bigl(\lvert \langle \mu_s, \beta_s \partial_x \mu_s\rangle\rvert + \lvert h_s(\mu_s)\rvert\bigr) \, \d s
\end{equation}
for $0 \leq t \leq T$ and a constant $C > 0$ depending on the coefficients $\lambda$, $a$, and $\sigma_0$ as well as the time horizon $T$. Moreover, if the initial condition $\mu_0$ is nonnegative with $M_2^2(\mu_0) < \infty$ and $h = 0$, then by enlarging the constant if necessary we find that $\ev \sup_{0 \leq t \leq T} M_2^2(\mu_t) \leq e^{C + C \lVert \beta \rVert_{L^{\infty}}}M_2^2(\mu_0)$.

\textit{Step 2}: Next, we employ a fixed point argument to find a solution for the nonlinear SPDE \eqref{eq:sfpe_joint}. Let us define $\bb{S}$ as the space of elements $\rho$ in $C_{\bb{F}}^2([0, T]; L^2(\R^2)) \cap L_{\bb{F}}^2([0, T]; H^{1, 0}(\R^2))$ such that $\ev \sup_{0 \leq t \leq T} M_2^2(\rho_t) \leq e^{C + C C_b} \ev[\lvert \xi\rvert^2 + \lvert \zeta\rvert^2]$, where $C$ is the constant from above and $C_b$ is the uniform bound on $b$ guaranteed by Assumption \ref{ass:spfe_joint} \ref{it:spde_growth}. The space $\bb{S}$ is a closed subset of the Banach space $C_{\bb{F}}^2([0, T]; L^2(\R^2)) \cap L_{\bb{F}}^2([0, T]; H^{1, 0}(\R^2))$ and, hence, a complete metric space. Consequently, we can apply the Banach fixed point theorem to contractions on $\bb{S}$. Moreover, we know that $\langle \rho_t - \rho_s, \varphi\rangle \to 0$ as $s \to t$ for any $\varphi \in L^2(\R^2)$, because $\rho$ has continuous trajectories. Since $\ev \sup_{0 \leq t \leq T} M_2^2(\rho_t) < \infty$, it is not difficult to deduce that a.s.\@ for all $f \in C_b(\R^2)$ it holds that $\langle \rho_t -  \rho_s, f\rangle \to 0$ as $s \to t$. However, the Borel $\sigma$-algebras on $\P^2(\R^2)$ induced by weak convergence and convergence in the $2$-Wasserstein distance coincide, so that $\rho$ is $\bb{F}$-progressively measurable as a process with values in $\P^2(\R^2)$. This ensure that for any $\rho \in \bb{S}$, the map $[0, T] \times \Omega \times \R^2 \mapsto (t, \omega, x, y) \mapsto b\bigl(t, x, \cal{S}(\rho_t), g_t(x, y)\bigr)$ is $\bb{F}$-progressively measurable.

Now, we construct a map $\Phi$ from $\bb{S}$ to itself as follows: for $\mu \in \bb{S}$ we define the process $\beta^{\mu}$ in $L_{\bb{F}}^{\infty}([0, T]; L^{\infty}(\R^2))$ by $\beta^{\mu}_t(x, y) = b\bigl(t, x, \cal{S}(\mu_t), g_t(x, y)\bigr)$ for $(t, \omega,x, y) \in [0, T] \times \Omega \times \R^2$. Then, we let $\Phi(\mu) \in C_{\bb{F}}^2([0, T]; L^2(\R^2)) \cap L_{\bb{F}}^2([0, T]; H^{1, 0}(\R^2))$ denote the unique strong solution to SPDE \eqref{eq:sfpe_random} with $\beta$ replaced by $\beta^{\mu}$, $h = 0$, and initial condition $\mu_0 = \L(\xi, \zeta)$, which is in $L^2(\R^2)$ by Assumption \ref{ass:spfe_joint} \ref{it:spde_ic}. From \eqref{eq:sfpe_estimate} and an application of Gr\"onwall's inequality, it follows that $\sup_{0 \leq t \leq T} \lVert \Phi(\mu)_t\rVert_{L^2}$ is bounded by a constant independent of $\mu$. Moreover, by the discussion below Equation \eqref{eq:sfpe_estimate}, $\ev \sup_{0 \leq t \leq T} M_2^2(\Phi(\mu)_t) \leq e^{C + C C_b} \ev[\lvert \xi\rvert^2 + \lvert \zeta\rvert^2]$, so that $\Phi(\mu) \in \bb{S}$. We will show that $\Phi$ is a contraction if $T$ is sufficiently small. For two flows $\mu^1$, $\mu^2 \in \bb{S}$ we can consider the difference $\Delta = (\Delta_t)_{0 \leq t \leq T}$ given by $\Delta_t = \Phi(\mu^1)_t - \Phi(\mu^2)_t$. The process $\Delta$ solves SPDE \eqref{eq:sfpe_random} with $\beta$ replaced by $\beta^{\mu_1}$ and $h$ defined by $h_t(\varphi) = \bigl\langle \Phi(\mu^2)_t, (\beta^{\mu_1} - \beta^{\mu_2}) \partial_x \varphi\bigr\rangle$. Thus, the estimate \eqref{eq:sfpe_estimate} implies
\begin{align} \label{eq:bounds_delta}
\begin{split}
    \sup_{0 \leq s \leq t} \lVert \Delta_t&\rVert_{L^2}^2 + \int_0^t \lVert \Delta_s\rVert_{H^{1, 0}}^2 \, \d s \\
    &\leq C \int_0^t \Bigl(\lvert \langle \Delta_s, \beta^{\mu_1}_s \partial_x \Delta_s\rangle\rvert + \bigl\lvert \bigl\langle \Phi(\mu^2)_s, (\beta^{\mu_1}_s - \beta^{\mu_2}_s) \partial_x \Delta_s\bigr\rangle\bigr\rvert\Bigr) \, \d s.
\end{split}
\end{align}
The quantity $\lvert \langle \Delta_s, \beta^{\mu_1}_s \partial_x \Delta_s\rangle\rvert$ is easily estimated by $\frac{C_b^2}{2\epsilon} \lVert \Delta_s\rVert_{L^2}^2 + \frac{\epsilon}{2} \lVert \Delta_s \rVert_{H^{1, 0}}^2$ for an $\epsilon > 0$ that we shall choose later. Next, by Assumption \ref{ass:spfe_joint} \ref{it:spde_lipschitz} it holds that
\begin{align*}
    \bigl\lvert \bigl\langle \Phi(\mu^2)_s, (\beta^{\mu_1}_s - \beta^{\mu_2}_s) \partial_x \Delta_s\bigr\rangle\bigr\rvert &\leq \lVert \Phi(\mu^2)_s\rVert_{L^2} \lVert \Delta_s\rVert_{H^{1, 0}} C_b \lVert \cal{S}(\mu^1_s) - \cal{S}(\mu^2_s)\rVert_{L^2} \\
    &\leq \frac{C_b^2 \lVert \Phi(\mu^2)_s\rVert_{L^2}^2}{2\epsilon}\lVert \mu^1_s - \mu^2_s\rVert_{L^2}^2 + \frac{\epsilon}{2}\lVert \Delta_s \rVert_{H^{1, 0}}^2.
\end{align*}
Here we used in the second inequality that the restriction of the map $\cal{S}$ defined in Equation \eqref{eq:subprobability_function} to $L^2(\R^2) \cap L^1(\R)$, which takes values in $L^2(\R)$, is $1$-Lipschitz with respect to $\lVert \cdot \rVert_{L^2}$. Inserting the previous two estimates into \eqref{eq:bounds_delta}, choosing $\epsilon = \frac{1}{2C}$, and rearranging implies that
\begin{equation*}
    \sup_{0 \leq s \leq t} \lVert \Delta_t\rVert_{L^2}^2 + \frac{1}{2}\int_0^t \lVert \Delta_s\rVert_{H^{1, 0}}^2 \, \d s \leq C^2 C_b^2 \int_0^t \lVert \Delta_s\rVert_{L^2}^2 + \lVert \Phi(\mu^2)_s\rVert_{L^2}^2\lVert \mu^1_s - \mu^2_s\rVert_{L^2}^2 \, \d s.
\end{equation*}
Then, we apply Gr\"onwall's inequality to obtain
\begin{equation*}
    \sup_{0 \leq s \leq t} \lVert \Delta_s\rVert_{L^2}^2 + \int_0^t \lVert \Delta_s\rVert_{H^{1, 0}}^2 \, \d s \leq e^{2 C^2C_b^2 t}t \sup_{0 \leq s \leq T}\lVert \Phi(\mu^2)_s\rVert_{L^2}^2\sup_{0 \leq s \leq t} \lVert \mu^1_s - \mu^2_s\rVert_{L^2}^2.
\end{equation*}
Since $\sup_{0 \leq s \leq T}\lVert \Phi(\mu^2)_s\rVert_{L^2}^2$ is bounded by a constant independent of $\mu^2$, we can choose $t_1 \in (0, T]$ small enough such that $\sup_{0 \leq s \leq t_1} \lVert \Delta_s\rVert_{L^2}^2 + \int_0^{t_1} \lVert \Delta_s\rVert_{H^{1, 0}}^2 \, \d s \leq \frac{1}{2}\sup_{0 \leq s \leq t_1} \lVert \mu^1_s - \mu^2_s\rVert_{L^2}^2$. Hence, $\Phi$ is a contraction when restricted to the time domain $[0, t_1]$ and so it possesses a unique fixed point $\mu$. This fixed point is clearly the unique solution to SPDE \eqref{eq:sfpe_joint} in $\bb{S}$ with input $g$ on the interval $[0, t_1]$. We can iteratively extend this solution $\mu$ to the intervals $[t_1, 2t_1 \land T]$, $[2t_1 \land T, 3t_1 \land T]$, and so forth, where \eqref{eq:sfpe_estimate} ensures that the expression $\sup_{0 \leq s \leq t} \lVert \mu_s\rVert_{L^2}^2 + \int_0^t \lVert \mu_s\rVert_{L^2}^2 \, \d s$ stays bounded for all $t \in [0, T]$. This procedure terminates after finitely many steps with a solution $\mu$ on the entire time domain $[0, T]$. The uniqueness of this solution in $\bb{S}$ follows from the uniqueness on the individual intervals $[kt_1 \land T, (k + 1)t_1 \land T]$. The estimate \eqref{eq:sfpe_uniform_bound} is a direct consequence of \eqref{eq:sfpe_estimate}.

Next, we check that the trajectories of $\mu$ lie in $C([0, T]; \P^2(\R^2))$. We already know from our definition of $\bb{S}$ that $\mu$ takes values in $C([0, T]; \P(\R^2))$. In addition, we can appeal to the SPDE satisfied by $\mu$ to show that a.s.\@ $M_2^2(\mu_s) \to M_2^2(\mu_t)$ as $s \to t$. Since convergence in $\P(\R^2)$ together with convergence of the second moment is equivalent to convergence in $\P^2(\R^2)$ (cf.\@ \cite[Definition 6.8]{villani_ot_2009}), we deduce that $\mu$ takes values in $C([0, T]; \P^2(\R^2))$.

So far we only proved uniqueness of solutions in the space $\bb{S}$, but we need to establish uniqueness within the space of $\bb{F}$-adapted $C([0, T]; \P^2(\R^2))$-valued processes. Assume we have another $\bb{F}$-adapted solution $\mu'$ of SPDE \eqref{eq:sfpe_joint} that takes values in $C([0, T]; \P^2(\R^2))$. Then, using a mollification argument together with standard SPDE estimates, we can show that $\mu'$ must actually take values in $C_{\bb{F}}^2([0, T]; L^2(\R^2)) \cap L_{\bb{F}}^2([0, T]; H^{1, 0}(\R^2))$ as well. 
But then since $\mu'$ solves SPDE \eqref{eq:sfpe_joint} it follows from the discussion below \eqref{eq:sfpe_estimate} that $\ev \sup_{0 \leq t \leq T} M_2^2(\mu'_t) < e^{C + C C_b} \ev[\lvert \xi\rvert^2 + \lvert \zeta\rvert^2]$, so that $\mu' \in \bb{S}$. Then, uniqueness in $\bb{S}$ guarantees that $\mu' = \mu$. 

The third assertion of the proposition is a simple application of \cite[Theorem 1.3.]{lacker_mimicking_2020}.

\end{proof}

Next, we want to study the variation of $\mu$ for infinitesimal perturbations of $g$. Let $h \define [0, T] \times \Omega \times \R^2 \to \R^{d_G}$ be a bounded $\bb{F}$-progressively measurable random function such that $g^{\epsilon} = g + \epsilon h$ maps into $G$ for all $\epsilon > 0$. Denote the unique solution to SPDE \eqref{eq:sfpe_joint} with input $g^{\epsilon}$ by $\mu^{\epsilon}$ and define $\cal{V}^{\epsilon} = \frac{1}{\epsilon} (\mu^{\epsilon} - \mu)$. We want to show that $\cal{V}^{\epsilon}$ converges in $C_{\bb{F}}^2([0, T]; L^2(\R^2)) \cap L_{\bb{F}}^2([0, T]; H^{1, 0}(\R^2))$ to the solution of the random linear SPDE
\begin{align} \label{eq:spde_variation}
\begin{split}
    \d \langle \cal{V}_t, \varphi\rangle &= \bigl\langle \cal{V}_t, \L_2\varphi(t, \cdot, \cdot, \nu_t, g_t)\big\rangle \, \d t + \langle \cal{V}_t, \sigma_0(t) \partial_x \varphi\rangle \, \d W_t + \bigl\langle \mu_t, \partial_g b_1(t, \cdot) h_t \partial_x \varphi\bigr\rangle \, \d t \\
    &\ \ \ + \biggl(\int_{\R \times [0, \infty)} \bigl\langle \cal{V}_t, e^{-\cdot}Db_0\bigl(t, x, \nu_t\bigr)\bigr\rangle \partial_x \varphi(x, y) \, \d \mu_t(x, y)\biggr) \, \d t
\end{split}
\end{align}
for $\varphi \in H^1(\R^2)$ with initial condition $\cal{V}_0 = 0$. Here we recall that $b_0$ and $b_1$ are such that $b(t, x, v, g') = b_0(t, x, v) + b_1(t, x, g')$ for $(t, x, v, g') \in [0, T] \times \R \times \cal{M}^2_{\leq 1}(\R) \times G$. Note that $b_1$ is linear in $g'$, so that $\partial_g b_1$ does not depend on $g'$ and, therefore, we write $\partial_g b_1(t, x)$ instead of $\partial_g b_1(t, x, g')$. Let us begin with the following lemma.

\begin{lemma} \label{eq:lfd_s}
Assume that $F \define \cal{M}^2_{\leq 1}(\R) \to \R$ has a linear functional derivative. Then $F \circ \cal{S} \define \P^2(\R \times [0, \infty)) \to \R$ has a linear functional derivative $D(F \circ \cal{S})$ given by $(F \circ \cal{S})(m)(x, y) = e^{- y} DF(\cal{S}(m))(x)$ for $(x, y) \in \R \times [0, \infty)$ and $m \in \P^2(\R \times [0, \infty))$.
\end{lemma}

\begin{proof}
Since $F$ has a linear functional derivative $DF$, we have for $m$, $m' \in \P^2(\R^2)$ that
\begin{align*}
    (F\circ \cal{S}&)(m) - (F\circ \cal{S})(m') \\
    &= \int_0^1 \int_{\R} DF\bigl(t \cal{S}(m) + (1 - t)\cal{S}(m')\bigr)(x) \, \d \bigl(\cal{S}(m) - \cal{S}(m')\bigr)(x) \, \d t \\
    &= \int_0^1 \int_{\R \times [0, \infty)} e^{- y} DF\bigl(\cal{S}(t m + (1 - t) m')\bigr)(x) \, \d (m - m')(x, y) \, \d t \\
    &= \int_0^1 \int_{\R \times [0, \infty)} D(F \circ \cal{S})(t m + (1 - t) m')(x, y) \, \d (m - m')(x, y) \, \d t,
\end{align*}
where we defined $D(F \circ \cal{S})(m)(x, y) = e^{-y} DF(\cal{S}(m))(x)$ for $(x, y) \in \R \times [0, \infty)$. Clearly, $(x, y) \mapsto D(F\circ \cal{S})(m)(x, y) = e^{-y}DF(\cal{S}(m))(x)$ has at most quadratic growth in $(x, y) \in \R \times [0, \infty)$ uniformly over $m$ in bounded subsets of $\P^2(\R \times [0, \infty))$ since $x \mapsto DF(p)(x)$ has at most quadratic growth in $x \in \R$ uniformly over $p$ in bounded subsets of $\cal{M}^2_{\leq 1}(\R)$. Hence, $D(F \circ \cal{S})$ is a linear functional derivative of $F \circ \cal{S}$.
\end{proof}

\begin{proposition} \label{prop:variation}
Let Assumptions \ref{ass:spfe_joint}, \ref{ass:adjoint}, and \ref{ass:smp} be satisfied. Fix $g \in \cal{G}_{\bb{F}}^2$ and a bounded $\bb{F}$-progressively measurable random function $h \define [0, T] \times \Omega \times \R^2 \to \R^{d_G}$, such that $g^{\epsilon} = g + \epsilon h$ maps into $G$ for all $\epsilon \geq 0$. Let $\mu^{\epsilon}$ be the solution to SPDE \eqref{eq:sfpe_joint} with input $g^{\epsilon}$ and define $\cal{V}^{\epsilon}_t = \frac{1}{\epsilon}(\mu^{\epsilon} - \mu)$, where $\mu = \mu^0$. Then $(\cal{V}^{\epsilon})_{\epsilon > 0}$ converges in $C_{\bb{F}}^2([0, T]; L^2(\R^2)) \cap L_{\bb{F}}^2([0, T]; H^{1, 0}(\R^2))$ to a solution $\cal{V}$ of SPDE \eqref{eq:spde_variation} such that $\supp \cal{V}_t \subset \R \times [0, \infty)$ for $\leb \otimes \pr$-a.e.\@ $(t, \omega) \in [0, T] \times \Omega$.
\end{proposition}

\begin{proof}
We set $\mathfrak{b}^{\epsilon}_t(x, y) = b_0\bigl(t, x, \nu^{\epsilon}_t, g^{\epsilon}_t(x, y)\bigr)$ for $(t, \omega, x, y) \in [0, T] \times \Omega \times \R^2$, where $\nu^{\epsilon}_t = \cal{S}(\mu^{\epsilon}_t)$. Our first objective is to derive an SPDE for $\cal{V}^{\epsilon}$. It holds for $\varphi \in H^1(\R^2)$ that
\begin{equation} \label{eq:differential_var}
    \d \langle \cal{V}^{\epsilon}, \varphi\rangle = \bigl\langle \cal{V}^{\epsilon}_t, \L_2\varphi(t, \cdot, \cdot, \nu^{\epsilon}_t, g^{\epsilon}_t)\big\rangle \, \d t + \langle \cal{V}^{\epsilon}_t, \sigma_0(t) \partial_x \varphi\rangle \, \d W^0_t + \bigl\langle \mu_t, \mathfrak{b}^{\epsilon}_t - \mathfrak{b}^0_t\bigr\rangle \, \d t.
\end{equation}
Let us compute the last term on the right-hand side further. We can decompose the difference $\mathfrak{b}^{\epsilon}_t(x, y) - \mathfrak{b}^0_t(x, y)$ as
\begin{align*}
    \bigl(\mathfrak{b}^{\epsilon}_t(x, y&) - b\bigl(t, x, \nu^{\epsilon}_t, g_t(x, y)\bigr)\bigr) + \bigl(b\bigl(t, x, \nu^{\epsilon}_t, g_t(x, y)\bigr) - \mathfrak{b}^0_t(x, y)\bigr) \\
    &= \bigl(b_1(t, x, g^{\epsilon}_t(x, y)) - b_1(t, x, g_t(x, y))\bigr) + \bigl(b_0(t, x, \nu^{\epsilon}_t) - b_0(t, x, \nu_t)\bigr).
\end{align*}
The function $b_1$ is linear in $g$ by Assumption \ref{ass:convexity}, which is in place by Assumption \ref{ass:smp}, so it holds that $b_1(t, x, g^{\epsilon}_t(x, y)) - b_1(t, x, g_t(x, y)) = \epsilon \partial_g b_1(t, x) h_t(x, y)$. Next, by Assumption \ref{ass:adjoint} \ref{it:lfd} and Lemma \ref{eq:lfd_s}, for $(t, x) \in [0, T] \times \R$ the map $\P^2(\R \times [0, \infty)) \ni m \mapsto b_0(t, x, \cal{S}(m))$ has the linear functional derivative $(x', y') \mapsto e^{- y'} Db_0(t, x, \cal{S}(m))(x')$, where $x' \mapsto Db_0(t, x, v)(x')$ denotes the linear functional derivative of the map $\cal{M}^2_{\leq 1}(\R) \ni v \mapsto Db_0(t, x, v)$. Thus, we obtain that
\begin{align} \label{eq:derivative_in_m}
    b_0(t, x, \nu^{\epsilon}_t) - b_0(t, x, \nu_t) &= \int_0^1 \Bigl\langle \mu^{\epsilon}_t - \mu_t, e^{- \cdot}Db_0\bigl(t, x, \eta \nu^{\epsilon}_t + (1 - \eta)\nu_t\bigr)\Bigr\rangle \, \d \eta \notag \\
    &= \int_0^1 \Bigl\langle \epsilon \cal{V}^{\epsilon}_t, e^{- \cdot}Db_0\bigl(t, x, \eta \nu^{\epsilon}_t + (1 - \eta)\nu_t\bigr)\Bigr\rangle \, \d \eta \notag \\
    &= \int_0^1 \bigl\langle \epsilon \cal{V}^{\epsilon}_t, B^{\epsilon, \eta}_t(x)\bigr\rangle \, \d \eta,
\end{align}
where $B^{\epsilon, \eta}_t(x)(x', y') = e^{- y'}Db_0\bigl(t, x, \eta \nu^{\epsilon}_t + (1 - \eta)\nu_t\bigr)(x')$ for $(t, \omega, x, x', y') \in [0, T] \times \Omega \times \R^3$. We insert the expressions $b_1(t, x, g^{\epsilon}_t(x, y)) - b_1(t, x, g_t(x, y)) = \epsilon \partial_g b_1(t, x) h_t(x, y)$ and \eqref{eq:derivative_in_m} into Equation \eqref{eq:differential_var}, which yields
\begin{align} \label{eq:var_eps_spde}
\begin{split}
    \d \langle \cal{V}^{\epsilon}_t, \varphi\rangle &= \bigl\langle \cal{V}^{\epsilon}_t, \L_2\varphi(t, \cdot, \cdot, \nu^{\epsilon}_t, g^{\epsilon}_t)\big\rangle \, \d t + \langle \cal{V}^{\epsilon}_t, \sigma_0(t) \partial_x \varphi\rangle \, \d W^0_t + \bigl\langle \mu_t, \partial_g b_1(t, \cdot) h_t \partial_x \varphi\bigr\rangle \, \d t \\
    &\ \ \ + \biggl(\int_0^1\biggl(\int_{\R \times [0, \infty)} \bigl\langle \cal{V}^{\epsilon}_t, B^{\epsilon, \eta}_t(x)\bigr\rangle \partial_x \varphi(x, y) \, \d \mu_t(x, y)\biggr) \d \eta\biggr) \, \d t
\end{split}
\end{align}
for $\varphi \in H^1(\R^2)$. Our goal is to show that the family $(\cal{V}^{\epsilon})_{\epsilon}$ is Cauchy in $C_{\bb{F}}^2([0, T]; L^2(\R^2)) \cap L_{\bb{F}}^2([0, T]; H^{1, 0}(\R^2))$. 

We begin with a uniform bound on $\sup_{0 \leq t \leq T} \lVert \cal{V}^{\epsilon}_t\rVert_{L^2}^2$. Clearly, the SPDE \eqref{eq:var_eps_spde} has the same form as SPDE \eqref{eq:sfpe_random} with $\beta$ and $h$ replaced by $\mathfrak{b}^{\epsilon}$ and
\begin{equation*}
    \mathfrak{h}^{\epsilon}_t(\varphi) = \bigl\langle \mu_t, \partial_g b_1(t, \cdot) h_t \partial_x \varphi\bigr\rangle \, \d \eta + \int_0^1\biggl(\int_{\R \times [0, \infty)} \bigl\langle \cal{V}^{\epsilon}_t, B^{\epsilon, \eta}_t(x)\bigr\rangle \partial_x \varphi(x, y) \, \d \mu_t(x, y)\biggr) \d \eta
\end{equation*}
for $\varphi \in H^{1, 0}(\R^2)$, respectively. Hence, we can apply the estimate \eqref{eq:sfpe_estimate} to see that
\begin{equation} \label{eq:var_eps_est}
    \sup_{0 \leq s \leq t} \lVert \cal{V}^{\epsilon}_s\rVert_{L^2}^2 + \int_0^t \lVert \cal{V}^{\epsilon}_s\rVert_{H^{1, 0}}^2 \, \d s \leq C \int_0^t \bigl(\lvert \langle \cal{V}^{\epsilon}_s, \mathfrak{b}^{\epsilon}_s \partial_x \cal{V}^{\epsilon}_s\rangle\rvert + \lvert \mathfrak{h}^{\epsilon}_s(\cal{V}^{\epsilon}_s)\rvert\bigr) \, \d s.
\end{equation}
Let us estimate the two terms on the right-hand side. For any $\delta > 0$ we find that $\lvert \langle \cal{V}^{\epsilon}_s, \mathfrak{b}^{\epsilon}_s \partial_x \cal{V}^{\epsilon}_s\rangle\rvert \leq \frac{C_b^2}{2\delta} \lVert \cal{V}^{\epsilon}_s\rVert_{L^2}^2 + \frac{\delta}{2} \lVert \cal{V}^{\epsilon}_s\rVert_{H^{1, 0}}^2$. Next, it holds that
\begin{align*}
    \lvert \mathfrak{h}^{\epsilon}_s(\cal{V}^{\epsilon}_s)\rvert &\leq \frac{C_b^2 \lVert h\rVert_{L^{\infty}}^2}{2 \delta} \lVert \mu_s\rVert_{L^2}^2 + \frac{\delta}{2}\lVert \cal{V}^{\epsilon}_s\rVert_{H^{1, 0}}^2 + \lVert \cal{V}^{\epsilon}_s\rVert_{L^2} \lVert \mu_s\rVert_{L^2} \lVert \partial_x \cal{V}^{\epsilon}_s\rVert_{L^2} \lVert h_b\rVert_{L^2} \\
    &\leq \frac{C_b^2 \lVert h\rVert_{L^{\infty}}^2 + \lVert h_b\rVert_{L^2}^2 \lVert \cal{V}^{\epsilon}_s\rVert_{L^2}^2}{2 \delta}\lVert \mu_s\rVert_{L^2}^2 + \delta\lVert \cal{V}^{\epsilon}_s\rVert_{H^{1, 0}}^2,
\end{align*}
where we used that $\bigl\langle \cal{V}^{\epsilon}_s, B^{\epsilon, \eta}_s(x)\bigr\rangle \leq \lVert \cal{V}^{\epsilon}_s\rVert_{L^2} \lVert h_b\rVert_{L^2}$ by Assumption \ref{ass:adjoint} \ref{it:lfd_l2}. Note that here $h_b$ does not refer to the random function $h \define [0, T] \times \Omega \times \R^2 \to \R^{d_G}$, but is the dominating function from Assumption \ref{ass:adjoint} \ref{it:lfd_l2}. Substituting this and the previous estimate into Equation \eqref{eq:var_eps_est}, choosing $\delta = \frac{1}{3C}$, rearranging, and finally applying Gr\"onwall's inequality implies that
\begin{equation*}
    \sup_{0 \leq t \leq T} \lVert \cal{V}^{\epsilon}_t\rVert_{L^2}^2 + \frac{1}{2}\int_0^T \lVert \cal{V}^{\epsilon}_t\rVert_{H^{1, 0}}^2 \, \d t \leq C
\end{equation*}
for a possibly enlarged constant $C$ independent of $\epsilon$.

Using the uniform bound on the family $(\cal{V}^{\epsilon})_{\epsilon > 0}$ we can show that it is Cauchy in $C_{\bb{F}}^2([0, T]; L^2(\R^2)) \cap L_{\bb{F}}^2([0, T]; H^{1, 0}(\R^2))$. Indeed, let us define $\Delta^{\epsilon, \delta} = \cal{V}^{\epsilon} - \cal{V}^{\delta}$. Note that again, the SPDE satisfied by $\Delta^{\epsilon, \delta}$ is of the same form as SPDE \eqref{eq:sfpe_random}. This time with $\beta$ and $h$ replaced by $\mathfrak{b}^{\epsilon}$ and
\begin{align*}
    \mathfrak{h}^{\epsilon, \delta}_t(\varphi) &= \mathfrak{h}^{\epsilon, \delta, 1}_t(\varphi) + \mathfrak{h}^{\epsilon, \delta, 2}_t(\varphi) + \mathfrak{h}^{\epsilon, \delta, 3}_t(\varphi)\\
    &= \bigl\langle \cal{V}^{\delta}_t, (\mathfrak{b}^{\epsilon}_t - \mathfrak{b}^{\delta}_t) \partial_x\varphi\bigr\rangle \\
    &\ \ \ + \int_0^1\biggl(\int_{\R \times [0, \infty)} \bigl\langle \Delta^{\epsilon, \delta}_t, B^{\epsilon, \eta}_t(x)\bigr\rangle \partial_x \varphi(x, y) \, \d \mu_t(x, y)\biggr) \d \eta \\
    &\ \ \ + \int_0^1\biggl(\int_{\R \times [0, \infty)} \bigl\langle \cal{V}^{\delta}_t, B^{\epsilon, \eta}_t(x) - B^{\epsilon, \eta}_t(x)\bigr\rangle \partial_x \varphi(x, y) \, \d \mu_t(x, y)\biggr) \d \eta.
\end{align*}
for $\varphi \in H^{1, 0}(\R^2)$, respectively. As before, we bound the terms $\lvert \langle \Delta^{\epsilon, \delta}_s, \mathfrak{b}^{\epsilon}_s \partial_x \Delta^{\epsilon, \delta}_s\rangle\rvert$ and $\lvert \mathfrak{h}^{\epsilon, \delta}_s(\Delta^{\epsilon, \delta}_s)\rvert$ appearing in the estimate for $\Delta^{\epsilon, \delta}$ provided by \eqref{eq:sfpe_estimate}. Analogous to the above, for any $\delta > 0$, we have $\lvert \langle \Delta^{\epsilon, \delta}_s, \mathfrak{b}^{\epsilon}_s \partial_x \Delta^{\epsilon, \delta}_s\rangle\rvert \leq \frac{C_b^2}{2\delta} \lVert \Delta^{\epsilon, \delta}_s\rVert_{L^2}^2 + \frac{\delta}{2} \lVert \Delta^{\epsilon, \delta}_s\rVert_{H^{1, 0}}^2$. Next, by the triangle inequality $\lvert \mathfrak{h}^{\epsilon, \delta}_s(\Delta^{\epsilon, \delta}_s)\rvert \leq \lvert \mathfrak{h}^{\epsilon, \delta, 1}_s(\Delta^{\epsilon, \delta}_s)\rvert + \lvert \mathfrak{h}^{\epsilon, \delta, 2}_s(\Delta^{\epsilon, \delta}_s)\rvert + \lvert \mathfrak{h}^{\epsilon, \delta, 3}_s(\Delta^{\epsilon, \delta}_s)\rvert$ and we estimate the three terms on the right-hand side separately. For the first, we get
\begin{align} \label{eq:theta_1_est}
    \lvert \mathfrak{h}^{\epsilon, \delta, 1}_s(\Delta^{\epsilon, \delta}_s)\rvert &\leq \frac{1}{2 \delta} \lVert \cal{V}^{\delta}_s (\mathfrak{b}^{\epsilon}_s - \mathfrak{b}^{\delta}_s) \rVert_{L^2}^2 + \frac{\delta}{2} \lVert \Delta^{\epsilon, \delta}_s\rVert_{H^{1, 0}}^2 \notag \\
    &\leq \frac{C_b^2\lVert \cal{V}^{\delta}_s\rVert_{L^2}^2}{2 \delta} \bigl(\lVert \mu^{\epsilon}_s - \mu^{\delta}_s\rVert_{L^2}^2 + \lVert h_s\rVert_{L^{\infty}}^2(\epsilon - \delta)^2 \bigr) + \frac{\delta}{2} \lVert \Delta^{\epsilon, \delta}_s\rVert_{H^{1, 0}}^2,
\end{align}
where we used the Lipschitz continuity of $b$ in $v$ (cf.\@ Assumption \ref{ass:mfl} \ref{it:continuity_coeff_mfl}, which holds by Assumption \ref{ass:spfe_joint}), the linearity of $b$ in the control argument, and that $\lVert \nu^{\epsilon}_s - \nu^{\delta}_s\rVert_{L^2} \leq \lVert \mu^{\epsilon}_s - \mu^{\delta}_s\rVert_{L^2}$. 
For the second term, we estimate
\begin{align} \label{eq:theta_2_est}
    \lvert \mathfrak{h}^{\epsilon, \delta, 2}_s(\Delta^{\epsilon, \delta}_s)\rvert &\leq \lVert \Delta^{\epsilon, \delta}_s\rVert_{L^2} \lVert \mu_s\rVert_{L^2} \lVert \partial_x \Delta^{\epsilon, \delta}_s\rVert_{L^2} \lVert h_b\rVert_{L^2}^2 \notag \\
    &\leq \frac{\lVert h_b\rVert_{L^2}^2 \lVert \mu_s\rVert_{L^2}^2}{2\delta}\lVert \Delta^{\epsilon, \delta}_s\rVert_{L^2}^2 + \frac{\delta}{2}\lVert \Delta^{\epsilon, \delta}_s\rVert_{H^{1, 0}}^2,
\end{align}
where as above we used that $\bigl\langle \cal{V}^{\epsilon}_s, B^{\epsilon, \eta}_s(x)\bigr\rangle \leq \lVert \cal{V}^{\epsilon}_s\rVert_{L^2} \lVert h_b\rVert_{L^2}$ by Assumption \ref{ass:adjoint} \ref{it:lfd_l2}. Lastly, it holds that
\begin{align} \label{eq:theta_3_est}
    \lvert \mathfrak{h}^{\epsilon, \delta, 3}_s(\Delta^{\epsilon, \delta}_s)\rvert &\leq \frac{\lVert \cal{V}^{\delta}_s\rVert_{L^2}^2}{2 \delta} \int_0^1 \biggl(\int_{\R \times [0, \infty)} \bigl\lVert B^{\epsilon, \eta}_s(x) - B^{\delta, \eta}_s(x)\bigr\rVert_{L^2}^2 \lvert \mu_s(x, y)\rvert^2 \, \d x \d y\biggr) \, \d \eta \notag \\
    &\ \ \ + \frac{\delta}{2}\lVert \Delta^{\epsilon, \delta}_s\rVert_{H^{1, 0}}^2 \notag \\
    &= \frac{\lVert \cal{V}^{\delta}_s\rVert_{L^2}^2}{2 \delta} I^{\epsilon, \delta}_s + \frac{\delta}{2}\lVert \Delta^{\epsilon, \delta}_s\rVert_{H^{1, 0}}^2.
\end{align}
We fix $\delta = \frac{1}{4C}$ and insert \eqref{eq:theta_1_est}, \eqref{eq:theta_2_est}, \eqref{eq:theta_3_est}, 
and the bound on $\lvert \langle \Delta^{\epsilon, \delta}_s, \mathfrak{b}^{\epsilon}_s \partial_x \Delta^{\epsilon, \delta}_s\rangle\rvert$ into the estimate for $\Delta^{\epsilon, \delta}$ provided by \eqref{eq:sfpe_estimate}. Then we rearrange and apply Gr\"onwall's inequality to obtain
\begin{align} \label{eq:delta_var_estimate}
\begin{split}
    \sup_{0 \leq t \leq T} \lVert \Delta^{\epsilon, \delta}_s\rVert_{L^2}^2 + \frac{1}{2}\int_0^T \lVert \Delta^{\epsilon, \delta}_s\rVert_{H^{1, 0}}^2 \, \d t \leq C \int_0^T \Bigl(\lVert \mu^{\epsilon}_t - \mu^{\delta}_t\rVert_{L^2}^2 + (\epsilon - \delta)^2 + I^{\delta, \epsilon}_t \Bigr) \, \d t
\end{split}
\end{align}
for a possibly enlarged $C$. Let us show that the right-hand side tends to zero as $\epsilon$, $\delta \to 0$. First, we have that $\lVert \mu^{\epsilon}_t - \mu^{\delta}_t\rVert_{L^2} \leq \epsilon \lVert \cal{V}^{\epsilon}_t\rVert_{L^2} + \delta \lVert \cal{V}^{\delta}_t\rVert_{L^2}$, which is bounded uniformly in $\epsilon$ and $\delta$ and tends to zero as $\epsilon$, $\delta \to 0$. Hence, it remains to take care of the quantity $I^{\delta, \epsilon}_t$. It is bounded by $4 \lVert h_b\rVert_{L^2}^2 \lVert \mu_t\rVert_{L^2}^2$, so if we prove that $I^{\delta, \epsilon}_t$ vanishes as $\epsilon$, $\delta \to 0$, then the dominated convergence theorem allows us to conclude the same for the right-hand side of \eqref{eq:delta_var_estimate}. By Assumption \ref{ass:smp} \ref{it:smp_cont} the map $\cal{M}^2_{\leq 1}(\R) \cap L^2(\R) \ni v \mapsto Db_0(t, x, v)(x')$ is continuous with respect to $\lVert \cdot \rVert_{L^2}$, so by definition of $B^{\epsilon, \eta}_t$ we have $\bigl\lvert B^{\epsilon, \eta}_t(x)(x', y') - B^{\delta, \eta}_t(x)(x', y')\bigr\rvert \to 0$ as $\epsilon$, $\delta \to 0$. Since by Assumption \ref{ass:adjoint} \ref{it:lfd_l2} it holds that $\lvert B^{\epsilon, \eta}_t(x)(x', y')\rvert \leq e^{-y'} h_b(x')$ and $e^{-\cdot} h_b \in L^2(\R \times [0, \infty])$, we can conclude that $\bigl\lVert B^{\epsilon, \eta}_s(x) - B^{\delta, \eta}_s(x)\bigr\rVert_{L^2}^2 \to 0$. This readily implies $I^{\epsilon, \delta}_t \to 0$ as $\epsilon$ and $\delta$ are taken to zero. Thus, it follows that $(\cal{V}^{\epsilon})_{\epsilon > 0}$ is Cauchy in $C_{\bb{F}}^2([0, T]; L^2(\R^2)) \cap L_{\bb{F}}^2([0, T]; H^{1, 0}(\R^2))$.

Let $\cal{V}$ denote the limit of $(\cal{V}^{\epsilon})_{\epsilon > 0}$ in $C_{\bb{F}}^2([0, T]; L^2(\R^2)) \cap L_{\bb{F}}^2([0, T]; H^{1, 0}(\R^2))$. Using the convergence of $\cal{V}^{\epsilon}$ to $\cal{V}$, we can pass to the limit in SPDE \eqref{eq:var_eps_spde} to see that $\cal{V}$ solves SPDE \eqref{eq:spde_variation}. Since $\mu^{\epsilon}_t$ and $\mu_t$ are supported in $\R \times [0, \infty)$ the same holds for $\cal{V}^{\epsilon}_t$ and its limit $\cal{V}_t$. This concludes the proof.
\end{proof}

\subsection{Necessary Stochastic Maximum Principle} \label{sec:smp}

The goal of this subsection is to establish the necessary SMP. For that, we first show the existence of an adjoint process satisfying BSPDE \eqref{eq:adjoint}, i.e.\@ we prove Proposition \ref{prop:adjoint_existence}. Next, we use the adjoint process and the variation $\cal{V}$ from the previous subsection to derive an expression for the derivative of the cost $J_{\text{cl}}$ defined in \eqref{eq:cost_func_closed}. Inserting the convexity of $h \mapsto b(t, x, v, h)p + e^{-y}f(t, x, v, h)$ implied by Assumption \ref{ass:convexity} into this expression yields the necessary SMP.

\begin{proof}[Proof of Proposition \ref{prop:adjoint_existence}]
We intend to apply Theorem \ref{thm:bspde}, so we simply have to show that the coefficients $\tilde{a}$, $\tilde{H}$, $\tilde{F}$, $\lambda$, $\sigma_0$, and $\tilde{\beta}$ as well as the terminal condition $\tilde{\psi}$ defined by $\tilde{a}_t(x, y) = a_t(x)$, $\tilde{H}_t(x, y, r, p) = \tilde{K}^{\mu}_t\bigl(x, y, p, g_t(x, y)\bigr)$, $\tilde{F}_t(x, y, u, v) = \tilde{F}^{\mu}_t(x, y, v)$, $\tilde{\beta}_t(x, y) = 0$, and $\psi(x, y) = e^{-y}D\psi(\nu_T)(x)$ for $(t, \omega, x, y, r, p, u, v) \in [0, T] \times \Omega \times \R \times (\ell, \infty) \times \R^2 \times L^2(\R \times (\ell, \infty))^2$ satisfy the properties stated for $a$, $H$, $F$, $\lambda$, $\alpha$, $\beta$, and $\psi$ in Assumption \ref{ass:bspde}. This is obvious for $\tilde{a}$, $\lambda$, $\sigma_0$, $\beta$, and $\psi$ so we will concentrate on $\tilde{H}$ and $\tilde{F}$. Let us begin with the growth conditions Assumption \ref{ass:bspde} \ref{it:growth_bspde}. We have for $(t, \omega, x, y, r, p, u, v) \in [0, T] \times \Omega \times \R \times (\ell, \infty) \times \R^2 \times L^2(\R \times (\ell, \infty))^2$ that
\begin{align*}
    \lvert \tilde{H}_t(x, y, r, p)\rvert &\leq \bigl\lvert b\bigl(t, x, \nu_t, g_t(x, y)\bigr)p\bigr\rvert + \bigl\lvert e^{-y} f\bigl(t, x, \nu_t, g_t(x, y)\bigr)\bigr\rvert \\
    &\leq C_b \lvert p\rvert + e^{-y} h_f(x)
\end{align*}
for $h_f \in L^2(\R)$ by Assumption \ref{ass:adjoint} \ref{it:cost_l2} and
\begin{align*}
    \lvert \tilde{F}_t(x, y, u, v)\rvert &\leq \bigl\lvert \bigl\langle \mu_t, e^{-y} Db_0(t, \cdot, \nu_t)(x) v\bigr\rangle\bigr\rvert + \bigl\lvert \bigl\langle \mu_t, e^{-y} D\tilde{f}_0(t, \cdot, \cdot, \nu_t)(x) \bigr\rangle\bigr\rvert \\
    &\leq e^{-y}h_b(x) \lVert \mu_t\rVert_{L^2} \lVert v\rVert_{L^2} + e^{-y}h_f(x)
\end{align*}
for $h_b \in L^2(\R)$ by Assumption \ref{ass:adjoint} \ref{it:lfd_l2}. Since $\lVert \mu_t\rVert_{L^2} \leq C \lVert \mu_0\rVert_{L^2} < \infty$ by Proposition \ref{prop:sfpe_joint}, the right-hand side is of the desired form. Next, we verify the Lipschitz conditions from Assumption \ref{ass:bspde} \ref{it:lipschitz_bspde}. For $(t, \omega, x, y, r, p, u, v)$ as above and $(r', p', u', v') \in \R^2 \times L^2(\R \times (\ell, \infty))^2$, it holds that
\begin{align*}
    \lvert \tilde{H}_t(x, y, r, p) - \tilde{H}_t(x, y, r', p')\rvert = \bigl\lvert b\bigl(t, x, \nu_t, g_t(x, y)\bigr)p - b\bigl(t, x, \nu_t, g_t(x, y)\bigr)p'\bigr\rvert \leq C_b \lvert p - p'\rvert
\end{align*}
by the boundedness of $b$ and
\begin{align*}
    \lvert \tilde{F}_t(x, y, u, v) - \tilde{F}_t(x, y, u', v')\rvert &= \bigl\lvert \bigl\langle \mu_t, e^{-y} Db_0(t, \cdot, \nu_t)(x) v\bigr\rangle - \bigl\langle \mu_t, e^{-y} Db_0(t, \cdot, \nu_t)(x) v'\bigr\rangle\bigr\rvert \\
    &\leq e^{-y}h_b(x)\lVert \mu_t\rVert_{L^2} \lVert v - v'\rVert_{L^2}
\end{align*}
by Assumption \ref{ass:adjoint} \ref{it:lfd_l2}. These bounds are as desired, so we can apply Theorem \ref{thm:bspde}.
\end{proof}

We will use the adjoint process $(\tilde{u}, \tilde{q}, \tilde{m})$ determined by BSPDE \eqref{eq:adjoint} to derive an expression for the G\^ateaux derivative of the cost function $J_{\text{cl}}$.

\begin{proposition} \label{eq:cost_derivative}
Let Assumptions \ref{ass:spfe_joint}, \ref{ass:adjoint}, and \ref{ass:smp} be satisfied. Then for any control $g \in \cal{G}_{\bb{F}}^2$ and any $\bb{F}$-progressively measurable random function $h \define [0, T] \times \Omega \times \R^2 \to \R^{d_G}$, such that $g + \epsilon h$ maps into $G$, it holds that
\begin{equation} \label{eq:gateaux_derivative}
    \frac{\d}{\d \epsilon} J_{\textup{cl}}(g + \epsilon h)\Big\vert_{\epsilon = 0} = \ev \int_0^T \bigl\langle \mu_t, \partial_g \tilde{K}^{\mu}_t(\cdot, \cdot, \partial_x \tilde{u}_t, g_t)h_t\bigr\rangle \, \d t.
\end{equation}
\end{proposition}

\begin{proof}
Set $g^{\epsilon} = g + \epsilon h$ and denote the solutions to SPDE \eqref{eq:sfpe_joint} corresponding to the inputs $g^{\epsilon}$ and $g$ by $\mu^{\epsilon}$ and $\mu$, respectively. Further, let $\cal{V}^{\epsilon} = \frac{1}{\epsilon}(\mu^{\epsilon} - \mu)$ and let $\cal{V}$ be the limit of $(\cal{V}^{\epsilon})_{\epsilon > 0}$ in $C_{\bb{F}}^2([0, T]; L^2(\R^2)) \cap L_{\bb{F}}^2([0, T]; H^{1, 0}(\R^2))$, which exists by Proposition \ref{prop:variation}. Then, we decompose
\begin{align} \label{eq:decomp_cost}
\begin{split}
    J_{\textup{cl}}(g^{\epsilon}&) - J_{\textup{cl}}(g) \\
    &= \ev \int_0^T \bigl\langle \mu^{\epsilon}_t - \mu_t, \tilde{f}(t, \cdot, \cdot, \nu^{\epsilon}_t, g^{\epsilon}_t)\bigr\rangle \, \d t + \ev\int_0^T \bigl\langle \mu_t, \tilde{f}_1(t, \cdot, \cdot, g^{\epsilon}_t) - \tilde{f}_1(t, \cdot, \cdot, g_t)\bigr\rangle \, \d t \\
    &\ \ \ + \ev\int_0^T \bigl\langle \mu_t, \tilde{f}_0(t, \cdot, \cdot, \nu^{\epsilon}_t) - \tilde{f}_0(t, \cdot, \cdot, \nu_t)\bigr\rangle \, \d t + \ev[\psi(\nu^{\epsilon}_T) - \psi(\nu_T)].
\end{split}
\end{align}
We divide the equation by $\epsilon$ and then take the limit as $\epsilon$ tends to zero. Owing to the continuity of the map $(x, v, g') \mapsto f(t, x, v, g')$ guaranteed by Assumption \ref{ass:mfl} \ref{it:continuity_cost_mfl} (which is in place by Assumption \ref{ass:spfe_joint}) as well as the growth condition on $f$ from Assumption \ref{ass:adjoint} \ref{it:cost_l2}, we can pass to the limit in the first expression on the right-hand side above. For the second term, we use Assumption \ref{ass:smp} \ref{it:smp_ext} and apply the fundamental theorem of calculus, whereby
\begin{align*}
    \tilde{f}_1\bigl(t, x, y, g^{\epsilon}_t(x, y)\bigr) - \tilde{f}_1\bigl(t, x, y, g_t(x, y)\bigr) = \epsilon\int_0^1 \partial_g f_1\bigl((t, x, g_t(x, y) + \epsilon \eta h_t(x, y)\bigr) h_t(x, y) \, \d \eta.
\end{align*}
Since $g' \mapsto \partial_g f_1(t, x, g')$ is continuous and bounded, we conclude by the dominated convergence theorem that the second expression tends to $\ev\int_0^T \bigl\langle \mu_t, \partial_g \tilde{f}_1(t, \cdot, \cdot, g_t)h_t\bigr\rangle \, \d t$. Next, by Assumption \ref{ass:adjoint} \ref{it:lfd} and Lemma \ref{eq:lfd_s}, for $(t, x, y) \in [0, T] \times \R \times [0, \infty)$ the map $\P^2(\R \times [0, \infty)) \ni m \mapsto \tilde{f}_0(t, x, y, \cal{S}(m))$ has the linear functional derivative $(x', y') \mapsto e^{- y'} D\tilde{f}_0(t, x, y, \cal{S}(m))(x')$, where $x' \mapsto D\tilde{f}_0(t, x, y, v)(x')$ denotes the linear functional derivative of the map $\cal{M}^2_{\leq 1}(\R) \ni v \mapsto D\tilde{f}_0(t, x, y, v)$. Hence, we can write
\begin{align*}
    \tilde{f}_0(t, x, y, \nu^{\epsilon}_t) - \tilde{f}_0(t, x, y, \nu_t) &= \int_0^1 \Bigl\langle \mu^{\epsilon}_t - \mu_t, e^{- \cdot}D\tilde{f}_0\bigl(t, x, y, \eta \nu^{\epsilon}_t + (1 - \eta)\nu_t\bigr)\Bigr\rangle \, \d \eta.
\end{align*}
Since $x' \mapsto D\tilde{f}_0(t, x, y, v)(x')$ is bounded by a function $h_f \in L^2(\R)$ by Assumption \ref{ass:adjoint} \ref{it:lfd_l2} and $\cal{M}^2_{\leq 1}(\R) \cap L^2(\R) \ni v \mapsto D\tilde{f}_0(t, x, y, v)(x')$ is continuous with respect to $\lVert \cdot \rVert_{L^2}$ by Assumption \ref{ass:smp} \ref{it:smp_cont}, divided by $\epsilon$ the quantity on the right-hand side above approaches $\bigl\langle \cal{V}_t, e^{- \cdot} D\tilde{f}_0(t, x, y, \nu_t)\bigr\rangle$. Consequently, the third summand in \eqref{eq:decomp_cost} converges to $\ev\int_0^T \int_{\R \times [0, \infty)} \bigl\langle \cal{V}_t, e^{- \cdot} D\tilde{f}_0(t, x, y, \nu_t)\bigr\rangle \, \d \mu_t(x, y) \, \d t$. The same reason allows us to pass to the limit as $\epsilon \to 0$ in the last term on the right-hand side of Equation \eqref{eq:decomp_cost}, so putting everything together, we obtain
\begin{align} \label{eq:gateaux_der_1}
\begin{split}
    \frac{\d}{\d \epsilon} &J_{\textup{cl}}(g + \epsilon h)\Big\vert_{\epsilon = 0} \\
    &= \ev\int_0^T \bigl\langle \cal{V}_t, \tilde{f}(t, \cdot, \cdot, \nu_t, g_t)\bigr\rangle \, \d t + \ev\int_0^T \bigl\langle \mu_t, \partial_g \tilde{f}_1(t, \cdot, \cdot, g_t)h_t\bigr\rangle \, \d t \\
    &\ \ \ + \ev\int_0^T \int_{\R \times [0, \infty)} \bigl\langle \cal{V}_t, e^{- \cdot} D\tilde{f}_0(t, x, y, \nu_t)\bigr\rangle \, \d \mu_t(x, y) \, \d t + \ev \bigl\langle \cal{V}_T, e^{-\cdot} D\psi(\nu_T)\bigr\rangle.
\end{split}
\end{align}

Next, we apply the generalisation of It\^o's formula from Theorem \ref{thm:ito_2_dim} with $\ell = -1$ to $\langle \cal{V}_t, \tilde{u}_t\rangle$. This implies that
\begin{align} \label{eq:var_adjoint}
    \bigl\langle \cal{V}_T, e^{-\cdot} &D\psi(\nu_T)\bigr\rangle \notag \\
    &= \langle \cal{V}_0, \tilde{u}_0\rangle + \int_0^T \Bigl\langle b(t, \cdot, \nu_t, g_t)\cal{V}_t - \partial_x (a_t \cal{V}_t) + \partial_g b_1(t, \cdot) h_t \mu_t,  \partial_x \tilde{u}_t \Bigr\rangle \, \d t \notag \\
    &\ \ \ + \int_0^T \biggl(\int_{\R \times [0, \infty)} \bigl\langle \cal{V}_t, e^{- \cdot}Db_0(t, x, \nu_t)\bigr\rangle \partial_x \tilde{u}_t(x, y) \, \d \mu_t(x, y)\biggr) \, \d t \notag \\
    &\ \ \ - \int_0^T \Bigl(\bigl\langle b_t(t, \cdot, \nu_t, g_t) \partial_x\tilde{u}_t + \tilde{f}(t, \cdot, \cdot, \nu_t, g_t) + \tilde{F}^{\mu}_t(\cdot, \cdot, \partial_x \tilde{u}_t), \cal{V}_t\bigr\rangle\Bigr) \notag \\
    & \ \ \ + \int_0^T \Bigl(\langle \partial_x \tilde{u}_t, \partial_x (a_t\cal{V}_t)\rangle - \langle \tilde{q}_t, \sigma_0(t)\partial_x \cal{V}_t\rangle\Bigr) \, \d t + \int_0^T \langle \partial_x \cal{V}_t, \sigma_0(t)\tilde{q}_t\rangle \, \d t \notag \\
    &\ \ \ + \int_0^T \bigl(\langle \cal{V}_t, \sigma_0(t) \partial_x \tilde{u}_t\rangle + \langle \tilde{q}_s, \cal{V}_t\rangle\bigr) \, \d W_t + \langle \tilde{m}_T, \cal{V}_T\rangle  \notag \\
    &= \int_0^T \Bigl(\bigl\langle \mu_t, \partial_g b_1(t, \cdot) h_t \partial_x \tilde{u}_t \bigr\rangle - \langle \cal{V}_t, \tilde{f}(t, \cdot, \cdot, \nu_t, g_t)\rangle\Bigr) \, \d t \\
    &\ \ \ - \int_0^T \biggl(\int_{\R \times [0, \infty)} \bigl\langle \cal{V}_t, e^{- \cdot}D\tilde{f}_0\bigl(t, x, \nu_t, g_t(x, y)\bigr)\bigr\rangle \, \d \mu_t(x, y)\biggr) \, \d t \notag \\
    &\ \ \ + \int_0^T \bigl(\langle \cal{V}_t, \sigma_0(t) \partial_x \tilde{u}_t\rangle + \langle \tilde{q}_s, \cal{V}_t\rangle\bigr) \, \d W_t + \langle \tilde{m}_T, \cal{V}_T\rangle \notag,
\end{align}
where we used that $\cal{V}_0 = 0$ and that
\begin{align*}
    \tilde{F}^{\mu}_t(x', y', \partial_x \tilde{u}_t) - \int_{\R \times [0, \infty)} e^{-y'} &Db_0(t, x, \nu_t)(x')\partial_x \tilde{u}_t(x, y) \, \d \mu_t(x, y) \\
    &= \int_{\R \times [0, \infty)} e^{-y'} D\tilde{f}_0(t, x, y, \nu_t)(x')\partial_x \tilde{u}_t(x, y) \, \d \mu_t(x, y)
\end{align*}
in the second equality. Now, we take expectations on both sides of \eqref{eq:var_adjoint}, so that the martingale terms disappear, and then plug the expression for $\ev \bigl\langle \cal{V}_T, e^{-\cdot} D\psi(\nu_T)\bigr\rangle$ provided by the right-hand side of \eqref{eq:var_adjoint} into Equation \eqref{eq:gateaux_der_1}. This implies that
\begin{equation*}
    \frac{\d}{\d \epsilon} J_{\textup{cl}}(g + \epsilon h)\Big\vert_{\epsilon = 0} = \ev\int_0^T \Bigl\langle \mu_t, \partial_g b_1(t, \cdot, \nu_t) h_t \partial_x \tilde{u}_t + \partial_g \tilde{f}_1(t, \cdot, \cdot, \nu_t)h_t\Bigr\rangle \, \d t,
\end{equation*}
which yields the desired expression \eqref{eq:gateaux_derivative}.
\end{proof}

From \eqref{eq:gateaux_derivative} and the convexity of the Hamiltonian in the control argument, guaranteed by Assumption \ref{ass:convexity} (which is in place by Assumption \ref{ass:smp}), we can deduce the necessary SMP.

\begin{proof}[Proof of Theorem \ref{thm:smp}]
Let $g$, $h \in \cal{G}_{\bb{F}}^2$. Since $G$ is convex, for any $\epsilon > 0$ the random function $g^{\epsilon} = (1 - \epsilon) g + \epsilon h = g + \epsilon (h - g)$ maps into $G$. Hence, we can apply Proposition \ref{eq:cost_derivative} with $h$ replaced by $h - g$, which implies that
\begin{align*}
    \frac{\d}{\d \epsilon} J_{\textup{cl}}(g^{\epsilon})\Big\vert_{\epsilon = 0} &= \ev \int_0^T \bigl\langle \mu_t, \partial_g \tilde{K}^{\mu}_t(\cdot, \cdot, \partial_x \tilde{u}_t, g_t)(h_t - g_t)\bigr\rangle \, \d t \\
    &\leq \ev \int_0^T \bigl\langle \mu_t, \tilde{K}^{\mu}_t(\cdot, \cdot, \partial_x \tilde{u}_t, h_t) - \tilde{K}^{\mu}_t(\cdot, \cdot, \partial_x \tilde{u}_t, g_t)\bigr\rangle \, \d t,
\end{align*}
where we used that the function $h' \mapsto \tilde{K}^{\mu}_t\bigl(x, y, \partial_x \tilde{u}_t(x, y), h'\bigr)$ is convex by Assumption \ref{ass:convexity}. On the other hand, we know that $g$ is a minimiser of the cost functional $J_{\text{cl}}$, so that
\begin{equation*}
    \frac{\d}{\d \epsilon} J_{\textup{cl}}(g^{\epsilon})\Big\vert_{\epsilon = 0} = \lim_{\epsilon \to 0}\frac{1}{\epsilon} \bigl(J_{\textup{cl}}(g + \epsilon(h - g)) - J_{\textup{cl}}(g)\bigr) \geq 0.
\end{equation*}
Combining the previous two inequalities yields
\begin{equation} \label{eq:conv_ineq_cost}
    \ev \int_0^T \bigl\langle \mu_t, \tilde{K}^{\mu}_t(\cdot, \cdot, \partial_x \tilde{u}_t, h_t) - \tilde{K}^{\mu}_t(\cdot, \cdot, \partial_x \tilde{u}_t, g_t)\bigr\rangle \, \d t \geq 0.
\end{equation}
Now, for any $\bb{F}$-progressively measurable set $A \subset [0, T] \times \Omega \times \R^2$ and any $h_{\ast} \in G$ we can define an $h \in \cal{G}_{\bb{F}}^2$ by $h_t(x, y) = h_{\ast}$ if $(t, \omega, x, y) \in A$ and $h_t(x, y) = g_t(x, y)$ otherwise. Applying \eqref{eq:conv_ineq_cost} with this choice of $h$ implies that
\begin{equation*}
    \ev\int_0^T \int_{\R^2} \bf{1}_A(t, \cdot, x, y) \Bigl(\tilde{K}^{\mu}_t\bigl(x, y, \partial_x \tilde{u}_t(x, y), h_{\ast}\bigr) - \tilde{K}^{\mu}_t\bigl(x, y, \partial_x \tilde{u}_t(x, y), g_t(x, y)\bigr)\Bigr) \, \d \mu_t(x, y) \, \d t
\end{equation*}
is nonnegative. Since $A$ and $h_{\ast}$ were arbitrary, we conclude that for all $h_{\ast} \in G$ it holds that
\begin{equation*}
    \tilde{K}^{\mu}_t\bigl(x, y, \partial_x \tilde{u}_t(x, y), h_{\ast}\bigr) \geq \tilde{K}^{\mu}_t\bigl(x, y, \partial_x \tilde{u}_t(x, y), g_t(x, y)\bigr)
\end{equation*}
for $\mu_t$-a.e.\@ $(x, y) \in \R \times [0, \infty)$ and $\leb \otimes \pr$-a.e.\@ $(t, \omega) \in [0, T] \times \Omega$.
\end{proof}



\subsection{Existence and Uniqueness of the Semilinear BSPDEs \texorpdfstring{\eqref{eq:adjoint_1d}}{} and \texorpdfstring{\eqref{eq:bspde_2d}}{}} \label{sec:bspde_1_2}

In this section we establish the existence of a unique strong solution $(u, q, m)$ of BSPDE \eqref{eq:adjoint_1d}. From this process $(u, q, m)$ we construct the unique strong solution $(\tilde{u}, \tilde{q}, \tilde{m})$ of BSPDE \eqref{eq:bspde_2d} by setting $\tilde{u}_t(x, y) = e^{-y} u_t(x)$, $\tilde{q}_t(x, y) = e^{-y} q_t(x)$, and $\tilde{m}_t(x, y) = e^{-y} m_t(x)$ for $(t, \omega, x, y) \in [0, T] \times \Omega \times \R \times (0, \infty)$.

\begin{proof}[Proof of Proposition \ref{prop:bspde_1}]
We apply Theorem \ref{thm:bspde}, so we simply have to verify Assumption \ref{ass:bspde} for the coefficients $a$, $H$, $F$, $\alpha$, and $\beta$ defined by $H_t(x, r, p) = H^{\nu}_t(x, r, p)$, $F_t(x, u, v) = F^{\nu}_t(x, u, v)$, $\alpha_t = \sigma_0(t)$, and $\beta_t(x) = 0$ for $(t, \omega, x, r, p, u, v) \in [0, T] \times \Omega \times \R^3 \times L^2(\R)^2$. We only check the growth and Lipschitz conditions, Assumptions \ref{ass:bspde} \ref{it:growth_bspde} and \ref{it:lipschitz_bspde}, for $H$. The coefficient $F$ can be dealt with in the same manner as with $\tilde{F}$ in the proof of Proposition \ref{prop:adjoint_existence} and the properties of the remaining coefficients are obvious. For $(t, \omega, x, r, p) \in [0, T] \times \Omega \times \R^3$ it holds that
\begin{align*}
    \lvert H_t(x, r, p)\rvert &\leq \biggl\lvert \inf_{g \in G} \bigl(b(t, x, \nu_t, g)p + f(t, x, \nu_t, g)\bigr)\biggr\rvert + \lvert \lambda_t(x)r\rvert \\
    &\leq C_b \lvert p\rvert + h_f(x) + C_{\lambda}\lvert r\vert,
\end{align*}
where the boundedness of $b$ and $\lambda$ is guaranteed by Assumption \ref{ass:spfe_joint} \ref{it:spde_growth} and $h_f \in L^2(\R)$ by Assumption \ref{ass:adjoint} \ref{it:cost_l2}, so that $H$ satisfies Assumption \ref{ass:bspde} \ref{it:growth_bspde}. Next, if in addition $(r', p') \in \R^2$, then
\begin{align*}
    \lvert H_t&(x, r, p) - H_t(x, r', p')\rvert \\
    &\leq \biggl\lvert\inf_{g \in G} \bigl(b(t, x, \nu_t, g)p + f(t, x, \nu_t, g)\bigr) - \inf_{g \in G} \bigl(b(t, x, \nu_t, g)p' + f(t, x, \nu_t, g)\bigr)\biggr\rvert \\
    &\ \ \ + \lvert \lambda_t(x) (r - r')\rvert \\
    &\leq \sup_{g \in G} \Bigl\lvert b(t, x, \nu_t, g)p + f(t, x, \nu_t, g) - \bigl(b(t, x, \nu_t, g)p' + f(t, x, \nu_t, g)\bigr)\Bigr\rvert + C_{\lambda}\lvert r - r'\rvert \\
    &\leq C_b\lvert p - p'\rvert + C_{\lambda}\lvert r - r'\rvert,
\end{align*}
so that $H$ also satisfies the Lipschitz condition from Assumption \ref{ass:bspde} \ref{it:lipschitz_bspde}. Thus, we can apply Theorem \ref{thm:bspde} to obtain a unique strong solution to BSPDE \eqref{eq:adjoint}.
\end{proof}

\begin{proof}[Proof of Proposition \ref{prop:bspde_2_from_1}]
We first show that the process $(\tilde{u}, \tilde{q}, \tilde{m}) \in \bb{H}(\R \times (0, \infty))$ given by $\tilde{u}_t(x, y) = e^{-y}u_t(x)$, $\tilde{q}_t(x, y) = e^{-y}q_t(x)$, $\tilde{m}_t(x, y) = e^{-y}m_t(x)$ for $(t, \omega, x, y) \in [0, T] \times \Omega \times \R \times (0, \infty)$, where $(u, q, m)$ is the unique strong solution of BSPDE \eqref{eq:adjoint_1d}, solves BSPDE \eqref{eq:bspde_2d}. Let $\varphi \in H^1_0(\R \times (0, \infty))$ and define $\varphi_0 \in H^1(\R)$ by $\varphi_0(x) = \int_{(0, \infty)} e^{-y} \varphi(x, y) \, \d y$. Then, using BSPDE \eqref{eq:adjoint_1d}, we find
\begin{align} \label{eq:bspde_2_from_1}
    \langle \tilde{u}_t, \varphi\rangle &= \langle u_t, \varphi_0\rangle \notag \\
    &= \langle D\psi, \varphi_0\rangle + \int_t^T \Bigl\langle H^{\nu}_s\bigl(\cdot, u_s(\cdot), \partial_x u_s(\cdot)\bigr) + F^{\nu}_s(\cdot, \partial_x u_s), \varphi_0 \Bigr\rangle \, \d s \notag \\
    &\ \ \ - \int_t^T \bigl\langle \partial_x u_s, \partial_x (a_s \varphi_0)\bigr\rangle + \bigl\langle q_s, \sigma_0(s) \partial_x \varphi_0\bigr\rangle  \, \d s - \int_t^T \langle q_s, \varphi_0\rangle \, \d W_s - \langle m_T - m_t, \varphi_0\rangle \notag \\
    &= \langle e^{-\cdot} D\psi, \varphi\rangle + \int_t^T \Bigl\langle H^{\nu}_s\bigl(\cdot, u_s(\cdot), \partial_x u_s(\cdot)\bigr) + F^{\nu}_s(\cdot, \partial_x u_s), \varphi_0 \Bigr\rangle \, \d s \\
    &\ \ \ - \int_t^T \bigl\langle \partial_x \tilde{u}_s, \partial_x(a_s \varphi)\bigr\rangle + \langle \tilde{q}_s, \sigma_0(s) \partial_x \varphi\rangle \, \d s - \int_t^T \langle \tilde{q}_s, \varphi\rangle \, \d W_s - \langle \tilde{m}_T - \tilde{m}_t, \varphi\rangle \notag .
\end{align}
The terminal condition and the second line are already of the desired form, so it remains to compute the two terms involving $H^{\nu}$ and $F^{\nu}$. For the first expression, we write
\begin{align} \label{eq:rewriting_h}
\begin{split}
    &\bigl\langle H^{\nu}_s(\cdot, u_s(\cdot), \partial_x u_s(\cdot)\bigr), \varphi_0\bigr\rangle \\
    &= \int_{\R \times (0, \infty)} e^{-y}\Bigl(\inf_{h \in G}\bigl(b(s, x, \nu_s, h)\partial_x u_s(x) + f(s, x, \nu_s, h)\bigr) - \lambda_s(x)u_s(x)\Bigr) \varphi(x, y) \, \d x \d y.
\end{split}
\end{align}
Now, if $\mu_t(x, y) > 0$ then
\begin{align*}
    e^{-y}\inf_{h \in G}\bigl(b(s, x, \nu_s, h)&\partial_x u_t(x) + f(s, x, \nu_s, h)\bigr) \\
    &= \inf_{h \in G} \bigl(b(s, x, \nu_s, h)\partial_x \tilde{u}_t(x, y) + \tilde{f}(s, x, y, \nu_s, h)\bigr) \\
    &= \inf_{h \in G} \tilde{K}^{\mu}_s\bigl(x, y, \partial_x \tilde{u}_s(x, y), h\bigr).
\end{align*}
If, on the other hand, $\mu_t(x, y) = 0$ we compute
\begin{align*}
    e^{-y}&\inf_{h \in G}\bigl(b(s, x, \nu_s, h)\partial_x u_t(x) + f(s, x, \nu_s, h)\bigr) \\
    &= e^{-y}\Bigl(b\bigl(s, x, \nu_s, g_{\ast}\bigl(s, x, \partial_x u_s(x)\bigr)\bigr)\partial_x u_s(x, y) + f\bigl(s, x, \nu_s, g_{\ast}\bigl(s, x, \partial_x u_s(x)\bigr)\bigr)\Bigr) \\
    &= \Bigl(b\bigl(s, x, g_{\ast}\bigl(s, x, \nu_s, \partial_x u_s(x)\bigr)\bigr)\partial_x \tilde{u}_s(x, y) + \tilde{f}\bigl(s, x, y, g_{\ast}\bigl(s, x, \nu_s, \partial_x u_s(x)\bigr)\bigr)\Bigr) \\
    &= \tilde{K}^{\mu}_s\bigl(x, y, \partial_x \tilde{u}_s(x, y), g_{\ast}\bigl(s, x, \partial_x u_s(x)\bigr)\bigr),
\end{align*}
so together we get
\begin{align*}
    e^{-y}\Bigl(\inf_{h \in G}\bigl(b(s, x, \nu_s, h)&\partial_x u_t(x) + f(s, x, \nu_s, h)\bigr) \\
    &= \bf{1}_{\mu_s(x, y) > 0}\inf_{h \in G} \tilde{K}^{\mu}_s\bigl(x, y, \partial_x \tilde{u}_s(x, y), h\bigr) \\
    &\ \ \ + \bf{1}_{\mu_s(x, y) = 0} \tilde{K}^{\mu}_s\bigl(x, y, \partial_x \tilde{u}_s(x, y), g_{\ast}\bigl(s, x, \partial_x u_s(x)\bigr)\bigr) \\
    &= \tilde{H}^{\mu}_s\bigl(x, y, \partial_x \tilde{u}_s(x, y)\bigr),
\end{align*}
where in the last step we simply apply the definition of $\tilde{H}^{\mu}$ from \eqref{eq:hamiltonian_2d}. Inserting this equality into \eqref{eq:rewriting_h} yields
\begin{align*}
    \bigl\langle H^{\nu}_s(\cdot, u_s(\cdot), &\partial_x u_s(\cdot)\bigr), \varphi_0\bigr\rangle \\
    &= \int_{\R \times (0, \infty)} \Bigl(\tilde{H}^{\mu}_s\bigl(x, y, \partial_x \tilde{u}_s(x, y)\bigr) + \lambda_s(x)\partial_y\tilde{u}_s(x, y)\Bigr) \varphi(x, y) \, \d x \d y \\
    &= \bigl\langle \tilde{H}^{\mu}_s\bigl(\cdot, \cdot, \partial_x \tilde{u}_s(\cdot, \cdot)\bigr), \varphi\bigr\rangle - \langle \tilde{u}_s, \lambda_s \partial_y \varphi\rangle,
\end{align*}
where we used in the first equality that $-e^{-y}u_s(x) = \partial_y \tilde{u}_s(x, y)$ and applied integration by parts in the second equality, using that the boundary terms vanish since $\varphi(0) = 0$. For the second term, we get
\begin{align*}
    \bigl\langle F^{\nu}_s(&x, \partial_x u_s), \varphi_0\bigr\rangle \\
    &= \int_{\R} \Bigl(\bigl\langle \nu_s, Db_0(s, \cdot, \nu_s)(x)\partial_x u_s\bigr\rangle + \bigl\langle \nu_s, Df_0(s, \cdot, \nu_s)(x)\bigr\rangle\Bigr) \varphi_0(x) \, \d x \\
    &= \int_{\R \times (0, \infty)} \Bigl(\bigl\langle \mu_s, e^{-y}Db_0(s, \cdot, \nu_s)(x)\partial_x \tilde{u}_s\bigr\rangle + \bigl\langle \mu_s, e^{-y}D\tilde{f}_0(s, \cdot, \cdot, \nu_s)(x)\bigr\rangle\Bigr)\varphi(x, y) \, \d x \d y \\
    &= \bigl\langle \tilde{F}^{\mu}_s(\cdot, \cdot, \partial_x \tilde{u}_s), \varphi\bigr\rangle,
\end{align*}
where we used in the second equality that $\nu_s = \cal{S}(\mu_s)$. Inserting the previous two equalities into \eqref{eq:bspde_2_from_1} yields
\begin{align*}
    \langle \tilde{u}_t, \varphi\rangle &= \langle e^{-\cdot} D\psi, \varphi\rangle + \int_t^T \Bigl\langle \tilde{H}^{\mu}_s\bigl(\cdot, \cdot, \partial_x \tilde{u}_s(\cdot, \cdot)\bigr) + \tilde{F}^{\mu}_s(\cdot, \cdot, \partial_x \tilde{u}_s), \varphi \Bigr\rangle \, \d s \\
    &\ \ \ - \int_t^T \bigl\langle \partial_x \tilde{u}_s, \partial_x(a_s \varphi)\bigr\rangle + \langle \tilde{u}_s, \lambda_s \partial_y \varphi\rangle + \langle \tilde{q}_s, \sigma_0(s) \partial_x \varphi\rangle \, \d s \\
    &\ \ \ - \int_t^T \langle \tilde{q}_s, \varphi\rangle \, \d W_s - \langle \tilde{m}_T - \tilde{m}_t, \varphi\rangle.
\end{align*}
Since $\varphi \in H^1_0(\R \times (0, \infty))$ was arbitrary, $(\tilde{u}, \tilde{q}, \tilde{m}) \in \bb{H}(\R \times (0, \infty))$ satisfies BSPDE \eqref{eq:bspde_2d}.

To show that $(\tilde{u}, \tilde{q}, \tilde{m})$ is the unique solution to BSPDE \eqref{eq:bspde_2d}, we apply Theorem \ref{thm:bspde}. Hence, we need to verify that the coefficients of BSPDE \eqref{eq:bspde_2d} satisfy Assumption \ref{ass:bspde}. We already did that for all coefficients except $\tilde{H}^{\mu}$ in the proof of Proposition \ref{prop:adjoint_existence}. The assumptions on $\tilde{H}^{\mu}$ can be verified in the same manner as the assumptions on $H$ in the proof of Proposition \ref{prop:bspde_1}. This concludes the proof.
\end{proof}

From Proposition \ref{prop:bspde_2_from_1} and the necessary SMP, Theorem \ref{thm:smp}, we can then deduce Corollary \ref{cor:equivalence_l2} in the manner outlined above the statement of Corollary \ref{cor:equivalence_l2}. Hence, there exists an optimal control $g^{\ast} \in \cal{G}_{\bb{F}}^1$ which takes the form $g^{\ast}_t(x) = g_{\ast}\bigl(t, x, \partial_x u_t(x)\bigr)$ for $x \in \R$ and $\pr \otimes \leb$-a.e.\@ $(t, \omega) \in [0, T] \times \Omega$. Our next goal is to extend the existence of such an optimal semiclosed-loop control independent of the intensity to more general coefficients and cost functions.

\subsection{Extending Equivalence to General Coefficients and Cost Functions} \label{sec:extending}

Throughout this subsection Assumptions \ref{ass:mfl} and \ref{ass:convexity} are in place and we assume that $\L(\xi, \zeta) \in L^2(\R)$. We first establish a weak stability result for SPDE \eqref{eq:sfpe_joint} with varying coefficients and inputs. Let $(b^n)_n$, $(\lambda^n)_n$, $(f^n)_n$, and $(\psi^n)_n$ be sequences of functions such that
\begin{enumerate}[noitemsep, label = (\roman*)]
    \item \label{it:assumptions} $b^n$, $\lambda^n$, $f^n$, and $\psi^n$ satisfy the same properties as $b$, $\lambda$, $f$, and $\psi$ in Assumptions \ref{ass:spfe_joint}, \ref{ass:adjoint}, and \ref{ass:smp};
    \item \label{it:uniform_growth} there exists a $C > 0$ such that for $(t, x, v, g) \in [0, T] \times \R \times \cal{M}^2_{\leq 1}(\R) \times G$ we have
    \begin{align*}
        \lvert b^n(t, x, v, g)\rvert + \lvert \lambda^n(t, x)\rvert &\leq C(1 + \lvert x \rvert + M_2(v)), \\
        \lvert f^n(t, x, v, g)\rvert + \lvert \psi^n(v)\rvert &\leq C(1 + \lvert x \rvert^2 + M^2_2(v));
    \end{align*}
    \item \label{it:unif_local_convergence} $b^n(t, x, v, g)$, $\lambda^n(t, x)$, $f^n(t, x, v, g)$, and $\psi^n(v)$ tend to $b(t, x, v, g)$, $\lambda(t, x)$, $f(t, x, v, g)$, and $\psi(v)$ uniformly for $(t, x, v, g)$ in compact subsets of $[0, T] \times \R \times \cal{M}^2_{\leq 1}(\R) \times G$;
    \item \label{it:stict_convex_2} the map $G \ni g \mapsto f^n_1(t, x, g)$ is strictly convex for all $(t, x) \in [0, T] \times \R$.
\end{enumerate}
Since $b^n$, $\lambda^n$, $f^n$, and $\psi^n$ satisfy the same properties as $b$, $\lambda$, $f$, and $\psi$ in Assumptions \ref{ass:spfe_joint}, \ref{ass:adjoint}, and \ref{ass:smp}, we can apply Corollary \ref{cor:equivalence_l2}, where we replace the coefficients and costs $b$, $\lambda$, $f$, and $\psi$ of the control problem by $b^n$, $\lambda^n$, $f^n$, and $\psi^n$ and the initial condition $\L(\xi, \zeta)$ by $\L(\xi, \zeta/n) \in L^2(\R^2)$. The latter ensures that the cumulative intensity will be started from zero in the limit. We will use the assumptions formulated above, in particular the convergence condition \ref{it:unif_local_convergence}, to show that the existence of an optimal semiclosed-loop control independent of the intensity transfers from the control problem with coefficients and costs $b^n$, $\lambda^n$, $f^n$, and $\psi^n$ to the limiting control problem with coefficients and costs $b$, $\lambda$, $f$, and $\psi$.

For $n \geq 1$ we define the cost functionals
\begin{equation}
    J_{\text{cl}}^n(g) = \ev\biggl[\int_0^T \bigl\langle \mu^n_t, \tilde{f}^n(t, \cdot, \cdot, \nu^n_t, g_t)\bigr\rangle \, \d t + \psi^n(\nu^n_T)\biggr],
\end{equation}
for $g \in \cal{G}_{\bb{F}}^2$, where $\mu^n = (\mu^n_t)_{0 \leq t \leq T}$ satisfies SPDE \eqref{eq:sfpe_joint} with $b$ and $\lambda$ replaced by $b^n$ and $\lambda^n$, respectively, initial condition $\mu^n_0 = \L(\xi, \zeta/n)$, and control $g$. 
For a control $g \in \cal{G}_{\bb{F}}^1$, which is independent of the intensity, we can rewrite $J_{\text{cl}}^n(g)$ as $\ev\bigl[\int_0^T \bigl\langle \nu^n_t, f^n(t, \cdot, \nu^n_t, g_t)\bigr\rangle \, \d t + \psi^n(\nu^n_T)\bigr]$, where $\nu^n = (\nu^n_t)_{0 \leq t \leq T}$ is a solution to SPDE \eqref{eq:sfpe} with $b$ and $\lambda$ replaced by $b^n$ and $\lambda^n$, respectively, initial condition $\nu^n_0 = \ev[e^{-\zeta/n}]\L(\xi)$, and control $g$. The definition of the initial condition is to be understood in the sense that for all $\varphi \in C_b(\R)$ it holds that $\langle \nu^n_0, \varphi\rangle = \ev[e^{-\zeta/n}] \ev[\varphi(\xi)]$.

Let $(E, d)$ be a complete separable metric space. We let $\bb{M}_{\leq T}^2(E)$ denote the space of square-integrable measures $m$ on $[0, T] \times E$ such that $m([s, t] \times E) \leq t - s$ for all $0 \leq s \leq t \leq T$. Here square-integrable means that $\int_E d^2(x, x_0) \, \d m(x)$ for a fixed $x_0 \in E$. 


We say that a random variable $\mathfrak{m}$ with values in $\bb{M}_{\leq T}^2(E)$ is $\bb{F}$-progressively measurable if for any $t \in [0, T]$, the random variable $\mathfrak{m}([0, s] \times A)$ is $\cal{B}([0, t]) \otimes \F_t$-measurable for all $s \in [0, t]$ and $A \in \cal{B}(E)$.

\begin{proposition} \label{prop:stability_sfpe}
Let the sequences $(b^n)$, $(\lambda^n)$, $(f^n)$, and $(\psi^n)$ satisfy Properties \ref{it:assumptions} to \ref{it:stict_convex_2} above.
Fix a sequence $(g^n)_{n \geq 1}$ in $\cal{G}_{\bb{F}}^1$ and denote the solution to SPDE \eqref{eq:sfpe} with coefficients $b$ and $\lambda$ replaced by $b^n$ and $\lambda^n$, respectively, control $g^n$, and initial condition $\nu^n_0 = \ev[e^{-\zeta/n}] \L(\xi)$ by $\nu^n = (\nu^n_t)_{0 \leq t \leq T}$. Define the $\bb{M}_{\leq T}^2(\R \times G)$-valued random variable $\mathfrak{m}^n$ by
\begin{equation*}
    \int_{[0, T] \times \R \times G} \varphi(t, x, g) \, \d \mathfrak{m}^n(t, x, g) = \int_{[0, T]} \langle \nu^n_t, \varphi(t, \cdot, g_n)\rangle\, \d t
\end{equation*}
for $\varphi \in C_b([0, T] \times \R \times G)$. Then the sequence $(\mathfrak{m}^n, \nu^n, W)$ is tight on $\bb{M}_{\leq T}^2(\R \times G) \times C([0, T]; \cal{M}^1_{\leq 1}(\R)) \times C([0, T])$ and any subsequential limit $(\mathfrak{m}, \nu, W)$, which we assume is defined on the same probability space with the same Brownian motion $W$,
solves the SPDE
\begin{equation} \label{eq:sfpe_relaxed}
    \d \langle \nu_t, \varphi\rangle = \int_{\R \times G} \L\varphi(t, x, \nu_t, g) \, \d \mathfrak{m}(t, x, g) + \langle \nu_t, \sigma_0(t) \partial_x \varphi\rangle \, \d W_t
\end{equation}
for $\varphi \in C^2_c(\R)$ with initial condition $\nu_0 = \L(\xi)$. Along the chosen subsequence the cost $J_{\textup{cl}}^n(g^n)$ converges to
\begin{equation*}
    \ev\biggl[\int_{[0, T] \times \R \times G} f(t, x, \nu_t, g) \, \d \mathfrak{m}(t, x, g) + \psi(\nu_T)\biggr].
\end{equation*}
\end{proposition}

Note that assuming that $(\mathfrak{m}, \nu, W)$ is defined on $(\Omega, \F, \pr)$ means that we might have to enlarge the probability space. We do this in such a way that $(\mathfrak{m}, \nu, W)$ are independent of $(\xi, B)$.

\begin{proof}[Proof of Proposition \ref{prop:stability_sfpe}]
It is not difficult to see that $\nu^n$ is given by $\nu^n_t = \cal{S}(\mu^n_t)$, where $\mu^n = (\mu^n_t)_{0 \leq t \leq T}$ solves SPDE \eqref{eq:sfpe_joint} with input $g^n$, coefficients $b$ and $\lambda$ replaced by $b^n$ and $\lambda^n$, respectively, and initial condition $\mu^n = \L(\xi, \zeta/n)$. By Proposition \ref{prop:sfpe_joint}, extending the weak setup $\bb{S}^{\ast}$ if necessary, we can find a solution $(X^n, \Lambda^n)$ of McKean--Vlasov SDE \eqref{eq:mfl_weak} with initial condition $(X^n_0, \Lambda^n_0) = (\xi, \zeta/n)$ such that $\mu^n_t = \L(X^n_t, \Lambda^n_t \vert \F_T)$.

\textit{Step 1}: We first show that the sequence $(\mathfrak{m}^n)_{n \geq 1}$ is tight on $\bb{M}_{\leq T}^2(\R \times G)$. By \cite[Proposition B.1]{lacker_mfg_controlled_mgale_2015} it is enough to show that the sequence of measures $(m^n)_n$ on $[0, T] \times \R \times G$ defined by $m^n(A) = \ev[\mathfrak{m}^n(A)]$ for $A \in \cal{B}([0, T] \times \R \times G)$ is compact in $\bb{M}_{\leq T}^2(\R \times G)$ and that the family $(M_2^2(\mathfrak{m}^n))_{n \geq 1}$ is uniformly integrable. To establish the former, by Lemma \ref{lem:measures}, it is enough to show that the sequence $(m^n)_n$ is uniformly square-integrable. Now, we have that
\begin{align*}
    \sup_{n \geq 1}\int_{[0, T] \times \R \times G} \bf{1}_{\lvert x \rvert \geq c} \lvert x\rvert \, \d m^n(t, x, g) \leq \int_0^T \sup_{n \geq 1}\ev[\bf{1}_{\lvert X^n_t\rvert \geq c} \lvert X^n_t\rvert^2] \, \d t.
\end{align*}
Hence, if we can prove that the family $(X^n_t)_{n \geq 1}$ is uniformly square-integrable for all $t \in [0, T]$, the same holds for $(m^n)_n$. Before we establish the former, let us consider the second criterion above, namely that $(M_2^2(\mathfrak{m}^n))_{n \geq 1}$ is uniformly integrable. Since $[0, T]$ and $G$ are compact, uniform integrability of $(M_2^2(\mathfrak{m}^n))_{n \geq 1}$ is equivalent to uniform integrability of $\bigl(\int_0^T \langle \nu^n_t, \lvert \cdot\rvert^2\rangle \, \d t\bigr)_{n \geq 1}$. We can estimate
\begin{equation*}
    \int_0^T \langle \nu^n_t, \lvert \cdot\rvert^2\rangle \, \d t \leq \int_0^T \ev[\lvert X^n_t\rvert^2 \vert \F_T] \, \d t \leq T \ev\bigl[(\lvert X^n\rvert^{\ast}_T)^2 \big\vert \F_T\bigr],
\end{equation*}
which shows that uniform integrability of $\bigl(\int_0^T \langle \nu^n_t, \lvert \cdot\rvert^2\rangle \, \d t\bigr)_{n \geq 1}$ is implied by uniform integrability of $\ev\bigl[(\lvert X^n\rvert^{\ast}_T)^2 \big\vert \F_T\bigr]$. Now, both the uniform square-integrability of $(X^n_t)_{n \geq 1}$ for all $t \in [0, T]$ and the uniform integrability of $\bigl(\ev\bigl[(\lvert X^n\rvert^{\ast}_T)^2 \big\vert \F_T\bigr]\bigr)_{n \geq 1}$ are implied by the uniform square-integrability of $(\lvert X^n\rvert^{\ast}_T)_{n \geq 1}$, so that is what we will show. Using standard estimates for the McKean--Vlasov SDE satisfied by $(X^n, \Lambda^n)$ and, subsequently, applying Gr\"onwall's inequality, we find
\begin{equation} \label{eq:square_estimate}
    (\lvert X^n\rvert^{\ast}_t)^2 \leq C \lvert \xi\rvert^2 + C\int_0^t (1 + M_2^2(\nu^n_s)) \, \d s + C Z^n_t,
\end{equation}
where $Z^n_t = \sup_{0x \leq s \leq t} \bigl\lvert \int_0^s \sigma(r, X^n_r) \, \d B_r + \int_0^s \sigma_0(r) \, \d W_r\bigr\rvert^2$. Further estimating $M^2_2(\nu^n_s) \leq \ev[\lvert X^n_s \rvert^2 \vert \F_T] \leq \ev[(\lvert X^n\rvert^{\ast}_t)^2\vert \F_T]$ for $s \in [0, t]$, taking conditional expectation with respect to $\F_T$ on both sides of \eqref{eq:square_estimate}, and then applying Gr\"onwall's inequality again yields that $\ev[(\lvert X^n\rvert^{\ast}_T)^2\vert \F_T] \leq C + C \ev[Z^n_T \vert \F_T]$ for a possibly enlarged constant $C > 0$. Since the diffusion coefficients $\sigma$ and $\sigma_0$ are bounded, the sequence $(\ev[Z^n_T \vert \F_T])_{n \geq 1}$ is uniformly integrable, so that the same is true for $(\ev[(\lvert X^n\rvert^{\ast}_T)^2\vert \F_T])_{n \geq 1}$ and, therefore, $(\sup_{0 \leq t \leq T} M_2^2(\nu^n_t))_{n \geq 1}$. Thus, in view of \eqref{eq:square_estimate}, the family $(\lvert X^n\rvert^{\ast}_T)_{n \geq 1}$ is uniformly integrable. 

Next, we establish the tightness of $(\nu^n)_{n \geq 1}$ on $C([0, T]; \cal{M}^1_{\leq 1}(\R))$. By Theorem 8.6 of Chapter 3 in \cite{ethier_convergence_markov_1986} to prove tightness it is enough to show that (i) $(\nu^n_t)_n$ is tight on $\cal{M}^1_{\leq 1}(\R)$ for all $t \in [0, T]$ and (ii) there exists a sequence of random variables $(\zeta_n)_n$ with $\sup_n \bb{E} \zeta_n < \infty$ such that for all $0 \leq s \leq t \leq T$, we have $\bb{E}\bigl[d_1(\nu^n_t, \nu^n_s) \big\vert \F_s\bigr] \leq \lvert t - s \rvert^{1/2} \bb{E}[\zeta_n \vert \F_s]$, where we recall that $d_1$ is introduced in Subsection \ref{sec:main_results_application}. We already showed (i) and the second property is easily established, using the expression $d_1(v_1, v_2) = \sup_{\lVert \varphi\rVert_{\text{Lip}}} \langle v_1 - v_2, \varphi\rangle$. We do not provide the details here, but note that the required estimates can be deduced along the same lines as the proof of Proposition 3.4 in \cite{hambly_mvcp_arxiv_2023}.

\textit{Step 2}: Let $(\mathfrak{m}, \nu, W)$ be a weak subsequential limit of $(\mathfrak{m}^n, \nu^n, W)_{n \geq 1}$. We show that $(\mathfrak{m}, \nu, W)$ solves SPDE \eqref{eq:sfpe_relaxed}. To achieve this we reformulate the SPDEs satisfied by $(\mathfrak{m}^n, \nu^n, W)$ and $(\mathfrak{m}, \nu, W)$ as martingale problems and deduce the stability of those martingale problems. Let us define the generator $\L^n$ by
\begin{equation*}
    \L^n \varphi(t, x, v, g) = b^n(t, x, v, g)\partial_x \varphi(x) - \lambda^n(t, x) \varphi(x) + a_t(x) \partial_x^2 \varphi(x)
\end{equation*}
for $\varphi \in C_b^2(\R)$ and $(t, x, v, g) \in [0, T] \times \R \times \cal{M}^2_{\leq 1}(\R) \times G$. By assumption $\nu^n$ satisfies SPDE \eqref{eq:sfpe} with $b$ and $\lambda$ replaced by $b^n$ and $\lambda^n$, respectively, initial condition $\nu^n_0 = \ev[e^{-\zeta/n}]\L(\xi)$, and control $g^n$. We can rewrite this as
\begin{equation*}
    \d \langle \nu^n_t, \varphi\rangle = \int_{\R \times G} \L^n\varphi(t, x, \nu^n_t, g) \, \d \mathfrak{m}^n(t, x, g) + \langle \nu^n_t, \sigma_0(t) \partial_x \varphi\rangle \, \d W_t
\end{equation*}
for $\varphi \in C^2_c(\R)$. We formulate this SPDE as a martingale problem. For $k \geq 1$, let $\Phi \in C^2_c(\R^{k + 1})$ and $\varphi \in C^2_c(\R; \R^k)$, and let $\langle v, \varphi\rangle$ denote the vector $(\langle v, \varphi_1\rangle, \dots, \langle v, \varphi_k\rangle)$ for $v \in \cal{M}^2_{\leq 1}(\R)$. We define the differential operator $\bb{L}^n$ acting on $\Phi$ and $\varphi$ by
\begin{align*}
    &\bb{L}^n(\Phi, \varphi)(t, x, w, v, g) \\
    &= \sum_{i = 1}^k \partial_{y_i} \Phi(\langle v, \varphi\rangle, w) \L^n\varphi_i(t, x, v, g)  + \sum_{i, j = 1}^k \frac{1}{2}\partial^2_{y_i y_j} \Phi(\langle v, \varphi\rangle, w) \langle v, \sigma_0(t) \partial_x \varphi_j\rangle \sigma_0(t) \partial_x \varphi_i(x) \\
    &\ \ \ + \sum_{i = 1}^k \frac{1}{2}\partial^2_{y_i y_{k + 1}} \Phi(\langle v, \varphi\rangle, w) \sigma_0(t) \partial_x \varphi_i(x)
\end{align*}
for $(t, x, w, v, g) \in [0, T] \times \R^2 \times \cal{M}^2_{\leq 1}(\R) \times G$. 
We define the operator $\bb{L}$ analogously, only replacing the appearance of $\L^n$ in the definition of $\bb{L}^n$ by $\L$. Now, applying It\^o's formula to $\Phi(\langle \nu^n_t, \varphi\rangle, W_t)$ implies that
\begin{align*}
    \d \Phi(\langle \nu^n_t, \varphi\rangle, W_t) &= \int_{\R \times G} \bb{L}^n(\Phi, \varphi)(t, x, W_t, \nu^n_t, g) \, \d \mathfrak{m}^n(t, x, g) + \frac{1}{2}\partial^2_{y_{k + 1}} \Phi(\langle \nu^n_t, \varphi\rangle, W_t) \, \d t \\
    &\ \ \ + \sum_{i = 1}^k \partial_{y_i} \Phi(\langle \nu^n_t, \varphi\rangle, W_t)\langle \nu^n_t, \sigma_0(t) \partial_x \varphi_i\rangle \, \d W_t + \partial_{y_{k + 1}} \Phi(\langle \nu^n_t, \varphi\rangle, W_t) \, \d W_t.
\end{align*}
In other words, the process $\cal{M}^n(\Phi, \varphi)$ defined by
\begin{align*}
    \cal{M}^n_t(\Phi, \varphi) &= \Phi(\langle \nu^n_t, \varphi\rangle, W_t) - \int_{[0, t] \times \R \times G} \bb{L}^n(\Phi, \varphi)(s, x, W_s, \nu^n_s, g) \, \d \mathfrak{m}^n(s, x, g) \\
    &\ \ \ - \int_0^t\frac{1}{2}\partial^2_{y_{k + 1}} \Phi(\langle \nu^n_s, \varphi\rangle, W_s) \, \d s
\end{align*}
is a martingale for any $\Phi \in C^2_c(\R^k)$ and $\varphi \in C^2_c(\R; \R^k)$. On the one hand, as in \cite[Lemma A.7]{hambly_mvcp_arxiv_2023}, we can show that for any convergent sequence $(m^n, v^n, w^n)_{n \geq 1}$ in $\bb{M}^2_{\leq T}(\R \times G) \times C([0, T]; \cal{M}^2_{\leq 1}(\R)) \times C([0, T])$ with limit $(m, v, w)$, the expression
\begin{align*}
    \int_{[0, t] \times \R \times G} \bb{L}^n(\Phi, \varphi)(s, x, w^n_s, v^n_s, g) \, \d m^n(s, x, g) + \int_0^t\frac{1}{2}\partial^2_{y_{k + 1}} \Phi(\langle v^n_s, \varphi\rangle, w^n_s) \, \d s
\end{align*}
converges to 
\begin{align*}
    \int_{[0, t] \times \R \times G} \bb{L}(\Phi, \varphi)(s, x, w_s, v_s, g) \, \d m(s, x, g) + \int_0^t\frac{1}{2}\partial^2_{y_{k + 1}} \Phi(\langle v_s, \varphi\rangle, w_s) \, \d s,
\end{align*}
and $\Phi(\langle v^n_t, \varphi\rangle, w^n_t)$ tends to $\Phi(\langle v_t, \varphi\rangle, w_t)$. On the other hand, it is not difficult to see that $(\cal{M}^n(\Phi, \varphi))_{n \geq 1}$ is tight on $C([0, T])$. Consequently, we see that along a further subsequence $\cal{M}^n(\Phi, \varphi)$ converges weakly to the process $\cal{M}(\Phi, \varphi)$ given by
\begin{align*}
    \cal{M}_t(\Phi, \varphi) &= \Phi(\langle \nu_t, \varphi\rangle, W_t) - \int_{[0, t] \times \R \times G} \bb{L}(\Phi, \varphi)(s, x, W_s, \nu_s, g) \, \d \mathfrak{m}(s, x, g) \\
    &\ \ \ - \int_0^t\frac{1}{2}\partial^2_{y_{k + 1}} \Phi(\langle \nu_s, \varphi\rangle, W_s) \, \d s.
\end{align*}
Moreover, since the local martingale property is stable under weak convergence and the expectation $\ev\lvert \cal{M}_T(\Phi, \varphi)\rvert^2$ is finite, it follows that $\cal{M}(\Phi, \varphi)$ is a martingale for all $\Phi \in C^2_c(\R^k)$ and $\varphi \in C^2_c(\R; \R^k)$. By Theorem II.7.1 from \cite{ikeda_sde_1989} we can conclude that
\begin{equation*}
    \langle \nu_t, \varphi_i\rangle =  \langle v_0, \varphi_i\rangle + \int_{[0, t] \times \R \times G} \L\varphi_i(s, x, \nu_s, g) \, \d \mathfrak{m}(s, x, g) + \int_0^t \langle \nu_s, \sigma_0(s) \partial_x \varphi_i\rangle \, \d W_s
\end{equation*}
holds a.s.\@ for any finite collection $\varphi_1$,~\ldots, $\varphi_k$ of function in $C^2_c(\R)$. Since $C^2_c(\R)$ is separable when equipped with the norm $\varphi \mapsto \lVert \varphi\rVert_{\infty} + \lVert \partial_x \varphi\rVert_{\infty} + \lVert \partial_x^2 \varphi\rVert_{\infty}$ and both sides of the above equation vary continuously as a function of $\varphi_i$, we conclude that $(\mathfrak{m}, \nu, W)$ satisfies SPDE \eqref{eq:sfpe_relaxed}.

\textit{Step 3}: It remains to establish the convergence of the costs $J^n_{\text{cl}}(g^n)$. Note that again proceeding as in \cite[Lemma A.7]{hambly_mvcp_arxiv_2023}, we can show that along the subsequence along which $(\mathfrak{m}^n, \nu^n, W)$ converges weakly to $(\mathfrak{m}, \nu, W)$, the integral
\begin{equation*}
    \int_0^T \bigl\langle \nu^n_t, f^n(t, \cdot, \nu^n_t, g^n_t)\bigr\rangle \, \d t + \psi(\nu^n_T) = \int_{0, T] \times \R \times G} f^n(t, x, \nu^n_t, g) \, \d \mathfrak{m}^n(t, x, g) + \psi(\nu^n_T)
\end{equation*}
converges weakly to $\int_{0, T] \times \R \times G} f(t, x, \nu_t, g) \, \d \mathfrak{m}(t, x, g) + \psi(\nu_T)$. Here, we crucially use the uniform convergence of $f^n$ and $\psi^n$ on compact subsets of $[0, T] \times \R \times \cal{M}^2_{\leq 1}(\R) \times G$, Property \ref{it:unif_local_convergence}, as well as the fact that convergent sequences in $\bb{M}_{\leq T}^2(\R \times G)$ are uniformly square-integrable together with the uniform quadratic growth of $f^n$ and $\psi^n$, Property \ref{it:uniform_growth}. Next, from Step 1 we know that the sequences $(M_2^2(\mathfrak{m}^n))_{n \geq 1}$ and $(M_2^2(\nu^n_T))_{n \geq 1}$ are uniformly integrable. Hence, using Property \ref{it:uniform_growth} again,
allows us to upgrade the weak subsequential convergence of $\int_0^T \bigl\langle \nu^n_t, f^n(t, \cdot, \nu^n_t, g^n_t)\bigr\rangle \, \d t + \psi(\nu^n_T)$ to subsequential convergence in mean. Thus, $J^n_{\text{cl}}(g^n)$ converges subsequentially to $\ev\bigl[\int_{0, T] \times \R \times G} f(t, x, \nu_t, g) \, \d \mathfrak{m}(t, x, g) + \psi(\nu_T)\bigr]$. This concludes the proof.
\end{proof}

Let $(\mathfrak{m}, \nu, W)$ be as in the statement of Proposition \ref{prop:stability_sfpe}. Firstly, let us note that the random variables $\mathfrak{m}$ and $\nu$ are not necessarily $\bb{F}$-progressively measurable, so we have to introduce a different filtration. We let $\bb{F}^{\mathfrak{m}, W}$ denote the completion of the filtration 
\begin{equation*}
    \Bigl(\sigma\bigl(\mathfrak{m}([0, s] \times A),\, W_s \define A \in \cal{B}(\R \times G),\, 0 \leq s \leq t\bigr)\Bigr)_{0 \leq t \leq T}.
\end{equation*}
The random measure $\mathfrak{m}$ is by definition $\bb{F}^{\mathfrak{m}, W}$-progressively measurable. Moreover, owing to Lemma \ref{lem:stability_of_total_mass} there is a family of random probability measures $(\Gamma_{t, x})_{(t, x) \in [0, T] \times \R}$ on $G$ such that $\d\mathfrak{m}(t, x, g) = \d \Gamma_{t, x}(g) \d \nu_t(x) \d t$. In particular, the process $\nu$ is $\bb{F}^{\mathfrak{m}, W}$-adapted as well. In addition, it follows from \cite[Lemma 3.2]{lacker_mfg_controlled_mgale_2015} that the map $\Gamma \define [0, T] \times \Omega \times \R \to \P(G)$ can be chosen in an $\bb{F}^{\mathfrak{m}, W}$-predictable manner. We now define $\bb{S}^{\ast} = (\Omega, \F, \pr, \bb{F}^{\mathfrak{m}, \xi, B, W}, \bb{F}^{\mathfrak{m}, W}, \xi, 0, B, W)$, where $\bb{F}^{\mathfrak{m}, \xi, W, B}$ is the completion of $\bb{F}^{\mathfrak{m}, W} \lor \bb{F}^{\xi, B}$ and $\bb{F}^{\xi, B}$ denotes the completion of the filtration generated by $\xi$ and $B$. One can check that $\bb{S}^{\ast}$ is a weak setup.

\begin{remark}
Let us note that Equation \eqref{eq:sfpe_relaxed} can be viewed as a relaxed formulation for the control of the stochastic Fokker--Planck equation \eqref{eq:sfpe}. Here a relaxed control is a progressively measurable $\bb{M}_{\leq T}^2(\R \times G)$-valued process $\mathfrak{m}$ such that the flow of measures $\nu = (\nu_t)_{0 \leq t \leq T}$ obtained through the disintegration $\d \mathfrak{m}(t, x, g) = \d \Gamma_{t, x}(g) \, \d \nu_t(x) \d t$ satisfies SPDE \eqref{eq:sfpe_relaxed}. This idea can be extended to arbitrary (stochastic) Fokker--Planck equations or, more generally, measure-valued (S)PDEs, such as those arising in filtering theory, and may be useful to establish the existence of optimal controls.
\end{remark}

Next, since Assumption \ref{ass:convexity} is in place, as in the proof of \cite[Theorem 3.7]{lacker_mfg_controlled_mgale_2015}, we can find $g^{\ast} \in \cal{G}_{\bb{F}^{\mathfrak{m}, W}}^1$ such that
\begin{equation} \label{eq:measurable_selection}
    \int_G b(t, x, \nu_t, g) \, \d \Gamma_{t, x}(g) = b\bigl(t, x, \nu_t, g^{\ast}_t(x)\bigr), \quad \int_G f(t, x, \nu_t, g) \, \d \Gamma_{t, x}(g) \geq f(t, x, \nu_t, g^{\ast}_t(x)).
\end{equation}
This allows us to rewrite SPDE \eqref{eq:sfpe_relaxed} as
\begin{equation} \label{eq:eq:sfpe_new}
    \d \langle \nu_t, \varphi\rangle = \bigl\langle \nu_t, \L\varphi(t, \cdot, \nu_t, g^{\ast}_t)\bigr\rangle \, \d t + \langle \nu_t, \sigma_0(t) \partial_x \varphi\rangle \, \d W_t
\end{equation}
for $\varphi \in C^2_c(\R)$. For the associated cost we obtain the upper bound
\begin{equation} \label{eq:cost_bound}
    \ev\biggl[\int_0^T \bigl\langle\nu_t, f(t, \cdot, \nu_t, g^{\ast}_t)\bigr\rangle \, \d t + \psi(\nu_T)\biggr] \leq \ev\biggl[\int_{[0, T] \times \R \times G} f(t, x, \nu_t, g) \, \d \mathfrak{m}(t, x, g) + \psi(\nu_T)\biggr].
\end{equation}

\begin{proposition} \label{prop:mv_representation}
Let $\nu$ and $g^{\ast}$ be as above. Then, extending the weak setup $\bb{S}^{\ast}$ if necessary, there exists a solution $(X, \Lambda)$ to McKean--Vlasov SDE \eqref{eq:mfl_weak} with initial conditions $X_0 = \xi$, $\Lambda_0 = 0$, subprobability distribution $\langle \nu_t, \varphi\rangle = \ev[e^{\Lambda_t}\varphi(X_t)\vert \F_T]$, and control $\gamma^{\ast} = (\gamma^{\ast}_t)_{0 \leq t \leq T}$ given by $\gamma^{\ast}_t = g^{\ast}_t(X_t)$. That is, $\gamma^{\ast}$ is a semiclosed-loop control with feedback function $g^{\ast}$ and $J_{\bb{S}^{\ast}}(\gamma^{\ast}) = \ev\bigl[\int_0^T \bigl\langle\nu_t, f(t, \cdot, \nu_t, g^{\ast}_t)\bigr\rangle \, \d t + \psi(\nu_T)\bigr]$.
\end{proposition}

\begin{proof}
Let us consider the random linear SPDE
\begin{equation*}
    \d \langle \tilde{\mu}_t, \varphi\rangle = \bigl\langle \tilde{\mu}_t, \L_2\varphi(t, \cdot, \nu_t, g^{\ast}_t)\bigr\rangle \, \d t + \langle \tilde{\mu}_t, \sigma_0(t) \partial_x \varphi\rangle \, \d W_t
\end{equation*}
for $\varphi \in C^2_c(\R^2)$ with initial condition $\tilde{\mu}_0 = \L(\xi, \zeta)$. Note that unlike in SPDE \eqref{eq:sfpe_joint}, the flow of subprobabilities $\nu$ is fixed and not determined by the SPDE. Hence, the SPDE is not nonlinear but linear with random coefficients. Since $\L(\xi, \zeta) \in L^2(\R^2)$, using similar arguments as in the proof of Proposition \ref{prop:sfpe_joint}, we can show that this SPDE has a unique strong solution. Moreover, if we define $\tilde{\nu} = (\tilde{\nu}_t)_{0 \leq t \leq T}$ by $\tilde{\nu}_t, = \ev[e^{-\zeta}]^{-1} \cal{S}(\tilde{\mu}_t)$, then $\tilde{\nu}$ solves the random linear SPDE
\begin{equation*}
    \d \langle \tilde{\nu}_t, \varphi\rangle = \bigl\langle \tilde{\nu}_t, \L\varphi(t, \cdot, \nu_t, g^{\ast}_t)\bigr\rangle \, \d t + \langle \tilde{\nu}_t, \sigma_0(t) \partial_x \varphi\rangle \, \d W_t
\end{equation*}
for $\varphi \in C^2_c(\R)$ with initial condition $\tilde{\nu}_0 = \L(\xi)$. This SPDE again has a unique solution, so since $\nu$ solves the same linear SPDE we obtain that $\tilde{\nu} = \nu$. In other words, it holds that $\nu_t = \ev[e^{-\zeta}]^{-1} \cal{S}(\tilde{\mu}_t)$. 

Now, by \cite[Theorem 3.1]{lacker_mimicking_2020}, extending the weak setup if necessary, we can find a solution $(X, \tilde{\Lambda})$ to the random SDE
\begin{equation*}
    \d X_t = b\bigl(t, X_t, \nu_t, g^{\ast}_t(X_t)\bigr) \, \d t + \sigma(t, X_t) \, \d B_t + \sigma_0(t) \, \d W_t, \quad \d \tilde{\Lambda}_t = \lambda(t, X_t) \, \d t
\end{equation*}
with initial conditions $X_0 = \xi$, $\tilde{\Lambda}_0 = \zeta$, such that $\L(X_t, \tilde{\Lambda}_t \vert \F_T) = \tilde{\mu}_t$. Note that $\nu$ is a fixed input and not determined by the SDE. We define $\Lambda_t = \int_0^t \lambda(s, X_s) \, \d s$ and set $\mu_t = \L(X_t, \Lambda_s \vert \F_T)$. Then it follows that 
\begin{equation*}
    \langle \cal{S}(\mu_t), \varphi\rangle = \ev[e^{-\Lambda_t}\varphi(X_t) \vert \F_T] = \ev[e^{-\zeta}]^{-1}\ev\bigl[e^{-\tilde{\Lambda}_t}\varphi(X_t) \big\vert \F_T\bigr] = \ev[e^{-\zeta}]^{-1}\cal{S}(\tilde{\mu}_t) = \nu_t,
\end{equation*}
where we used that $\zeta$ is independent of $X$, $\Lambda$, and $\F_T$. Hence, $(X, \Lambda)$ is a solution to McKean--Vlasov SDE \eqref{eq:mfl_weak}. This, in particular, implies that
\begin{equation*}
    J_{\bb{S}}(\gamma^{\ast}) = \ev\biggl[\int_0^T \bigl\langle\nu_t, f(t, \cdot, \cdot, \nu_t, g^{\ast}_t)\bigr\rangle \, \d t + \psi(\nu_T)\biggr],
\end{equation*}
where $\gamma^{\ast} = (\gamma^{\ast}_t)_{0 \leq t \leq T}$ is defined by $\gamma^{\ast}_t = g^{\ast}_t(X_t)$.
\end{proof}

We can use the subsequential convergence established in Proposition \ref{prop:stability_sfpe} and the representation from Proposition \ref{prop:mv_representation} to show that if the controls $g^n \in \cal{G}_{\bb{F}}^1$ in the statement of Proposition \ref{prop:stability_sfpe} are chosen optimally, then a random function $g^{\ast} \in \cal{G}_{\bb{F}}^1$ chosen according to the criterion \eqref{eq:measurable_selection} yields an optimal semiclosed-loop control which is independent of the intensity. 

For convenience, in the following we let $V^n$ denote the infimum of $J^n_{\textup{cl}}(g)$ over control $g \in \cal{G}_{\bb{F}}^2$. From our discussion in Section \ref{sec:main_results_application} we know that $V^n$ coincides with the value of the strong formulation introduced at the beginning of Section \ref{sec:main_results_application} with $b$, $\lambda$, $f$, and $\psi$ replaced by $b^n$, $\lambda^n$, $f^n$, and $\psi^n$, respectively, and initial condition $(\xi, \zeta/n)$.

\begin{proposition} \label{prop:cp_stability}
Choose the controls $g^n \in \cal{G}_{\bb{F}}^1$ in the statement of Proposition \ref{prop:stability_sfpe} optimally. 
Then the control $\gamma^{\ast}$ from Proposition \ref{prop:mv_representation} is an optimal semiclosed-loop control and the feedback function $g^{\ast} \in \cal{G}_{\bb{F}^{\mathfrak{m}, W}}^1$ does not depend on the intensity.
\end{proposition}

Note that optimal controls $g^n \in \cal{G}_{\bb{F}}^1$, i.e.\@ controls for which $J^n_{\textup{cl}}(g^n) = V^n_{\textup{cl}}$, exist by Corollary \ref{cor:equivalence_l2}.
 
\begin{proof}[Proof of Proposition \ref{prop:cp_stability}]
Note that regardless of which weakly convergent subsequence $(\mathfrak{m}^{n_k}, \nu^{n_k}, W)_{k \geq 1}$ we choose in Proposition \ref{prop:stability_sfpe}, in view of Equation \eqref{eq:cost_bound} and Proposition \ref{prop:mv_representation}, we always obtain that
\begin{equation*}
    \lim_{k \to \infty} J_{\text{cl}}^{n_k}(g^{n_k}) = \ev\biggl[\int_{[0, T] \times \R \times G} f(t, x, \nu_t, g) \, \d \mathfrak{m}(t, x, g) + \psi(\nu_T)\biggr] \geq J_{\bb{S}^{\ast}}(\gamma^{\ast}) \geq V_{\bb{S}^{\ast}} = V.
\end{equation*}
Hence, we find that $\liminf_{n \to \infty} V^n \geq V$. 

Next, let us deduce the converse inequality. To achieve this we exploit the equivalence between the strong formulation introduced at the beginning of Section \ref{sec:main_results_application} and the SPDE formulation that we study here. Let us denote the cost functional for the strong formulation with $b$, $\lambda$, $f$, and $\psi$ replaced by $b^n$, $\lambda^n$, $f^n$, and $\psi^n$, respectively, and initial condition $(\xi, \zeta/n)$ by $J^n$. Then $J_{\textup{cl}}(g^n) = V^n = \inf_{\gamma} J^n(\gamma)$, where the infimum is over $\bb{F}^{\xi, \zeta, B, W}$-progressively measurable $G$-valued processes $\gamma = (\gamma_t)_{0 \leq t \leq T}$, and $V = \inf_{\gamma}J(\gamma)$, where the infimum is over $\bb{F}^{\xi, B, W}$-progressively measurable $G$-valued processes $\gamma = (\gamma_t)_{0 \leq t \leq T}$. Now, let us choose an $\bb{F}^{\xi, B, W}$-progressively measurable $G$-valued processes $\gamma = (\gamma_t)_{0 \leq t \leq T}$ with $J(\gamma) \leq V + \epsilon$. It holds that $J_{\textup{cl}}^n(g^n) = V^n \leq J^n(\gamma)$, so if we can show that $\lim_{n \to \infty}J^n(\gamma) = J(\gamma)$, then we have proved that
\begin{equation*}
    V \leq \limsup_{n \to \infty} V^n \leq \limsup_{n \to \infty} J^n(\gamma) = J(\gamma) \leq V + \epsilon.
\end{equation*}
Letting $\epsilon \to 0$, then yields the equality $V \leq J_{\bb{S}^{\ast}}(\gamma^{\ast}) \leq \lim_{n \to \infty} J_{\textup{cl}}^n(g^n) = \lim_{n \to \infty} V^n = V$. 
Hence, $\gamma^{\ast}$ is an optimal semiclosed-loop control whose feedback function $g^{\ast}$ is independent of the intensity.

It remains to prove that $J^n(\gamma) \to J(\gamma)$ for the $\bb{F}^{\xi, B, W}$-progressively measurable $G$-valued process $\gamma$ from above. Let $(X^n, \Lambda^n)$ be the solution to McKean--Vlasov SDE \eqref{eq:mfl} with coefficients $b$ and $\lambda$ replaced by $b^n$ and $\lambda^n$, respectively, initial condition $(\xi, \zeta/n)$, and control $\gamma$. By $(X, \Lambda)$ we denote the solution to McKean--Vlasov SDE \eqref{eq:mfl} with initial condition $(\xi, 0)$ and control $\gamma$. Then owing to the uniform in $n$ linear growth of $b^n$ and $\lambda^n$ and the locally uniform convergence of $b^n$ to $b$ and $\lambda^n$ to $\lambda$ (cf.\@ Properties \ref{it:unif_local_convergence} and \ref{it:uniform_growth}), we can show that $\lim_{n \to \infty}\ev(\lvert X^n - X\rvert^{\ast}_T)^2 = 0$. The only difficulty in obtaining this convergence lies in the fact that the mapping $\cal{S}$ defined in \eqref{eq:subprobability_function} is only locally Lipschitz continuous as a map $\P^2(\R \times [0, \infty)) \to \cal{M}^2_{\leq 1}(\R)$, where $\P^2(\R \times [0, \infty))$ is equipped with the $2$-Wasserstein distance and $\cal{M}^2_{\leq 1}(\R)$ with the metric $d_1$. However, this hurdle can be overcome through stopping time arguments. We do not provide the details here and instead refer to \cite[Propositions 3.12 and A.2]{hambly_mvcp_arxiv_2023}. Once we know that $(X^n)_{n \geq 1}$ converges to $X$ in $L^2$-$\sup$, the convergence of the costs $J^n(\gamma)$ to $J(\gamma)$ follows straightforwardly from the uniform quadratic growth assumptions placed on $f^n$ and $\psi^n$ as well as the locally uniform convergence of $f^n$ to $f$ and $\psi^n$ to $\psi$ (cf.\@ Properties \ref{it:unif_local_convergence} and \ref{it:uniform_growth}). This concludes the proof.
\end{proof}

Finally, to prove Theorem \ref{thm:equivalence} we wish to apply Proposition \ref{prop:cp_stability}. Starting from the coefficients $b$ and $\lambda$ as well as cost functions $f$ and $\psi$ satisfying Assumptions \ref{ass:mfl} and \ref{ass:convexity}, we want to mollify and cut off these functions appropriately to obtain maps $b^n$, $\lambda^n$, $f^n$, and $\psi^n$ that satisfy Properties \ref{it:assumptions} to \ref{it:stict_convex_2} outlined at the beginning of this subsection. If we find such maps, then Proposition \ref{prop:cp_stability} implies the existence of an optimal semiclosed-loop control independent of the intensity. To construct the desired maps, we use Proposition \ref{prop:mollification_convex} and Proposition \ref{prop:mollificaton_lfd}.

\begin{proof}[Proof of Theorem \ref{thm:equivalence}]
Owing to Assumption \ref{ass:mfl}, the functions $b$ and $f$ decompose as $b(t, x, v, g) = b_0(t, x, v) + b_1(t, x, g)$ and $f(t, x, v, g) = f_0(t, x, v) + f_1(t, x, g)$ for $(t, x, v, g) \in [0, T] \times \R \times \cal{M}^2_{\leq 1}(\R) \times G$. Now, we can apply Proposition \ref{prop:mollification_convex} to obtain a sequence $(f_1^n)_{n \geq 1}$ that converges locally uniformly to $f_1$. Further, we define the sequence $(b_1^n)_{n \geq 1}$ by $b_1^n(t, x, g) = b(t, -n \lor x \land n, g)$. Next, by Proposition \ref{prop:mollificaton_lfd} we can find sequences $(\tilde{b}_0^n)_{n \geq 1}$, $(\tilde{f}_0^n)_{n \geq 1}$, and $(\psi^n)_{n \geq 1}$ that converge locally uniformly to $b_0$, $f_0$, and $\psi$, respectively. We then set $b_0^n(t, x, v) = (\eta_{1/n} \ast \tilde{b}_0^n(t, \cdot, g))(x)$ and $f_0^n(t, x, v) = (\eta_{1/n} \ast \tilde{f}_0^n(t, \cdot, g))(x)$, where $\eta \define \R \to \R$ is the standard mollifier and $\eta_{\epsilon}(x) = \frac{1}{\epsilon}\eta(\frac{x}{\epsilon})$. This ensures the Lipschitz continuity of $(x, v) \mapsto b_0^n(t, x, v)$ uniformly in $t \in [0, T]$ and the uniform continuity of $(x, v) \mapsto f_0^n(t, x, v)$ for $(t, x, v)$ in compact subsets of $[0, T] \times \R \times \cal{M}^2_{\leq 1}(\R)$ without destroying Properties \ref{it:unif_conv} to \ref{it:func_der_bound} from Proposition \ref{prop:mollificaton_lfd}. Then, we define the functions $b^n$, $f^n$, and $\lambda^n$ by $b^n(t, x, v, g) = b_0^n(t, x, v) + b_1^n(t, x, g)$, $f^n(t, x, v, g) = f_0^n(t, x, v) + f_1^n(t, x, g)$, and $\lambda^n(t, x) = \lambda(t, x) \land n$ for $(t, x, v, g) \in [0, T] \times \R \times \cal{M}^2_{\leq 1}(\R) \times G$. It is straightforward to check that the functions $b^n$, $\lambda^n$, $f^n$, and $\psi^n$ satisfy Properties \ref{it:assumptions} to \ref{it:stict_convex_2}. Hence, as stated above, we can conclude with Proposition \ref{prop:cp_stability}.
\end{proof}

\appendix

\section{Appendix} \label{sec:appendix}

\subsection{Auxiliary Results for Section \ref{sec:bspde}}

\begin{lemma} \label{lem:stoch_int_conv}
Let $\cal{H}$ be a separable Hilbert space. Suppose that $(u^n)_{n \geq 1}$ is a sequence in $D_{\bb{F}}^2([0, T]; \cal{H})$ that converges to some process $u \in D_{\bb{F}}^2([0, T]; \cal{H})$, and that $(m^n)_{n \geq 1}$, $(\tilde{m}^n)_{n \geq 1}$ are sequences in $\cal{M}_{\bb{F}}^2([0, T]; \cal{H})$ that converge to processes $m$, $\tilde{m} \in \cal{M}_{\bb{F}}^2([0, T]; \cal{H})$, respectively. Next, let $F_1 \define \cal{H} \to \cal{H}$ and $F_2 \define \cal{H} \to B_2(\cal{H})$ be continuous and bounded on bounded sets of $\cal{H}$. Then
\begin{align} 
    \int_0^t \langle F_1(u^n_{s-}), \d m^n_s\rangle_{\cal{H}} &\to \int_0^t \langle F_1(u_{s-}), \d m_s\rangle_{\cal{H}}, \label{eq:conv_stoch_int_1} \\
    \int_0^t \trace\bigl(F_2(u^n_{s-}) \, \d [[m^n, \tilde{m}^n]]^c_s\bigr) &\to \int_0^t \trace\bigl(F_2(u_{s-}) \, \d [[m, \tilde{m}]]^c_s\bigr) \label{eq:conv_stoch_int_2}
\end{align}
uniformly in $t \in [0, T]$ in probability.
\end{lemma}

\begin{proof}
We prove the convergences in \eqref{eq:conv_stoch_int_1} and \eqref{eq:conv_stoch_int_2} separately. For the first one we write
\begin{align} \label{eq:decomp_first_integral}
\begin{split}
    \int_0^t \langle F_1(u^n_{s-})&, \d m^n_s\rangle_{\cal{H}} - \int_0^t \langle F_1(u_{s-}), \d m_s\rangle_{\cal{H}} \\
    &= \int_0^t \bigl\langle F_1(u^n_{s-}) - F_1(u_{s-}), \d m^n_s\bigr\rangle_{\cal{H}} + \int_0^t \bigl\langle F_1(u_{s-}), \d (m^n_s - m_s)\bigr\rangle_{\cal{H}}.
\end{split}
\end{align}
We will deal with both terms on the right-hand side in sequence. Our goal is to prove that both converge to zero in probability. For $k \geq 1$ let $\pi_k \define \cal{H} \to \cal{H}$ denote the projection onto the centred ball $B_k$ of radius $k$ in $\cal{H}$. Now, since $\ev \sup_{0 \leq s \leq t} \lVert u^n_s - u_s\rVert_{\cal{H}}^2 \to 0$, for any $\delta > 0$ we can find $K \geq 1$ such that for all $k \geq K$ it holds that 
\begin{align*}
    \pr\Bigl(\pi_k(u^n_{s-}) = u^n_{s-},\, \pi_k(u_{s-}) = u_{s-} \ \text{for } s \in (0, t]\Bigr) &\leq \pr\biggl(\sup_{0 \leq s \leq t}\lVert u^n_s\rVert_{\cal{H}}^2 \lor \sup_{0 \leq s \leq t}\lVert u_s\rVert_{\cal{H}}^2 \leq k\biggr) \\
    &\geq 1 - \delta.
\end{align*}
Hence, for any $\epsilon > 0$ we obtain
\begin{align} \label{eq:probability_bound}
\begin{split}
    \pr\biggl(\sup_{0 \leq t \leq T}\biggl\lvert &\int_0^t \bigl\langle F_1(u^n_{s-}) - F_1(u_{s-}), \d m^n_s\bigr\rangle_{\cal{H}}\biggr\rvert > \epsilon\biggr) \\
    &\leq \pr\biggl(\sup_{0 \leq t \leq T}\biggl\lvert \int_0^t \bigl\langle F_1(\pi_k(u^n_{s-})) - F_1(\pi_k(u_{s-})), \d m^n_s\bigr\rangle_{\cal{H}}\biggr\rvert > \epsilon\biggr) + \delta.
\end{split}
\end{align}
Next, using Proposition \ref{prop:bdg} we find that 
\begin{align} \label{eq:expectation_bound_conv}
    \pr\biggl(\sup_{0 \leq t \leq T}\biggl\lvert \int_0^t \bigl\langle F_1&(\pi_k(u^n_{s-})) - F_1(\pi_k(u_{s-})), \d m^n_s\bigr\rangle_{\cal{H}}\biggr\rvert > \epsilon\biggr) \notag \\
    &\leq \frac{1}{\epsilon}\ev\sup_{0 \leq t \leq T}\biggl\lvert \int_0^t \bigl\langle F_1(\pi_k(u^n_{t-})) - F_1(\pi_k(u_{t-})), \d m^n_s\bigr\rangle_{\cal{H}}\biggr\rvert \notag \\
    &\leq \frac{C_1}{\epsilon} \ev\biggl(\int_0^T \bigl\lVert F_1(\pi_k(u^n_{t-})) - F_1(\pi_k(u_{t-}))\bigr\rVert_{\cal{H}}^2 \, \d [m^n]_t\biggr)^{1/2} \notag \\
    &\leq \frac{C_1}{\epsilon} \ev\biggl[\sup_{0 \leq t \leq T} \bigl\lVert F_1(\pi_k(u^n_t)) - F_1(\pi_k(u_t))\bigr\rVert_{\cal{H}} \bigl([m^n]_T - [m^n]_0\bigr)^{1/2}\biggr] \notag \\
    &\leq \frac{C_1^2}{\delta \epsilon^2} \ev \sup_{0 \leq t \leq T} \bigl\lVert F_1(\pi_k(u^n_t)) - F_1(\pi_k(u_t))\bigr\rVert_{\cal{H}}^2 + \delta \ev\bigl([m^n]_T - [m^n]_0\bigr).
\end{align}
Now, since $m^n$ converges to $m$ in $\cal{M}_{\bb{F}}^2([0, T]; \cal{H})$, the quantity $\ev\bigl([m^n]_T - [m^n]_0\bigr)$ is bounded uniformly in $n \geq 1$ by some constant $C > 0$. Moreover, the map $F_1 \circ \pi_k$ is continuous and bounded, so since $\ev \sup_{0 \leq t \leq T} \lVert u^n_t - u_t\rVert_{\cal{H}}^2 \to 0$ by the dominated convergence theorem, we may choose $n \geq 1$ large enough so that $\ev \sup_{0 \leq t \leq T} \bigl\lVert F_1(\pi_k(u^n_t)) - F_1(\pi_k(u_t))\bigr\rVert_{\cal{H}}^2 \leq \frac{\delta^2 \epsilon^2}{C_1^2}$. Plugging this into \eqref{eq:expectation_bound_conv} and combining it with Equation \eqref{eq:probability_bound} implies that
\begin{equation*}
    \pr\biggl(\sup_{0 \leq t \leq T} \biggl\lvert \int_0^t \bigl\langle F_1(u^n_{s-}) - F_1(u_{s-}), \d m^n_s\bigr\rangle_{\cal{H}}\biggr\rvert > \epsilon\biggr) \leq (2 + C)\delta.
\end{equation*}
Since $\delta$ and $\epsilon$ were arbitrary this implies that $\int_0^t \bigl\langle F_1(u^n_{s-}) - F_1(u_{s-}), \d m^n_s\bigr\rangle_{\cal{H}}$ converges uniformly in $t \in [0, T]$ to zero in probability.

Next, let us consider the second term on the right-hand side of \eqref{eq:decomp_first_integral}. Arguing as above for any $\delta > 0$ we can choose $K \geq 1$ and $k \geq K$ such that for any $\epsilon > 0$,
\begin{align*}
    \pr\biggl(\biggl\lvert\int_0^t \bigl\langle F_1&(u_{s-}), \d (m^n_s - m_s)\bigr\rangle_{\cal{H}}\biggr\rvert > 
    \epsilon\biggr) \\
    &\leq \frac{1}{\epsilon}\ev\biggl\lvert\int_0^t \bigl\langle F_1(\pi_k(u_{s-})), \d (m^n_s - m_s)\bigr\rangle_{\cal{H}}\biggr\rvert + \delta \\
    &\leq \frac{\delta}{c_k^2} \ev \sup_{0 \leq t \leq T} \lVert F_1(\pi_k(u_t))\rVert_{\cal{H}}^2 + \frac{C_1^2 c_k^2}{\delta \epsilon^2} \ev\bigl[[m^n]_t - [m^{\epsilon}]_0 - ([m]_t - [m]_0)\bigr] + \delta \\
    &\leq 2\delta + \frac{C_1^2 c_k^2}{\delta \epsilon^2} \ev\bigl[[m^n]_t - [m^n]_0 - ([m]_t - [m]_0)\bigr],
\end{align*}
where $c_k = \sup_{h \in B_k}\lVert F_1(h)\rVert_{\cal{H}}$. Then we choose $\epsilon$ large enough such that $\ev\bigl[[m^n]_t - [m^n]_0 - ([m]_t - [m]_0)\bigr] \leq \frac{\delta^2 \epsilon^2}{C_1^2 c_n^2}$ to obtain $\pr\bigl(\sup_{0 \leq t \leq T}\bigl\lvert\int_0^t \bigl\langle F_1(u_{s-}), \d (m^n_s - m_s)\bigr\rangle_{\cal{H}}\bigr\rvert > \epsilon\bigr) \leq 3 \delta$, which means that $\int_0^t \bigl\langle F_1(u_{s-}), \d (m^n_s - m_s)\bigr\rangle_{\cal{H}}$ converges uniformly in $t \in [0, T]$ to zero in probability as well. This concludes our treatment of the convergence in \eqref{eq:conv_stoch_int_1}.

We continue with \eqref{eq:conv_stoch_int_2}. By appealing to the polarisation identity $[[m, \tilde{m}]] = \frac{1}{4}([[m + \tilde{m}]] - [[m - \tilde{m}]])$, we can reduce to the case $\tilde{m}^n = m^n$ and $\tilde{m} = m$. Then similarly to Equation \eqref{eq:decomp_first_integral}, we decompose
\begin{align*}
    \int_0^t \trace\bigl(&F_2(u^n_{s-}) \, \d [[m^n]]^c_s\bigr) - \int_0^t \trace\bigl(F_2(u_{s-}) \, \d [[m]]^c_s\bigr) \\
    &= \int_0^t \trace\bigl((F_2(u^n_{s-}) - F_2(u_{s-})) \, \d [[m^n]]^c_s\bigr) + \int_0^t \trace\bigl(F_2(u_{s-}) \, \d ([[m^n]]^c_s - [[m]]^c_s)\bigr).
\end{align*}
For the first term we apply the inequality
\begin{equation} \label{eq:bound_trace}
    \sup_{0 \leq t \leq T}\biggl\lvert \int_0^t \trace\bigl((F_2(u^n_{s-}) - F_2(u_{s-})) \, \d [[m^n]]^c_s\bigr) \biggr\rvert \leq \int_0^T \bigl\lVert F_2(u^n_{t-}) - F_2(u_{t-})\bigr\rVert \, \d [m^n]^c_t,
\end{equation}
where $\bigl\lVert F_2(u^n_{t-}) - F_2(u_{t-})\bigr\rVert$ denotes the operator norm of $F_2(u^n_{t-}) - F_2(u_{t-}) \define \cal{H} \to \cal{H}$ and then proceed as above to show that the right-hand side converges to zero in probability as $n \to \infty$. To deal with the second term, we write $[[m^n]]^c_s - [[m]]^c_s = [[m^n - m, m^n]]^c_s + [[m, m^n - m]]^c_s$ and then apply the Kunita--Watanabe inequality 
to obtain
\begin{align*}
    \sup_{ 0 \leq t \leq T} \biggl\lvert \int_0^t \trace\bigl(F_2&(u_{s-}) \, \d ([[m^n]]^c_s - [[m]]^c_s)\bigr)\biggr\rvert \\
    &\leq \sqrt{2}\biggl(\int_0^T \lVert F_2(u_{t-})\rVert \, \d ([m^n]_t + [m]_t)\biggr)^{1/2}\biggl(\int_0^T \lVert F_2(u_{t-})\rVert \, \d [m^n - m]_t\biggr)^{1/2} \\
    &\leq \delta \int_0^T \lVert F_2(u_{t-})\rVert \, \d ([m^n]_t + [m]_t) + \frac{1}{2\delta} \int_0^T \lVert F_2(u_{t-})\rVert \, \d [m^n - m]_t.
\end{align*}
The first expression on the right-hand can be made arbitrarily small by choosing $\delta > 0$ appropriately. The second term tends to zero in probability as $n \to \infty$. Hence, $\sup_{ 0 \leq t \leq T} \bigl\lvert \int_0^t \trace\bigl(F_2(u_{s-}) \, \d ([[m^n]]^c_s - [[m]]^c_s)\bigr)\bigr\rvert \to 0$ in probability, which concludes the proof.
\end{proof}

\subsection{Auxiliary Results for Section \ref{sec:equivalence}}

Let $(E, d)$ be a complete separable metric space. We let $\bb{M}_{\leq T}^2(E)$ denote the space of square-integrable measures $m$ on $[0, T] \times E$, i.e.\@ those with $\int_{[0, T] \times E} d^2(x, x_0) \, \d m(x) < \infty$ for some fixed $x_0 \in E$, such that $m([s, t] \times E) \leq t - s$ for all $0 \leq s \leq t \leq T$. We equip $\bb{M}_{\leq T}^2(E)$ with the metric
\begin{equation*}
    d_2(m^1, m^2) = TW_2\bigl(\tfrac{1}{T}\tilde{m}^1, \tfrac{1}{T}\tilde{m}^2\bigr) + \bigl\lvert m^1([0, T] \times E) - m^2([0, T] \times E)\bigr\rvert,
\end{equation*}
where $\tilde{m}^i$ is defined by $\d \tilde{m}^i(t, x) = \d m^i(t, x) + (m^i_t(E) - 1) \d \delta_{x_0} \, \d t$, where $(m^i_t)_{0 \leq t \leq T} \in L^1([0, T])$ is the weak derivative of the map $t \mapsto m^i([0, t] \times E)$. Here $W_2$ denotes the $2$-Wasserstein distance on the space of square-integrable probability measures on $[0, T] \times E$.


\begin{lemma} \label{lem:measures}
A sequence in $\bb{M}_{\leq T}^2(E)$ is compact if and only if it is tight and uniformly square-integrable.
\end{lemma}

\begin{proof}
This follows essentially from \cite[Theorem 6.9]{villani_ot_2009}.
\end{proof}


\begin{lemma} \label{lem:stability_of_total_mass}
Let $(E', d')$ be another complete separable metric space and let $(m^n)_n$ be a convergent sequence in $\bb{M}_{\leq T}^2(E \times E')$ with limit $m$. Assume that there exists a family of probability measures $(p^n_{t, x})_{(t, x) \in [0, T] \times E}$ such that $\d m^n(t, x, y) = \d p^n_{t, x}(y) \d \pi^{\#}m^n(t, x)$, where $\pi \define [0, T] \times E \times E' \to [0, T] \times E$ is the projection onto $[0, T] \times E$. Then the same holds for $m$, i.e.\@ we find a family of probability measures $(p_{t, x})_{(t, x) \in [0, T] \times E}$ such that $\d m(t, x, y) = \d p_{t, x}(y) \d \pi^{\#}m(t, x)$.
\end{lemma}

\begin{proof}
Let us define the space $E_{\ast} = E \times \{0\} \cup \{(x_0, 1)\}$, which we endow with the product metric $d_{\ast}((x, y), (x', y')) = d(x, x') + \lvert y - y'\rvert$ for $x$, $x' \in E$ and $y$, $y' \in \{0, 1\}$. We will identify the elements $x$ of $E$ with $(x, 0) \in E_{\ast}$ and view $\ast = (x_0, 1) \in E_{\ast}$ as an isolated point. Now, it holds that $\pi^{\#}m^n([s, t] \times E) \leq t - s$ for any $0 \leq s \leq t \leq T$, so there exists a flow of subprobability measures $(v^n_t)_{0 \leq t \leq T}$, such that $\d \pi^{\#}(t, x) = \d v^n_t(x) \d t$. Using this disintegration, to any $m^n$, we can associate a measure $\bar{m}^n$ on $[0, T] \times E_{\ast} \times E'$ by $\d \bar{m}^n(t, x, y) = \d m^n(t, x, y) + \d \delta_{y_0}(y) (1 - v^n_t(E)) \d \delta_{\ast}(x) \d t$ for some fixed $y_0 \in E'$. The family $(\bar{m}^n)_n$ is tight and we let $\bar{m}$ denote its limit along some convergent subsequence $(n_k)_k$. Since $\frac{1}{T}\bar{m}$ is a probability measure and $\bar{m}([s, t] \times E_{\ast} \times E') = t - s$ for $0 \leq s \leq t \leq T$, we can disintegrate $\bar{m}$ as $\d \bar{m}(t, x, y) = \d p_{t, x}(y) \d \bar{v}_t(x) \d t$ for families of probability measures $(p_{t, x})_{(t, x) \in [0, T] \times E_{\ast}}$ and $(\bar{v}_t)_{0 \in [0, T]}$. We define the flow of subprobability measures $(v_t)_{0 \in [0, T]}$ by $v_t(A) = \bar{v}_t(A)$ for $A \in \cal{B}(E)$ and claim that $\d m(t, x, y) = \d p_{t, x}(y) \d v_t(x) \d t$, which yields the desired disintegration of $m$. Let $\varphi \define [0, T] \times E \times E' \to \R$ be continuous and bounded and define $\bar{\varphi} \in C_b([0, T] \times E_{\ast} \times E')$ by $\bar{\varphi}(t, x, y) = \varphi(t, x, y)$ if $x \in E$ and $\bar{\varphi}(t, x, y) = 0$ if $x = \ast$. Then, it holds that
\begin{align*}
    \langle m, \varphi\rangle &= \lim_{k \to \infty} \langle m^{n_k}, \varphi\rangle = \lim_{k \to \infty} \langle \bar{m}^{n_k}, \bar{\varphi}\rangle = \langle \bar{m}, \bar{\varphi}\rangle \\
    &= \int_0^T \int_E \int_{E'} \bar{\varphi}(t, x, y) \, \d p_{t, x}(y) \, \d \bar{v}_t(x) \, \d t + \int_0^T \biggl(\int_{E'} \bar{\varphi}(t, \ast, y) \, \d p_{t, x}(y)\biggr) \bar{v}_t(\{\ast\}) \, \d t \\
    &= \int_0^T \int_E \int_{E'} \varphi(t, x, y) \, \d p_{t, x}(y) \, \d v_t(x) \, \d t.
\end{align*}
Since $\varphi$ was arbitrary, this proves the claim. 
\end{proof}

\subsection{Regularisation of Functions}


\begin{proposition} \label{prop:mollification_convex}
Let $\varphi \define [0, T] \times \R \times G \to \R$ satisfy the same conditions as $f_1$ in Assumptions \ref{ass:mfl} and \ref{ass:convexity}. Then we can find a sequence $(\varphi_n)_n$ of functions $\varphi_n \define [0, T] \times \R \times \R^d \to \R$ such that
\begin{enumerate}[noitemsep, label = \normalfont(\roman*)]
    \item \label{it:unif_conv_g} $\varphi_n(t, \cdot, \cdot)$ converges to $\varphi(t, \cdot, \cdot)$ uniformly on compact subsets of $\R \times G$ for all $t \in [0, T]$;
    \item \label{it:unif_quad_growth_g} $\lvert \varphi_n(t, x, g) \rvert \leq C'(1 + \lvert x\rvert^2)$ for a constant $C' > 0$ independent of $n$;
    \item \label{it:bounded_in_l2_g} there exists a function $h_n \in L^2(\R)$ with $\lvert \varphi_n(t, x, g)\rvert \leq h_n(x)$ for all $(t, x, g) \in [0, T] \times \R \times G$;
    \item \label{it:strictly_convex_cont_diff} the map $g \mapsto \varphi_n(t, x, g)$ is strictly convex and continuous differentiable, and $\partial_g \varphi_n$ is bounded.
\end{enumerate}
\end{proposition}

\begin{proof}
We construct an approximation of $\varphi$ in two steps. Let $n \geq 1$ and define $\tilde{\varphi}_n \define \R^d \to \R$ through $\inf$-convolution, i.e.\@ we set $\tilde{\varphi}_n(t, x, g) = \inf_{h \in G} \bigl(\varphi(t, x, h) + n\lvert g - h\rvert^2\bigr)$. Since $g \mapsto \varphi(t, x, g)$ is convex, the same holds for $g \mapsto \tilde{\varphi}_n(t, x, g)$. The second step is to mollify $\tilde{\varphi}_n$. Let $\eta \define \R^d \to \R$ be the standard mollifier and set $\eta_{\epsilon}(x) = \frac{1}{\epsilon^d}\eta(\frac{x}{\epsilon})$. We define $\tilde{\varphi}_n^{\epsilon}(t, x, g) = (\eta_{\epsilon} \ast \tilde{\varphi}_n(t, x, \cdot)(g)$ and finally set $\varphi_n^{\epsilon}(t, x, g) = e^{-x^2/n} \bigl(\tilde{\varphi}_n^{\epsilon}(t, x, g) + \frac{1}{n} \lvert g\rvert^2\bigr)$.
We will now choose $\epsilon = \epsilon_n$ to ensure the locally uniform convergence of $\varphi_n^{\epsilon_n}$. We estimate
\begin{align} \label{eq:decomp_convex}
\begin{split}
    \lvert \varphi_n^{\epsilon}(t, x, g) - \varphi(t, x, g)\rvert &\leq \lvert \varphi_n^{\epsilon}(t, x, g) - \tilde{\varphi}_n^{\epsilon}(t, x, g)\rvert + \lvert \tilde{\varphi}_n^{\epsilon}(t, x, g) - \tilde{\varphi}_n(t, x, g)\rvert \\
    &\ \ \ \lvert \tilde{\varphi}_n(t, x, g) - \varphi(t, x, g)\rvert.
\end{split}
\end{align}
We will estimate the three terms on the right-hand side above. Let us fix a compact subset $K$ of $\R$ and let $\omega_K$ be a modulus of continuity of $g \mapsto \varphi(t, x, g)$ which is uniform in $[0, T] \times K$, i.e.\@ $\lvert \varphi(t, x, g) - \varphi(t, x, g')\rvert \leq \omega_K(\lvert g - g'\rvert)$ for all $g$, $g' \in G$ and $(t, x) \in [0, T] \times K$. Let us start with the third term on the right-hand side of \eqref{eq:decomp_convex}. For $g \in G$, we estimate
\begin{align*}
    \lvert \tilde{\varphi}_n(t, x, g) - \varphi(t, x, g)\rvert &= \biggl\lvert \inf_{h \in G} \varphi(t, x, h) + n\lvert g - h\rvert^2 - \varphi(t, x, g)\bigg\rvert \\
    &\leq \inf_{\delta \in [0, \text{diam}(G)]} -\omega_K(\delta) + n \delta^2.
\end{align*}
Define $\delta_n = \sup\{\delta \in [0, \text{diam}(G)] \define -\omega_K(\delta) + n\delta^2 \leq 0\}$. Clearly, it holds that $\delta_n \to 0$, so that
\begin{equation*}
    \lvert \tilde{\varphi}_n(t, x, g) - \varphi(t, x, g)\rvert \leq \inf_{\delta \in [0, \text{diam}(G)]} -\omega_K(\delta) + n \delta^2 \leq \omega_K(\delta_n) \to 0
\end{equation*}
as $n \to \infty$. Moreover, from the definition of $\tilde{\varphi}_n$, it immediately follows that $\lvert \tilde{\varphi}_n(t, x, g)\rvert \leq \lvert \varphi(t, x, g)\rvert \leq C(1 + \lvert x\rvert^2)$ for $(t, x, g) \in [0, T] \times \R \times G$.

Next, we consider the second expression on the right-hand side of \eqref{eq:decomp_convex}. We have
\begin{align*}
    \lvert \tilde{\varphi}_n^{\epsilon}(t, x, g) - \tilde{\varphi}_n(t, x, g)\rvert &\leq \int_{\R^d} \eta_{\epsilon}(g - h) \bigl\lvert \tilde{\varphi}_n(t, x, h) - \tilde{\varphi}_n(t, x, g)\bigr\rvert \, \d h \\ 
    &\leq \sup_{h \in B_{\epsilon}(g)} \bigl\lvert \tilde{\varphi}_n(t, x, h) - \tilde{\varphi}_n(t, x, g)\bigr\rvert \\
    &\leq \sup_{h \in B_{\epsilon}(g)} \sup_{y \in G} n \bigl\lvert \lvert g - y\rvert^2 - \lvert h - y\rvert^2\bigr\rvert,
\end{align*}
where $B_{\epsilon}(h)$ denotes the ball of radius $\epsilon$ centered at $h \in \R^d$. Since $G$ is compact, the expression in the third line above tends to zero as $\epsilon \to 0$ uniformly over $g \in G$. Consequently, we may choose $\epsilon_n > 0$ such that $\epsilon_n \to 0$ and $\lvert \tilde{\varphi}_n^{\epsilon_n}(t, x, g) - \tilde{\varphi}_n(t, x, g)\rvert \leq \epsilon_n$ for all $(t, x, g) \in [0, T] \times \R \times G$.

Lastly, we bound the first term on the right-hand side of \eqref{eq:decomp_convex}. We have
\begin{equation} \label{eq:some_bound}
    \lvert \varphi_n^{\epsilon}(t, x, g) - \tilde{\varphi}_n^{\epsilon}(t, x, g)\rvert \leq (1 - e^{-x^2/n}) \lvert \tilde{\varphi}_n^{\epsilon}(t, x, g)\rvert + \frac{1}{n}\sup_{g \in G}\lvert g\rvert^2.
\end{equation}
Now, it holds that
\begin{equation} \label{eq:bound_of_mollified}
    \lvert \tilde{\varphi}_n^{\epsilon}(t, x, g)\rvert \leq \int_{\R} \eta_{\epsilon}(h - g) \lvert \tilde{\varphi}_n(t, x, h)\rvert \, \d h \leq \sup_{h \in B_{\epsilon}(g)} \lvert \tilde{\varphi}_n(t, x, h)\rvert \leq C(1 + \lvert x\rvert^2).
\end{equation}
Inserting this into \eqref{eq:some_bound} yields
\begin{equation*}
    \lvert \varphi_n^{\epsilon}(t, x, g) - \tilde{\varphi}_n^{\epsilon}(t, x, g)\rvert \leq C(1 - e^{-x^2/n})(1 + \lvert x \rvert^2) + \frac{1}{n}\sup_{h \in G}\lvert h\rvert^2
\end{equation*}
and the bound on the right-hand side converges to zero uniformly in $\epsilon > 0$ and $x \in K$. Putting everything together implies that $\varphi_n \define [0, T] \times \R \times G \to G$ defined by $\varphi_n^{\epsilon_n}(t, x, g) = \varphi_n(t, x, g)$ converges to $\varphi$ uniformly on compact subsets of $[0, T] \times \R \times G$, so Property \ref{it:unif_conv_g} holds. Moreover, by \eqref{eq:bound_of_mollified} we obtain that
\begin{align*}
    \lvert \varphi_n(t, x, g)\rvert &\leq e^{-x^2/n}\lvert \tilde{\varphi}_n^{\epsilon_n}(t, x, g)\rvert + e^{-x^2/n}\frac{1}{n}\lvert g\rvert^2 \leq Ce^{-x^2/n}(1 + \lvert x \rvert^2) + e^{-x^2/n}\frac{1}{n}\sup_{h \in G}\lvert h\rvert^2
\end{align*}
for all $(t, x, g) \in [0, T] \times \R \times G$, which establishes Property \ref{it:unif_quad_growth_g}. Next, note that the left-hand side is in $L^2(\R)$ as a function of $x$, so Property \ref{it:bounded_in_l2_g} is satisfied as well. Lastly, the strict convexity and continuous differentiability of $g \mapsto \varphi_n(t, x, g)$ follow by construction and the boundedness of $\partial_g \varphi_n(t, x, g)$ is a consequence of the prefactor $e^{-x^2/n}$. This establishes Property \ref{it:strictly_convex_cont_diff} and concludes the proof.
\end{proof}

\begin{proposition} \label{prop:mollificaton_lfd}
Let $\varphi \define [0, T] \times \R \times \cal{M}^2_{\leq 1}(\R) \to \R$ be a measurable map such that $(x, v) \mapsto \varphi(t, x, v)$ is uniformly continuous for $(t, x, v)$ in compact subsets of $[0, T] \times \R \times \cal{M}^2_{\leq 1}(\R)$ and $\lvert \varphi(t, x, v)\rvert \leq C(1 + \lvert x \rvert + M_2^p(v))$ for some $p \in [1, 2]$ and $C > 0$. Then there exists a sequence $(\varphi_n)_n$ of measurable function $\varphi_n \define [0, T] \times \R \times \cal{M}^2_{\leq 1}(\R) \to \R$ such that
\begin{enumerate}[noitemsep, label = \normalfont(\roman*)]
    \item \label{it:unif_conv} $\varphi_n(t, \cdot, \cdot)$ converges to $\varphi(t, \cdot, \cdot)$ uniformly on compact subsets of $\R \times \cal{M}^2_{\leq 1}(\R)$ for all $t \in [0, T]$;
    \item \label{it:unif_quad_growth} $\lvert \varphi_n(t, x, v)\rvert \leq C'(1 + \lvert x \rvert^p + M_2^p(v))$ for a constant $C' > 0$ independent of $n$;
    \item \label{it:bounded_in_l2} there exists a function $h_n \in C_c(\R)$ with $\lvert \varphi_n(t, x, v)\rvert \leq h_n(x)$ for all $(t, x, v) \in [0, T] \times \R \times \cal{M}^2_{\leq 1}(\R)$;
    \item \label{it:func_der_bound} $v \mapsto \varphi_n(t, x, v)$ admits a linear functional derivative $D \varphi_n \define [0, T] \times \R \times \cal{M}^2_{\leq 1}(\R) \times \R \to \R$ such that $\lvert D \varphi_n(t, x, v)(y)\rvert \leq h_n(y)$ for all $(t, x, v) \in [0, T] \times \R \times \cal{M}^2_{\leq 1}(\R)$ and $y \in \R$ and $\cal{M}^2_{\leq 1}(\R) \cap L^2(\R) \ni v \mapsto D\varphi_n(t, x, v)(y)$ is continuous with respect to $\lVert \cdot \rVert_{L^2}$.
\end{enumerate}
\end{proposition}

\begin{proof}
We construct an approximation of $\varphi$ in three steps. First we pick $n \geq 1$ and $k \geq 1$ and set $x_i = -n + i 2^{-k}$ for $i = 0$,~\ldots, $2n 2^k = L$. Then we let $(\psi_i)_{i = 1, \dots, L - 1}$ be a partition of unity subordinated to the open cover $((x_i - 2^{-k}, x_i + 2^{-k}))_{i = 1, \dots, L - 1}$ and set $\kappa_n(x) = \sum_{i = 1}^{L - 1} \psi_i(x)$ for $x \in \R$. Now, for $v \in \mathcal{M}^2_{\leq 1}(\R)$ we define the subprobability measures $v_n \in \mathcal{M}^2_{\leq 1}(\R)$ by $\d v_n(x) = \kappa_n(x) \, \d v(x)$ and $v_{n, k} \in \mathcal{M}^2_{\leq 1}(\R)$ by $\d v_{n, k}(x) = \sum_{i = 1}^{L - 1} \varpi_i(v) \bar{\psi}_i(x) \, \d x$, where $\varpi_i(v) = \int_{\R} \psi_i(x) \, \d v(x)$ and $\bar{\psi}_i(x) = c_i \psi_i(x)$ with $c_i^{-1} = \int_{\R} \psi_i(x) \, \d x$. Next, for $w \in [0, \infty)^{L - 1}$ we let $m_w \in \cal{M}(\R)$ be defined by $\d m_w(x) = \sum_{i = 1} w_i \bar{\psi}_i(x) \, \d x$. Then, we introduce the function $\Phi_{n, k} \define [0, T] \times \R \times \R^{L - 1} \to \R$ given by $\Phi_{n, k}(t, x, w) = \varphi(t, x, m_{w^+})$, where $w^+ = (w_i \lor 0)_{i = 1, \dots, L - 1}$ for $w \in \R^{L - 1}$. Finally, for $\epsilon > 0$ we set $\Phi_{n, k}^{\epsilon}(t, x, w) = (\eta_{\epsilon} \ast \Phi_{n, k}(t, x, \cdot))(w)$ and $\varphi_{n, k}^{\epsilon}(t, x, v) = \Phi_{n, k}^{\epsilon}(t, x, \varpi(v))$, where $\eta \define \R^{L - 1} \to \R$ is the standard mollifier and $\eta_{\epsilon}(x) = \frac{1}{\epsilon^{L - 1}}\eta(\frac{x}{\epsilon})$ for $x \in \R^{L - 1}$. Our goal is to show that choosing first $\epsilon = \epsilon_{n, k}$ small enough and then $k = k_n$ sufficiently large, the resulting sequence $(\varphi_n)_{n \geq 1}$, defined by $\varphi_n(t, x, v) = \kappa_n(x)\varphi_{n, k_n}^{\epsilon_{n, k_n}}(t, x, v)$, satisfies Properties \ref{it:unif_conv}, \ref{it:unif_quad_growth}, and \ref{it:func_der_bound}. 

Let us start with Property \ref{it:unif_conv}. Note that for $n \geq 1$, on the set $\cal{K}_n$ of measures $v \in \cal{M}^2_{\leq 1}(\R)$ such that $\supp v \subset [-n, n]$ the metrics $d_0$ and $d_2$ are strongly equivalent, so there exists a constant $c_n > 0$ such that $d_2 \leq c_n d_0$. Now, for $k \geq 1$, and $\epsilon > 0$, we can write
\begin{align} \label{eq:decomp_approx}
\begin{split}
    \lvert \varphi_{n, k}^{\epsilon}(t, x, v) - \varphi(t, x, v)\rvert &\leq \lvert \varphi_{n, k}^{\epsilon}(t, x, v) - \kappa_n(x)\varphi(t, x, v_{n, k})\rvert + \lvert \varphi(t, x, v_{n, k}) - \varphi(t, x, v_n)\rvert \\
    &\ \ \ + \lvert \kappa_n(x)\varphi(t, x, v_n) - \varphi(t, x, v)\rvert.
\end{split}
\end{align}
We inspect the three terms in turn. For the first, we obtain
\begin{align} \label{eq:distance_to_moll}
    \lvert \varphi_{n, k}^{\epsilon}(t, x, v&) - \kappa_n(x)\varphi(t, x, v_{n, k})\rvert \notag \\
    &\leq \int_{\R^{L - 1}} \eta_{\epsilon}(\varpi(v) - w) \bigl\lvert \Phi_{n, k}(t, x, w) - \Phi_{n, k}(t, x, \varpi(v))\bigr\rvert \, \d w \notag \\
    &\leq \sup_{w \in B_{\epsilon}(\varpi(v))} \lvert \varphi_{n, k}(t, x, m_{w^+}) - \varphi(t, x, v_{n, k})\vert,
\end{align}
where $B_{\epsilon}(w)$ is the ball of radius $\epsilon$ centered at $w \in \R^{L - 1}$. For $w \in B_{\epsilon}(\varpi(v))$, we estimate
\begin{align} \label{eq:d_2_bound}
    d_2(m_{w^+}, v_{n, k}) &\leq c_n d_0(m_{w^+}, v_{n, k}) \notag \\
    &= c_n\sup_{\lVert \Phi\rVert_{\text{bLip} \leq 1}} \langle m_{w^+} - v_{n, k}, \Phi\rangle \notag \\
    &= c_n\sup_{\lVert \Phi\rVert_{\text{bLip} \leq 1}} \sum_{i = 1}^{L - 1} \int_{\R} ((w_i \lor 0) - \varpi_i(v)) \Phi(x) \bar{\psi}_i(x) \, \d x \notag \\
    &\leq c_n\sum_{i = 1}^{L - 1} \lvert w_i - \varpi_i(v)\rvert \notag \\
    &\leq c_n(L - 1)\epsilon.
\end{align}
For $t \in [0, T]$, let us define the modulus of continuity $\omega_n$ by
\begin{equation*}
    \omega_n(\delta) = \sup\biggl\{\sup_{(t, x) \in [0, T] \times [-n, n]}\lvert \varphi(t, x, v) - \varphi(t, x, v')\rvert \define v, v' \in \cal{K}_n,\, d_2(v, v') \leq \delta\biggr\}.
\end{equation*}
By assumption, this modulus is continuous at zero, i.e.\@ $\omega_n(\delta) \to 0$ as $\delta \to 0$. Inserting \eqref{eq:d_2_bound} into \eqref{eq:distance_to_moll} and applying $\omega_n$ yields
\begin{equation*}
    \bigl\lvert \varphi_{n, k}^{\epsilon}(t, x, v) - \kappa_n(x)\varphi(t, x, v_{n, k})\bigr\rvert \leq \omega_n\bigl(c_n (2n2^k - 1)\epsilon\bigr).
\end{equation*}
Now, we chose $\epsilon_{n, k} > 0$ such that $\omega_n\bigl(c_n (2n2^k - 1)\epsilon_{n, k}\bigr) \leq \frac{1}{n}\bigr\}$, whence $\omega_n\bigl(c_n (2n2^k - 1)\epsilon_{n, k}\bigr) = \frac{1}{n} \to 0$. This finishes our analysis of the first summand of \eqref{eq:decomp_approx}. 

We estimate the second term in \eqref{eq:decomp_approx} by
\begin{align*}
    \lvert \varphi(t, x, v_{n, k}) - \varphi(t, x, v_n)\rvert \leq \omega^n_t(d_2(v_{n, k}, v_n)) \leq \omega_n(c_nd_0(v_{n, k}, v_n)).
\end{align*}
We bound the $d_0$-metric appearing on the right-hand side by
\begin{align*}
    d_0(v_{n, k}, v_k) &= \sup_{\lVert \Phi\rVert_{\text{bLip} \leq 1}} \langle v_{n, k} - v_k, \Phi\rangle \\
    &= \sup_{\lVert \Phi\rVert_{\text{bLip} \leq 1}} \sum_{i = 1}^{L - 1} \biggl(\int_{\R} \varpi_i(v) \Phi(x) \bar{\psi}_i(x) \, \d x - \int_{\R} \Phi(x)\psi_i(x) \, \d v(x)\biggr) \\
    &\leq \sup_{\lVert \Phi\rVert_{\text{bLip} \leq 1}} \sum_{i = 1}^{L - 1} \biggl(\int_{\R} \varpi_i(v) \Phi(x_i) \bar{\psi}_i(x) \, \d x - \int_{\R} \Phi(x_i)\psi_i(x) \, \d v(x)\biggr) \\
    &\ \ \ + \sum_{i = 1}^{L - 1} \biggl(\int_{\R} \varpi_i(v) \lvert x_i - x\rvert \bar{\psi}_i(x) \, \d x + \int_{\R} \lvert x_i - x\rvert \psi_i(x) \, \d v(x)\biggr) \\
    &\leq 2^{-k} v(\R),
\end{align*}
where we note that the expression in the third line vanishes. In view of the previous display, we obtain $\lvert \varphi(t, x, v_{n, k}) - \varphi(t, x, v_n)\rvert \leq \omega_n(c_n 2^{-k})$. We set $k_n = \inf\{k \geq 1 \define \omega_n(c_n 2^{-k}) \leq \frac{1}{n}\}$.

Let us now study the last term in \eqref{eq:decomp_approx}. Fix compact subsets $K$ and $\cal{K}$ of $\R$ and $\cal{M}_{\leq 1}^2(\R)$, respectively and let $\omega_t$ be a modulus of continuity of $K \times \cal{K} \ni (x, v) \mapsto \varphi(t, x, v)$. We estimate
\begin{equation} \label{eq:distance_truncation}
    d_2^2(v_n, v) \leq \int_{\R} (1 + \lvert x \rvert^2) (1 - \kappa_n(x)) \, \d v(x),
\end{equation}
where we use the coupling $\pi \in \P^2(\R^2)$ of $\tilde{v} = v + (1 - v(\R)) \delta_0$ and $\tilde{v}_n = v_n + (1 - v_n(\R)) \delta_0$ defined by
\begin{align*}
    \int_{\R^2} h(x, y) \, \d \pi(x, y) &= \int_{\R} h(x, x) \, \d v_n(x) + (1 - v_n(\R))\int_{\R} h(x, 0) (1 - \kappa_n(x)) \, \d v(x) \\
    &\ \ \ + (1 - v(\R))(1 - v_n(\R)) h(0, 0)
\end{align*}
for $h \in C_b(\R^2)$ to estimate the $2$-Wasserstein distance between $\tilde{v}$ and $\tilde{v}_n$. The right-hand side of Equation \eqref{eq:distance_truncation} vanishes uniformly in $v \in \cal{K}$ as $n \to \infty$. Thus, we obtain that
\begin{align*}
    \lvert \kappa_n(x)\varphi(t, x, &v_n) - \kappa_n(x)\varphi(t, x, v)\rvert + \lvert \kappa_n(x)\varphi(t, x, v) - \varphi(t, x, v)\rvert \\
    &\leq \omega_t(d_2(v_n, v)) + C\bigl(1 + \lvert x\rvert^p + M_2^p(v)\bigr)\bf{1}_{x\notin [-(n - 1), n - 1]} \to 0
\end{align*}
uniformly over $(x, v) \in K \times \cal{K}$. Setting $\varphi_n(t, x, v) = \varphi_{n, k_n}^{\epsilon_n}(t, x, v)$ with $\epsilon_n = \epsilon_{n, k_n}$ and putting everything together yields $\bigl\lvert \varphi_n(t, x, v) - \varphi(t, x, v)\bigr\rvert \to 0$ uniformly over $(x, v)$ in compact subsets of $\R \times \cal{M}^2_{\leq 1}(\R)$ as $n \to \infty$.

For Property \ref{it:unif_quad_growth}, similarly to \eqref{eq:decomp_approx}, we write
\begin{align*}
    \lvert \varphi_n(t, x, v)\rvert &\leq \lvert \varphi_n(t, x, v) - \kappa_n(x)\varphi(t, x, v_n)\rvert + \lvert \kappa_n(x)\varphi(t, x, v_n)\rvert \\
    &\leq \bigl\lvert \varphi_n(t, x, v) - \kappa_n(x)\varphi(t, x, v_{n, k_n})\bigr\rvert + \bigl\lvert \varphi(t, x, v_{n, k_n}) - \varphi(t, x, v_n)\bigr\rvert \\
    &\ \ \ + C\bigl(1 + \lvert x \rvert^p + M_2^p(v_n)\bigr) \\
    &\leq \frac{2}{n} + C\bigl(1 + \lvert x \rvert^2 + M_2^p(v_n)\bigr) \\
    &\leq (C + 2)(1 + \lvert x \rvert^p + M_2^p(v_n)\bigr),
\end{align*}
so choosing $C' = C + 2$ gives the uniform growth bound. Moreover, we can further estimate the right-hand side above by $(C + 2)(1 + \lvert x \rvert^p + n^p\bigr)$. Then noting that $\lvert \varphi_n(t, x, v)\rvert$ vanishes outside of $[-n, n]$, immediately gives Property \ref{it:bounded_in_l2}.

It remains to establish Property \ref{it:func_der_bound}. Let us first prove that $v \mapsto \varphi_n(t, x, v)$ has a linear functional derivative. For $v$, $v' \in \cal{M}^2_{\leq 1}(\R)$, setting $\Phi_n(t, x, w) = \Phi_{n, k_n}^{\epsilon_n}(t, x, w)$, we compute
\begin{align*}
    \varphi_n(t, x, v) - \varphi_n(t, x, v') &= \sum_{i = 1}^{L - 1}\int_0^1 \partial_{w_i}\Phi_n\bigl(t, x, s \varpi(v) + (1 - s)\varpi(v')\bigr) (\varpi(v) - \varpi(v')) \, \d s \\
    &= \sum_{i = 1}^{L - 1} \int_0^1 \partial_{w_i}\Phi_n\bigl(t, x, w(s v + (1 - s)v')\bigr) \int_{\R} \psi_i(y) \, \d (v - v')(y) \, \d s \\
    &= \int_0^1 \bigl\langle v - v', D\varphi_n\bigl(t, x, s v + (1 - s)v'\bigr)\bigr\rangle \, \d s,
\end{align*}
where we use in the second equality that the function $v \mapsto \varpi(v)$ is linear and where we defined $D\varphi_n(t, x, v)(y) = \sum_{i = 1}^{K - 1} \partial_{w_i}\Phi_n(t, x, \varpi(v)) \psi_i(y)$ for $y \in \R$ and $v \in \cal{M}^2_{\leq 1}(\R)$. For $y \in \R$, we have that $\lvert D\varphi_n(t, x, v)(y)\rvert \leq \sum_{i = 1}^{L - 1} \lvert \partial_{w_i} \Phi_n(t, x, \varpi(v))\rvert\psi_i(y)$ and for $i = 1$,~\ldots, $K$ it holds with $\epsilon_n = \epsilon_{n, k_n}$ that
\begin{align*}
     \lvert \partial_{w_i}& \Phi_n(t, x, \varpi(v))\rvert \\
     &\leq \int_{\R^{L - 1}} \bigl\lvert \partial_{w_i}\eta_{\epsilon_n}(\varpi(v) - w)\bigr\rvert \kappa_n(x) \lvert \Phi_{n, k_n}(t, x, w)\rvert \, \d w \\
     &\leq \lVert \partial_{w_i}\eta_{\epsilon_n}\rVert_{\infty}\kappa_n(x) \sup_{w \in B_{\epsilon_n}(\varpi(v))} \lvert \Phi_{n, k_n}(t, x, w)\rvert \\
     &\leq \lVert \partial_{w_i}\eta_{\epsilon_n}\rVert_{\infty} \kappa_n(x) \biggl(\sup_{w \in B_{\epsilon_n}(\varpi(v))} \omega_n(d_2(m_{w^+}, v_{n, k})) + \omega_n(d_2(v_{n, k}, v_n)) + \lvert \varphi(t, x, v_n)\rvert\biggr) \\
     &\leq \lVert \partial_{w_i}\eta_{\epsilon_n}\rVert_{\infty} \bigl(\tfrac{2}{n} + \kappa_n(x) C(1 + \lvert x \rvert^p + n^p)\bigr).
\end{align*}
Hence, $\lvert D\varphi_n(t, x, v)(y)\rvert \leq \lVert \partial_{w_i}\eta_{\epsilon_n}\rVert_{\infty} \bigl(\tfrac{2}{n} + C(1 + 2n^p)\bigr)\sum_{i = 1}^{L - 1} \psi_i(y) = h_n(y)$. Clearly, the function $h_n$ has compact support and $\lVert h_n\rVert_{\infty} \leq C_n(1 + C)$ for an appropriately chosen $C_n > 0$. Lastly, the continuity of $\cal{M}^2_{\leq 1}(\R) \cap L^2(\R) \ni v \mapsto D\varphi(t, x, v)(y)$ with respect to the norm $\lVert \cdot \rVert_{L^2}$ is a simple consequence of the truncation and mollification.
\end{proof}

\subsection{Control of Conditional Processes} \label{app:cond_proc}

We briefly describe the setup of the control of conditional processes. On a filtered probability space $(\Omega, \F, \bb{F}, \pr)$ equipped with a $d$-dimensional $\bb{F}$-Brownian motion $W = (W_t)_{t \geq 0}$ and a $d$-dimensional $\F_0$-measurable random variable $\xi$, we consider the SDE
\begin{equation}
    \d X_t = \alpha_t \, \d t + \d W_t, \quad X_0 = \xi
\end{equation}
for a $d$-dimensional Brownian motion $W = (W_t)_{t \geq 0}$ and an $\bb{F}$-progressively measurable $\R^d$-valued process $\alpha = (\alpha_t)_{t \geq 0}$ such that $\ev \int_0^T \lvert \alpha_t\rvert^2 \, \d t < \infty$. We call $\mathfrak{a} = (\Omega, \F, \bb{F}, \pr, W, \alpha)$ a \textit{weak control} and denote the space of all such objects by $\bb{A}$. The \textit{cost functional} is given by
\begin{equation}
    J(\mathfrak{a}) = \int_0^T \ev[f(t, X_t, \alpha_t) \vert \tau > t] \, \d t + \ev[g(X_T) \vert \tau > T],
\end{equation}
where $\tau = \inf\{t > 0 \define X_t \notin D\}$ for an open subset $D \subset \R^d$ with $C^1$-boundary and $\bb{E}$ denotes the expectation with respect to $\pr$. We let the \textit{value} $V$ be the infimum of $J(\mathfrak{a})$ over weak controls $\mathfrak{a} \in \bb{A}$. We could consider an extended setup with a drift and diffusion coefficient that depend on time and the location of the diffusion or even its conditional law $\mu_t = \pr(X_t \in \cdot \vert \tau > t)$. The general case introduces additional difficulties, so for the purpose of illustration we stick to this simple setting. 

\begin{assumption} \label{ass:conditional}
Let $f \define [0, T] \times \R^d \times \R^d \to \R$ and $g \define \R^d \to \R$ be measurable and assume there exists $C > 0$ such that
\begin{enumerate}[noitemsep, label = (\roman*)]
    \item for all $(t, x, a)$ we have
    \begin{equation*}
        \lvert f(t, x, a)\rvert + \lvert g(x)\rvert \leq C(1 + \lvert x\rvert^2 + \lvert a\rvert^2);
    \end{equation*}
    \item \label{it:cont_cond} for all $t$ the maps $(x, a) \mapsto f(t, x, a)$ and $x \mapsto g(x)$ are continuous;
    \item \label{it:convex_cond} for all $(t, x)$ the map $a \mapsto f(t, x, a)$ is convex.
\end{enumerate}
\end{assumption}

Our first goal is to prove the equivalence between the weak formulation just introduced and the closed-loop formulation. A \textit{closed-loop control} is a measurable function $a \define [0, \infty) \times \R^d \to \R^d$ such that the SDE
\begin{equation} \label{eq:cond_weak}
    \d X_t = a(t, X_t) \, \d t + \d W_t, \quad X_0 = \xi
\end{equation}
has a weak solution on some filtered probability space $(\Omega, \F, \bb{F}, \pr)$ equipped with a $d$-dimensional $\bb{F}$-Brownian motion $W$ and an $\F_0$-measurable random variable $\xi$, for which $\ev \int_0^T \lvert a(t, X_t)\rvert^2 \, \d t < \infty$. Clearly, SDE \eqref{eq:cond_weak} exhibits uniqueness in law for any measurable function $a \define [0, T] \times \R^d \to \R^d$, so we can introduce the cost functional
\begin{equation}
    J_{\text{cl}}(a) = \int_0^T \ev\bigl[f\bigl(t, X_t, a(t, X_t)\bigr) \big\vert \tau > t\bigr] \, \d t + \ev[g(X_T) \vert \tau > T]
\end{equation}
as well as the associated value $V_{\text{cl}}$, which is simply the infimum of $J_{\text{cl}}(a)$ over all closed-loop controls. It is clear that any closed-loop control $a$ induces a weak control $\mathfrak{a}$ through the assignment $\alpha_t = a(t, X_t)$, so that $V \leq V_{\text{cl}}$. We obtain the reverse inequality through a mimicking argument.

\begin{proposition}
Let Assumption \ref{ass:conditional} be satisfied. For any $K > 0$ and any weak control $\mathfrak{a} = (\Omega, \F, \bb{F}, \pr, W, \alpha)$ with $\lvert \alpha_t\rvert \leq K$ for $\leb \otimes \pr$-a.e.\@ $(t, \omega) \in [0, T] \times \Omega$, there exists a closed-loop control $a$ with corresponding weak solution $(\tilde{\Omega}, \tilde{\F}, \tilde{\bb{F}}, \tilde{\pr}, \tilde{W}, \tilde{\xi}, \tilde{X})$ to SDE \eqref{eq:cond_weak} such that $\pr(X_t \in \cdot \vert \tau > t) = \tilde{\pr}(\tilde{X}_t \in \cdot \vert \tilde{\tau} > t)$ for all $t \in [0, T]$ and $J_{\textup{cl}}(a) \leq J(\mathfrak{a})$, where $\tilde{\tau} = \inf\{t > 0 \define \tilde{X}_t \notin D\}$. In particular, it holds that $V_{\textup{cl}} = V$.
\end{proposition}

\begin{proof}
Let $\mathfrak{a}$ be as in the statement of the theorem and define $X^{\tau}$ by $X^{\tau}_t = X_{\tau \land t}$. The process $X^{\tau}$ satisfies the SDE
\begin{equation*}
    \d X^{\tau}_t = \chi_D(X^{\tau}_t) \alpha_t \, \d t + \chi_D(X^{\tau}_t) \, \d W_t, \quad X^{\tau}_0 = \xi,
\end{equation*}
where $\chi_D \define \R^d \to \{0, 1\}$ is given by $\chi_D(x) = \bf{1}_{x \in D}$. Now, proceeding as in the proof of Theorem 3.7 in \cite{lacker_mfg_controlled_mgale_2015}, using the convexity of $f$ in the control argument guaranteed by Assumption \ref{ass:conditional} \ref{it:convex_cond}, we can find a measurable function $a \define [0, \infty) \times \R^d \to \R^d$ bounded by $K$ such that
\begin{align} \label{eq:meas_sel_cond}
\begin{split}
    \ev\bigl[\chi_D(X^{\tau}_t)\alpha_t \big\vert X^{\tau}_t\bigr] &= \chi_D(X^{\tau}_t)a(t, X^{\tau}_t), \\ \ev\bigl[\chi_D(X^{\tau}_t)f(t, X^{\tau}_t, \alpha_t) \big\vert X^{\tau}_t\bigr] &\geq \chi_D(X^{\tau}_t)f\bigl(t, X^{\tau}_t, a(t, X^{\tau}_t)\bigr)
\end{split}
\end{align}
for $t \in [0, T]$. We may assume that $a(t, x) = 0$ for all $x \in \R^d$ if $t > T$. Then the mimicking theorem by Brunick and Shreve \cite{brunick_mimicking_2013} yields a filtered probability space $(\tilde{\Omega}, \tilde{\F}, \tilde{\bb{F}}, \tilde{\bb{Q}})$ equipped with a $d$-dimensional $\tilde{\bb{F}}$-Brownian motion $\tilde{B}$ and a $d$-dimensional $\tilde{\bb{F}}$-adapted process $\tilde{Y}$ such that
\begin{equation*}
    \d \tilde{Y}_t = \chi_D(\tilde{Y}_t) a(t, \tilde{Y}_t) \, \d t + \chi_D(\tilde{Y}_t) \, \d \tilde{B}_t, \quad \tilde{Y}_0 \sim \L(\xi)
\end{equation*}
and $\L^{\tilde{\pr}}(\tilde{Y}_t) = \L(X^{\tau}_t)$. The latter implies together with \eqref{eq:meas_sel_cond} that
\begin{align} \label{eq:cond_cost_bound}
    J(\mathfrak{a}) &= \int_0^T \frac{\ev\bigl[\chi_D(X^{\tau}_t) f(t, X^{\tau}_t, \alpha_t)\bigr]}{\ev\chi_D(X^{\tau}_t)} \, \d t + \frac{\ev[\chi_D(X^{\tau}_T)g(X^{\tau}_T)]}{\ev\chi_D(X^{\tau}_T)} \notag \\
    &\geq \int_0^T \frac{\ev\bigl[\chi_D(X^{\tau}_t) f\bigl(t, X^{\tau}_t, a(t, X^{\tau}_t)\bigr)\bigr]}{\ev\chi_D(X^{\tau}_t)} \, \d t + \frac{\ev[\chi_D(X^{\tau}_T)g(X^{\tau}_T)]}{\ev\chi_D(X^{\tau}_T)} \notag \\
    &= \int_0^T \frac{\ev^{\tilde{\bb{Q}}}\bigl[\chi_D(\tilde{Y}_t) f\bigl(t, \tilde{Y}_t, a(t, \tilde{Y}_t)\bigr)\bigr]}{\ev^{\tilde{\bb{Q}}}\chi_D(\tilde{Y}_t)} \, \d t + \frac{\ev^{\tilde{\bb{Q}}}[\chi_D(\tilde{Y}_T)g(\tilde{Y}_T)]}{\ev^{\tilde{\bb{Q}}}\chi_D(\tilde{Y}_T)}.
\end{align}

We wish to extend $\tilde{Y}$ to a solution of SDE \eqref{eq:cond_weak}. For that, we perform the change of measure 
\begin{equation*}
    \frac{\d \tilde{\pr}}{\d \tilde{\bb{Q}}}\bigg\vert_{\tilde{\F}} = \cal{E}\biggl(\int_0^{\cdot} (1 - \chi_D(\tilde{Y}_t)) a(t, \tilde{X}_t) \cdot \d \tilde{B}_t\biggr)_{\infty},
\end{equation*}
where the process $\tilde{X} = (\tilde{X}_t)_{t \geq 0}$ is given by $\tilde{X}_t = \tilde{Y}_t$ if $t < \tilde{\tau}$ and $\tilde{X}_t = \tilde{Y}_{\tilde{\tau}} + \tilde{B}_t - \tilde{B}_{\tilde{\tau}}$ otherwise. By Girsanov's theorem, the process $\tilde{W} = (\tilde{W}_t)_{t \geq 0}$ defined by $\tilde{W}_t = \tilde{B}_t - \int_0^t (1 - \chi_D(\tilde{Y}_s)) a(s, \tilde{X}_s) \, \d s$ is a $d$-dimensional $\tilde{\bb{F}}$-Brownian motion under $\tilde{\pr}$. Moreover, we claim that $\tilde{X}$ is a weak solution to SDE \eqref{eq:cond_weak} on $(\tilde{\Omega}, \tilde{\F}, \tilde{\bb{F}}, \tilde{\pr})$ with Brownian motion $\tilde{W}$ and initial condition $\tilde{\xi}$. Indeed, $\tilde{X}$ clearly satisfies SDE \eqref{eq:cond_weak} on $[0, \tilde{\tau})$, since $\tilde{W} = \tilde{B}$ on $[0, \tilde{\tau})$ and $\tilde{Y}$ solves SDE \eqref{eq:cond_weak} on $[0, \tilde{\tau})$ with Brownian motion $\tilde{B}$. Next, assume that $t > \tilde{\tau}$. Then, it holds that
\begin{align*}
    \tilde{X}_t &= \tilde{Y}_{\tilde{\tau}} + \tilde{B}_t - \tilde{B}_{\tilde{\tau}} \\
    &= \tilde{Y}_0 + \int_0^{\tilde{\tau}} \chi_D(\tilde{Y}_s) a(s, \tilde{Y}_s) \, \d s + \tilde{B}_{\tilde{\tau}}  + \tilde{B}_t - \tilde{B}_{\tilde{\tau}} \\
    &= \tilde{Y}_0 + \int_0^{\tilde{\tau}} a(s, \tilde{X}_s) \, \d s + \int_0^t (1 - \chi_D(\tilde{Y}_s)) a(s, \tilde{X}_s) \, \d s + \tilde{W}_t \\
    &= \tilde{Y}_0 + \int_0^t a(s, \tilde{X}_s) \, \d s + \tilde{W}_t,
\end{align*}
so $\tilde{X}$ also solves SDE \eqref{eq:cond_weak} for $t > \tilde{\tau}$. Moreover, in view of Equation \eqref{eq:cond_cost_bound} we have that
\begin{align*}
    J(\mathfrak{a}) &\geq \int_0^T \frac{\ev^{\tilde{\pr}}\bigl[\chi_D(\tilde{X}_t) f\bigl(t, \tilde{X}_t, a(t, \tilde{X}_t)\bigr)\bigr]}{\ev^{\tilde{\pr}}\chi_D(\tilde{X}_t)} \, \d t + \frac{\ev^{\tilde{\pr}}[\chi_D(\tilde{X}_T)g(\tilde{X}_T)]}{\ev^{\tilde{\pr}}\chi_D(\tilde{X}_T)} = J_{\text{cl}}(a),
\end{align*}
where we used that $\tilde{\bb{Q}}$ and $\tilde{\pr}$ coincide on $\tilde{\F}_{\tilde{\tau} \land T}$. It remains to show that the conditional laws of $X_t$ and $\tilde{X}_t$ coincide. This follows from $\L^{\tilde{\pr}}(\tilde{Y}_t) = \L(X^{\tau}_t)$, since for $t \in [0, T]$ we have
\begin{align*}
    \pr(X_t \in A \vert \tau > t) &= \frac{\ev[\bf{1}_{\tau > t} \bf{1}_{X_t \in A}]}{\ev\bf{1}_{\tau > t}} = \frac{\ev[\chi_D(X^{\tau}_t )\bf{1}_{X^{\tau}_t \in A}]}{\ev\chi_D(X^{\tau}_t)} = \frac{\ev^{\tilde{\bb{Q}}}[\chi_D(\tilde{Y}_t )\bf{1}_{\tilde{Y}_t \in A}]}{\ev^{\tilde{\bb{Q}}}\chi_D(\tilde{Y}_t)} \\
    &= \frac{\ev^{\tilde{\pr}}[\chi_D(\tilde{X}^{\tilde{\tau}}_t )\bf{1}_{\tilde{X}^{\tilde{\tau}}_t \in A}]}{\ev^{\tilde{\pr}}\chi_D(\tilde{X}^{\tilde{\tau}}_t)} = \tilde{\pr}(\tilde{X}_t \in A \vert \tilde{\tau} > t),
\end{align*}
where $\tilde{\tau} = \inf\{t > 0 \define \tilde{Y} \notin D\} = \inf\{t > 0 \define \tilde{X} \notin D\}$.

The last statement follows from a simple limit procedure. Let us choose a sequence of weak controls $\mathfrak{a}^k = (\Omega^k, \F^k, \bb{F}^k, W^k, \alpha^k)$ such that $\lvert \alpha_t^k\rvert \leq k$ and $\lim_{k \to \infty} J(\mathfrak{a}^k) = V$. Such a sequence exists owing to the continuity condition from Assumption \ref{ass:conditional} \ref{it:cont_cond}. According to what we proved above, for each $k \geq 1$, there exists a closed-loop control $a^k$ such that $J_{\text{cl}}(a^k) = J(\mathfrak{a}^k)$. Consequently, we get
\begin{equation*}
    V_{\text{cl}} \leq \limsup_{k \to \infty} J_{\text{cl}}(a^k) = \limsup_{k \to \infty} J(\mathfrak{a}^k) = V,
\end{equation*}
which concludes the proof.
\end{proof}



\section*{Acknowledgement}
This research has been supported by the EPSRC Centre for Doctoral Training in Mathematics of Random Systems: Analysis, Modelling and Simulation (EP/S023925/1).

\printbibliography

\end{document}